\documentclass[reqno,12p]{amsart}
\usepackage[T1]{fontenc}
\usepackage{geometry}
\usepackage{mathtools,amssymb,amsthm,mathrsfs,color,lineno,paralist,graphicx,float}
\usepackage[colorlinks,
linkcolor=red,
anchorcolor=green,
citecolor=blue,
]{hyperref}
\usepackage{amsmath}
\usepackage{pdfpages}
\usepackage{amstext}
\usepackage{stmaryrd}
\usepackage{bbm}
\usepackage{comment}
\usepackage{faktor,url}
\usepackage{xfrac}
\usepackage{mathabx}
\usepackage{enumitem}
\graphicspath{{Dibujos/}}

\numberwithin{equation}{section}
\numberwithin{figure}{section}
\theoremstyle{plain}
\newtheorem{thm}
{\protect\theoremname}[section]
  \theoremstyle{definition}
  
  \theoremstyle{remark}
  \newtheorem{rem}{\protect\remarkname}
  \theoremstyle{plain}
  \newtheorem{lem}[thm]{\protect\lemmaname}
  \theoremstyle{plain}
  \newtheorem{prop}[thm]{\protect\propositionname}
  \theoremstyle{plain}

  \providecommand{\definitionname}{Definition}
  \providecommand{\lemmaname}{Lemma}
  \providecommand{\propositionname}{Proposition}
  \providecommand{\remarkname}{Remark}
\providecommand{\theoremname}{Theorem}
\providecommand{\corollaryname}{Corollary}

\title{A degenerate Arnold diffusion mechanism in the restricted 3-body problem}

\author[M. Guardia]{Marcel Guardia}
\address[MG]{Departament de Matem\`atiques i Inform\`atica, Universitat de Barcelona, Gran Via, 585, 08007 Barcelona, Spain  \& Centre de Recerca Matemàtica, Edifici C, Campus Bellaterra, 08193 Bellaterra, Spain}
\email{guardia@ub.edu}

\author[J. Paradela]{Jaime Paradela}
\address[JP]{ Departament de Matem\`atiques, Universitat Polit\`ecnica de Catalunya, Diagonal 647, 08028 Barcelona, Spain}
\email{jaime.paradela@upc.edu}

\author[T. M. Seara]{Tere M. Seara}
\address[TS]{Departament de Matem\`atiques, Universitat Polit\`ecnica de Catalunya, Diagonal 647, 08028 Barcelona, Spain \& Centre de Recerca Matemàtica, Edifici C, Campus Bellaterra, 08193 Bellaterra, Spain}
\email{tere.m-seara@upc.edu }

\begin{document}

\begin{abstract}
A major question in dynamical systems is to understand the mechanisms driving global instability in the 3-body problem (3BP), which models the motion of three  bodies under Newtonian gravitational interaction. The 3BP is called restricted if one of the bodies has zero mass and the other two, the primaries, have strictly positive masses $m_0,m_1$. We consider the  restricted planar elliptic 3-body problem (RPE3BP) where the primaries revolve in Keplerian ellipses. We prove that the RPE3BP exhibits topological instability: for any values of the masses $m_0,m_1$, except $m_0=m_1$, we build orbits along which the angular momentum  of the massless body  experiences an arbitrarily large variation provided the eccentricity of the orbit of the primaries is positive but small enough.

In order to prove this result we show that a degenerate Arnold diffusion mechanism, which moreover involves exponentially small phenomena, takes place in the RPE3BP. Our work extends the one  of Delshams, Kaloshin, de la Rosa, and Seara (2019) for the a priori unstable case $m_1/m_0\ll1$,  to the case of arbitrary masses  $m_0,m_1>0$, where the model displays features of the so-called \textit{a priori stable} setting.

\end{abstract}

\maketitle

\tableofcontents


\section{Introduction}

The $N$-body problem models the motion of $N$ bodies under mutual gravitational interaction. While the system is integrable for $N=2$, understanding its global dynamics for $N\geq 3$ is probably one of the oldest and most challenging questions in dynamical systems. A major achievement in this direction was the proof of the existence of a positive measure set of quasiperiodic motions in the $N$-body problem. This result was first established by Arnold in \cite{MR0170705}, who gave a master application of the KAM technique to the case of $3$ coplanar bodies. The proof was later extended to  case  $N\geq 3$  in the work of  F\'{e}joz and Herman \cite{MR2104595} (see also \cite{MR1364478,MR2836051}). On the other hand, Herman conjectured in his ICM address \cite{Herman1998} that the set of non-wandering points for the flow of the $N$-body problem is nowhere dense on every energy level for $N\geq 3$. This is in accordance with the general belief that the $N$-body problem, although strongly degenerate, displays the main features of a ``typical'' Hamiltonian system. This conjecture would imply topological instability for the $N$-body problem in a very strong sense.

The existence of topological instability in Hamiltonian systems was first investigated by Arnold in \cite{MR0163026},  where he constructed an example of a nearly integrable Hamiltonian in which this kind of behavior occurs. To that end, Arnold proposed a mechanism giving rise to unstable motions based on the existence of a transition chain of invariant tori: a sequence of invariant irrational tori which are connected by transverse heteroclinic orbits. This mechanism is nowadays called the \textit{Arnold mechanism}. Arnold verified that this mechanism takes place in a cleverly built model usually referred to as the Arnold model, and he  conjectured that topological instability is indeed a common phenomenon in the complement of integrable Hamiltonian systems \cite{MR0170705}.  Despite the enormous amount of research  (see for example  \cite{MR1986314, MR2153027, MR2184276, MR2393423, MR3646879, MR3653060, MR3649479, MR4298716, MR4160091} and the references therein), the Arnold diffusion phenomenon, and more generally the dynamics in the complement of the KAM tori set, is still poorly understood, and even more poorly for real-analytic or non-convex Hamiltonians.

In \cite{MR0163026}, Arnold conjectured that the mechanism of instability based on the existence of transition chains \textit{``is applicable to the general case (for example, to the problem of 3 bodies)''}. However, results concerning the existence of Arnold diffusion in the 3-body problem or related models are rather scarce (see  \cite{capinski2018arnold, MR3927089, https://doi.org/10.48550/arxiv.2210.11311,CFG} and also  \cite{MR3545968, MR3551192} for numerical-based results).

The 3-body problem is called ``restricted'' if one of the bodies has zero mass and the other two, the primaries, have strictly positive masses $m_0,m_1$. In this limit problem, the motion of the primaries is just a 2-body problem and the dynamics of the massless body is governed by the gravitational interaction with the primaries. In this work, we consider the case in which the primaries revolve around each other in Keplerian ellipses of eccentricity $\zeta\in (0,1)$  and the massless body moves on the same plane as the primaries. This model, usually known in the literature as the restricted planar elliptic 3-body problem (RPE3BP), is a $2+1/2$ degrees-of-freedom Hamiltonian system. For $\zeta=0$, i.e. for the restricted planar circular 3-body problem, the rotational symmetry prevents the existence of topological instability in nearly integrable settings, see \cite{MR3455155} and Remark \ref{rem:integrabilitycircular} below.

The goal of this paper is to prove that a \emph{degenerate Arnold diffusion mechanism} takes place in the RPE3BP:  we show that for any value of the masses   of the primaries ($m_0\neq m_1$), there exist orbits of the RPE3BP along which the angular momentum of the massless body experiences any predetermined drift provided the eccentricity of the orbits of the primaries is positive but small enough. Notice that the angular momentum is a conserved quantity in the 2-body problem, which can be seen as a limit problem of the restricted 3-body problem when $m_1/ m_0\to 0$.

To the best of our knowledge the first complete proof of existence of Arnold diffusion in celestial mechanics was obtained in \cite{MR3927089}, in which the authors showed the existence of topological instability in the RPE3BP. Nevertheless, this result was established under the strong hypothesis $m_1/m_0\ll1$ (see Section \ref{sec:Previousresults} for a more precise description of the setting). Under this condition, the problem falls in the so-called a priori unstable regime for the study of Arnold diffusion and can be analyzed by means of classical perturbation theory (see Section \ref{sec:aprioriunstablestable}). Our result extends the work in \cite{MR3927089}  to the case of \emph{arbitrary masses} $m_0,m_1>0$,  a setting in which the problem displays many features of the so-called \emph{a priori stable} case.


\subsection{Main result}\label{sec:Hamiltonianandeqsofmotion}

Fix a Cartesian reference system with origin at the center of mass of the primaries and choose units so that the total mass of the primaries is equal to $1$. In these coordinates, the primaries, which we denote by $q_0$ and $q_1$, move  along  Keplerian ellipses of eccentricity $\zeta\in (0,1)$ whose time parametrization reads
\[
q_0(t)=\mu \varrho(t)(\cos f(t),\sin f(t))\qquad\qquad q_1(t)=-(1-\mu)\varrho(t)(\cos f(t),\sin f(t)),
\] 
where $m_0=1-\mu$ and $m_1=\mu\in (0,1/2]$ are the masses of $q_0$ and $q_1$, $\varrho:\mathbb{T}\to \mathbb{R}$  ($\mathbb T=\mathbb R/ 2\pi \mathbb Z$) is the distance between the primaries 
and $f:\mathbb{T}\to\mathbb{T}$ is the so-called \textit{true anomaly} (see \cite{MR0005824}). The RPE3BP describes the motion  of a massless body $q\in \mathbb{R}^2$ in the gravitational field generated by the primaries and it is governed by the second order differential equations
\begin{equation}\label{eq: second order diff eq RP23BP}
\ddot{q}=(1-\mu)\frac{q-q_0(t)}{|q-q_0(t)|^3}+\mu \frac{q-q_1(t)}{|q-q_1(t)|^3}.
\end{equation}
It is a classical fact that the RPE3BP admits a Hamiltonian structure. Introducing $p$ and $E$ the conjugate momenta to $q$ and $t$, and the gravitational potential
\begin{equation}\label{eq:originalpotentiall}
U(q,t)=\frac{1-\mu}{|q-q_0(t)|}+\frac{\mu}{|q-q_1(t)|}
\end{equation}
the RPE3BP is a three degrees-of-freedom Hamiltonian system with respect to
\[
\mathcal{H}(q,p,t,E)=\frac{|p|^2}{2}-U(q,t)+E
\]
and the canonical symplectic structure in the extended phase space $T^* (\mathbb{R}^2\times \mathbb{T})$. The following is our main result.

\begin{thm}\label{thm:Maintheorem}
Let $G(q,p)=|q\wedge p|$ be the angular momentum of the massless body, let $\mu$ be the mass ratio and let $\zeta\in(0,1)$ be the eccentricity of the ellipse described by the primaries orbit. Then, for any $\mu\in (0,1/2)$,  there exists $G_*>0$ such that, for any $\zeta\in (0,G_*^{-3})$ and  any values $G_1,G_2$ satisfying
\[
G_*\leq G_1< G_2\leq \zeta^{-1/3},
\]
there exists $T>0$ and an orbit $\gamma$ of the RPE3BP for which 
\[
G\circ\gamma(0)\leq G_1\qquad\qquad\text{and}\qquad\qquad G_2\leq G\circ\gamma(T).
\]
\end{thm}

\subsection{Previous results: Arnold diffusion and unstable motions in celestial mechanics}\label{sec:Previousresults}
A  number of works have shown the existence of different types of unstable motions in the 3-body problem or its restricted versions. For example, oscillatory orbits (orbits that leave every bounded region but return infinitely often to some fixed bounded region, see \cite{MR1509241}) and/or chaotic behavior in particular configurations of the restricted 3-body problem \cite{sitnikov1960existence, llibre1980oscillatory,bolotin2006symbolic,MR2350333,MR3455155,MR1829194,guardia2021symbolic,baldomá2023coorbital,https://doi.org/10.48550/arxiv.2207.14351, https://doi.org/10.48550/arxiv.2212.05684}. 

However, results concerning the existence of Arnold diffusion in the 3-body problem or related models are rather scarce. Some remarkable works are \cite{MR3545968,MR3551192,capinski2018arnold,MR3927089,https://doi.org/10.48550/arxiv.2210.11311,CFG}.  In \cite{MR3545968} and \cite{MR3551192}, the authors combine numerical with analytical techniques to study the existence of diffusion orbits in the restricted  3-body problem close to $L_1$ and along mean motion resonances respectively.  In \cite{capinski2018arnold} the authors give a computer assisted proof of the existence of Arnold diffusion in the restricted planar elliptic 3-body problem. Moreover, some very interesting features of the random behavior, such as convergence to a stochastic process, are studied. In the recent works \cite{https://doi.org/10.48550/arxiv.2210.11311,CFG}, the authors show that the Arnold diffusion mechanism takes place in the spatial 4-body problem.

Of major importance, and closely related to the setting of the present work, is the paper \cite{MR3927089} (see also \cite{MR1223448,MR1284416} for previous partial results). To the best of our knowledge it constituted the first complete analytic proof of Arnold diffusion in celestial mechanics.

\begin{thm}[Theorem 1 in \cite{MR3927089}]\label{thm:TereTheorem}
There exist  $G_*>0$ and $c>0$ such that,  for any eccentricity $\zeta\in (0,cG_*^{-1})$ and  any values $G_1,G_2$ satisfying
\[
G_*\leq G_1<G_2\leq c/\zeta,
\]
if the mass ratio satisfies
\begin{equation}\label{cond:expsmall}
\mu\ll \exp(-G_2^3/3),
\end{equation}
 there exists $T>0$ and an orbit $\gamma:[0,T]\to \mathbb{R}_+\times\mathbb{T}^2\times\mathbb{R}^3$ of the RPE3BP for which 
\[
G\circ\gamma(0)\leq G_1\qquad\qquad G_2\leq G\circ\gamma(T).
\]
\end{thm}

\subsection{The  \cite{MR3927089}-approach: A priori unstable versus a priori stable setting}\label{sec:aprioriunstablestable}

The proof of Theorem \ref{thm:Maintheorem} follows the same geometric approach as Theorem  \ref{thm:TereTheorem}. However,  the fact that in Theorem \ref{thm:Maintheorem} the parameter $\mu$ is not taken small changes radically the proof. 

We devote this section to explain the main steps of the approach in  \cite{MR3927089} and compare the regimes considered in the two mentioned theorems. Compactifying the phase space (see \cite{MR362403}), one can detect an invariant cylinder at infinity foliated by periodic orbits. It corresponds to the limit of orbits with asymptotic 0 radial velocity; that is, the limit of the so-called parabolic orbits. This cylinder is the key object  to construct the unstable orbits. The  main steps in this construction of the drifting orbits are the following:

\begin{enumerate}
    \item Prove that the cylinder at infinity  has stable and unstable invariant manifolds. Note that this is not obvious because the periodic orbits in the cylinder are degenerate (Floquet exponents are equal to 0).
    \item Prove that these invariant manifolds intersect transversally along two distinct homoclinic channels. The orbits in these channels are asymptotic in the past and in the future to (possibly different) periodic orbits in the cylinder.
    \item Construct a \textit{transition chain of periodic orbits}. That is, a sequence of periodic orbits connected by transverse heteroclinic orbits.
    \item Obtain an orbit which shadows, i.e. follows closely, the transition chain.
\end{enumerate}

Let us compare this construction with the classical diffusion mechanism introduced by Arnold in \cite{MR0163026} for the $2+1/2$ degrees-of-freedom Hamiltonian 
\[
H(q,p,\theta,I,t)=\frac{I^2}{2}+\frac{p^2}{2}+\varepsilon(\cos q-1)(1+\mu(\sin \theta+\cos t)).
\]
This model can be seen as a toy model describing the dynamics of a convex Hamiltonian on a neighbourhood of a simple resonance. Arnold's classical mechanism relies on the invariant cylinder $N=\{q=p=0\}$ which is foliated by partially hyperbolic invariant tori  among which irrational tori, i.e. non-resonant, are dense. For $\mu>0$ but sufficiently small with respect to $\varepsilon$, in fact $0<\mu <e^{-c/\varepsilon}$, $c>0$, Arnold shows that the unstable manifold of any torus intersects the stable manifold of any other torus which is sufficiently close. In this way he constructs a transition chain of irrational tori and then a true orbit shadowing  this chain. 

The degeneracies of the restricted 3-body problem make all the tori in the cylinder at infinity resonant. In particular, this cylinder is foliated by periodic orbits. Therefore, the approach above can be seen as a \textit{degenerate Arnold mechanism. }

\begin{rem}
    In what concerns the implementation of the Arnold mechanism along the aforementioned cylinder at infinity, the restricted 3-body problem actually presents two different kinds of  degeneracies. First, the fact that this cylinder is not normally hyperbolic (although as we will see later it possesses stable and unstable manifolds).  Second, the fact that the cylinder is foliated by periodic orbits instead of a family of invariant tori with a continuum of frequencies.
\end{rem}

It is well known  in Arnold diffusion that, even if a mechanism is expected to take place in a model, its proof may differ radically depending on the parameter range considered. Indeed, Arnold's original model has two parameters $\varepsilon$ and $\mu$. Unfolding  $\varepsilon$ breaks the Liouville-Arnold integrability  keeping the system integrable, while unfolding $\mu$ breaks integrability. Arnold's proof provides drift in action provided that the parameter which breaks integrability is exponentially small with respect to the one which only breaks Liouville-Arnold integrability, that is, $0<\mu <e^{-c/\varepsilon}$, $c>0$. Note that, proving the existence Arnold diffusion in the very same model when both parameters are small but independent \emph{is still an open problem}. In fact, it is also open if they satisfy a power-like relation $\mu = \mathcal{O}(\varepsilon^n$), for $n \ge 0$. 

Following \cite{CherchiaG94} these two regimes  are often referred to as: 
\begin{itemize}
\item A priori unstable when  the parameter breaking integrability is taken exponentially small with respect to the other one. The  most well known example is Arnold's model with $0<\mu <e^{-c/\varepsilon}$.
\item  A priori stable when both  parameters are taken small but of the same order.
\end{itemize}

One should remark that, when studying the existence of diffusion along a simple resonance of a nearly integrable Hamiltonian system, one obtains a normal form which generalizes Arnold's model but in which the aforementioned parameters satisfy, in general, a power-like relation. 
Therefore understanding the a priori stable regime is of major importance.

\begin{rem}
Although the study of the a priori stable regime has seen important advances (in a cusp-residual sense, see \cite{MR4298716}) in the smooth category (and  especially for 3 degrees-of-freedom Hamiltonians), none of these results have been extended to the case of real-analytic a priori stable Hamiltonians.  Indeed, to the best of our knowledge, all known examples of real-analytic Hamiltonians exhibiting Arnold diffusion concern two-parameter unfoldings of an integrable system and diffusion is obtained when the parameter breaking integrability is exponentially small with respect to the parameter which, while keeping the system integrable, only breaks Arnold-Liouville integrability. Therefore, according to the above classification (which follows \cite{CherchiaG94}) these systems do not fall into the a priori stable setting.

Moreover, we want to stress that there are no results which guarantee that a given one-parameter family of Hamiltonian systems undergoes Arnold diffusion, neither  in the smooth category nor in the analytic one. This shows the enormous difficulty of proving the existence of this phenomenon in  concrete models. In fact, the first example of a real-analytic a priori stable system exhibiting topological instability was recently constructed by B. Fayad in \cite{MR4495839}. The techniques are however different from the Arnold diffusion mechanism. 
\end{rem}

In the model considered in the present paper, the mass ratio $\mu$ is the parameter breaking integrability (see \cite{llibre1980oscillatory}). In order to work in a close to integrable regime (even when $\mu$ is not small) we study the \textit{singular perturbation framework} corresponding to the region of the phase space where the angular momentum $G$ is large.  Although it does not make sense to consider $G=\infty$, one can introduce a proper scaling (c.f. Section \ref{sec:scaling}) on the region $G\geq G_*$ such that, letting $G_*\to\infty$ can be related to the break up of Liouville-Arnold integrability of the associated normal form (see Remark \ref{rem:scalingremark}).

The fundamental difference between Theorems \ref{thm:Maintheorem} and \ref{thm:TereTheorem} is the parameter range considered. The paper \cite{MR3927089} deals with an a priori unstable setting (see \eqref{cond:expsmall} in Theorem \ref{thm:TereTheorem}) whereas the present paper deals with an a priori stable regime. Note that the degeneracies of the restricted 3-body problem make the a priori stable regime very particular. This is explained in Section \ref{sec:degenerateHamiltonians} where we refer to it as a \emph{strongly degenerate a priori stable model.}

Let us explain here how the a priori stable setting makes  the approach detailed above hard to implement to prove Theorem \ref{thm:Maintheorem}.

Step (1), the existence of the stable/unstable invariant manifolds of the periodic orbits at infinity is a consequence of \cite{MR2030148} and applies to both range of parameters.
Step (4) is a consequence of a lambda Lemma (often also called inclination Lemma) proven in \cite{https://doi.org/10.48550/arxiv.2207.14351} and also applies in both parameter ranges. 

On the contrary the proof of Steps (2) and (3) is radically different. This is where this paper presents the most important novelties, which are explained in Section \ref{sec:proofmainthm}. Indeed, the main bulk of the present paper is to analyze the stable and unstable invariant manifolds of the cylinder and the geometry of their intersections. The reason is that the transversality between these manifolds is exponentially small with respect to the angular momentum $G$. Namely, it is a beyond all orders phenomenon and, therefore, any perturbative technique which relies on expanding the parameterization of the invariant manifolds in inverse powers of the angular momentum $1/G$ fails.

In the a priori unstable regime, when the mass ratio $\mu$ is taken small, one can perform power expansions in this parameter. Then, assuming that it is exponentially small with respect to the threshold of angular momentum considered, one can apply classical perturbative arguments (that is a suitable version of Melnikov theory, see \cite{MR3846872,MR3927089}), to study the intersections between the invariant manifolds of the cylinder.

On the contrary, in the a priori stable regime, the mass ratio $\mu$ is just a fixed parameter and therefore the mentioned approach cannot be applied. This falls into what is usually called \emph{exponentially small splitting of invariant manifolds}.

The analysis of exponentially small splitting  has drawn major attention in the past decades due to its relevance for the study of instability mechanisms in real-analytic Hamiltonian systems. Remarkable progress has been made in a number of works in low dimensional models (just to cite a few works, see \cite{MR923969,MR1204314,Gelfreich94,Gelfreich97a,Treshev97,MR1669417,MR2025418,MartinSS11,MR2981260,MR3007729}). In higher dimension, results are much more scarce (see \cite{MR1478530,MR1841877,MR1964346} where the exponentially small splitting between the stable and unstable manifolds of a partially hyperbolic  invariant torus is investigated). However, the tools developed in these works concern only the existence homoclinic orbits to a given invariant torus and do not suffice to establish the existence of (large) transition chains.

To the best of our knowledge, the present work is the first in which  a (large) transition chain is built in the presence of exponentially small phenomena without any assumption on the relation between the parameters involved.

\subsection{Strongly degenerate nearly integrable Hamiltonian systems}\label{sec:degenerateHamiltonians}

As already explained in the previous section, the Arnold diffusion regime that leads to the drift in angular momentum in Theorem \ref{thm:Maintheorem} is a priori stable and strongly degenerate in certain aspects. 

Let us be more specific. To simplify the exposition we focus on 2 and a half degrees-of-freedom Hamiltonians. Consider a  Hamiltonian system which is close to a Liouville-Arnold integrable Hamiltonian. In action-angle coordinates it is of the form
\[
H(\theta,I)=H_0(I)+\varepsilon H_1(\theta,I,t),\qquad \theta\in\mathbb{T}^2, \, I\in U\subset\mathbb{R}^2,\, t\in\mathbb{T},\qquad 0 <\varepsilon\ll 1.
\]
Then, depending on the properties of $H_0$ and $H_1$ one can classify the model as follows:
\begin{enumerate}
    \item \emph{Non-degenerate} a priori stable Hamiltonian systems
    \item  \emph{Properly degenerate} a priori stable Hamiltonian systems
    \item \emph{Strongly degenerate} a priori stable Hamiltonian systems.
\end{enumerate}

\emph{Non-degenerate} models refer to those with $H_0$ presenting some non-degeneracy, for instance steepness (one could consider more strict non-degeneracy conditions such as convex $H_0$). Non-degenerate Hamiltonians are those Arnold's conjecture refers too (see \cite{MR3646879,MR4298716, MR3653060,GideaM22} for results in the finite regularity class and convexity assumptions on $H_0$).  To understand Arnold diffusion for such Hamiltonians, one must deal with two regimes: simple resonances (regions in action space where there exist one resonant relation among frequencies $\omega(I)=(\nabla H_0(I),1)$) and double resonances (two resonant relations). Indeed, the diffusing orbits ``travel along'' simple resonances  and cross infinitely many double resonances. Those double resonances can be classified as strong when the two resonant relations have ``comparable strength'' and weak when one is much stronger than the other. 

We call \emph{properly degenerate}  those Hamiltonians such that $H_0$ does not depend on one of  the actions. This is typical in celestial mechanics models due to the degeneracies of the 2-body problem (see, for instance \cite{MR3551192,https://doi.org/10.48550/arxiv.2210.11311,CFG}). In this case there is a trivial resonance relation. If one fixes another one, one can look for solutions drifting along it.

Finally, the third class is what we refer to as \emph{strongly degenerate} Hamiltonian systems. They are such that $H_0$ presents the same degeneracy as in the properly degenerate case and that $H_1$ also presents a degeneracy in such a way that the normally hyperbolic cylinder that appears at simple resonance has very particular resonant dynamics: all orbits are periodic. Namely, there exists a simple resonance foliated by double resonances. This is the setting of the present paper. 

Indeed,  in the first class, one expects to have a ``network'' of normally hyperbolic cylinders which cross weak double resonances but break up at the strong ones. In the second class, one expects that there exists a ``large'' cylinder whose dynamics is close to integrable and has multiple time scales. The geometric features of the third class are similar to those of the second one but with the extra degeneracy that the cylinder dynamics is fully resonant.

The tools developed in the present paper  to prove Theorem \ref{thm:Maintheorem} apply to ``typical'' strongly degenerate Hamiltonian systems. An example of such type is  a degenerate version of the Arnold model (see \cite{MR0163026}) 
\begin{equation}\label{eq:degenerateArnold}
H(\varphi_1,I_1,\varphi_2,I_2,t;\varepsilon,\mu)=H_0(I_1)+ \mu H_1 (\varphi_1,\varphi_2,I_2,t;\varepsilon),\quad\qquad (\varphi_1,I_1)\in \mathbb{T}\times\mathbb{R},\ (\varphi_2,I_2)\in \mathbb{T}\times\mathbb{R},\ t\in\mathbb{T}
\end{equation}
where
\[
H_0(I_1)=\frac{I_1^2}{2},\qquad\qquad H_1 (\varphi_1,\varphi_2,I_2,t;\varepsilon)=\varepsilon (\cos \varphi_1-1)\left(1+\frac{I_2^2}{2}+\sin \varphi_2+\cos t\right).
\]
Observe that, as for the classical Arnold model, for any $\mu,\varepsilon\geq 0$,
\[
N=\{(\varphi_1,I_1,\varphi_2,I,t)\in \mathbb{T}\times\mathbb{R}\times\mathbb{T}\times\mathbb{R}\times\mathbb{T}\colon \varphi_1=I_1=0\}
\]
is a normally hyperbolic invariant cylinder. However, for this degenerate model, it is foliated by fully resonant 
tori with frequencies
\[
\omega(I)=(0,1)
\]
and therefore, the tori are foliated by periodic orbits.

\subsection{Main tools for the proof of Theorem \ref{thm:Maintheorem}}

In this work we introduce a set of techniques to analyze the Arnold diffusion phenomenon in\emph{ strongly degenerate} a priori stable Hamiltonians (according to the classification introduced in the previous section). As already mentioned, the main player in the geometric mechanism for diffusion in this setting is a normally hyperbolic invariant cylinder which is foliated by fully resonant tori. The main ideas are:
\begin{itemize}
    \item By exploiting the symplectic features of the problem, we develop a formalism for the analysis of the splitting between the invariant manifolds of pairs of different fully resonant partially hyperbolic tori. As we explain in Section \ref{sec: The intersection problem} the splitting between these manifolds is heavily anisotropic.
    \item Our analysis on the existence of intersections between the invariant manifolds of pairs of resonant tori allows us to later recover information about the global geometry of the intersection between the invariant manifolds of the cylinder. Rather remarkably, we show that, even if their splitting angle is exponentially small, the invariant manifolds of the cylinder intersect transversally along two different homoclinic  manifolds, both of which contain an annulus (a bounded cylinder). Namely, the homoclinic manifolds, also called homoclinic channels, are global.
    \item  We obtain a detailed asymptotic analysis of the so-called scattering maps (associated to both homoclinic channels). In particular, we are able to obtain asymptotic formulas for the difference between scattering maps associated to different homoclinic channels.
\end{itemize}

These ideas will be explained in full detail in Section \ref{sec:proofmainthm}. For the sake of clarity the techniques are tailored to the problem at hand, but we have tried to make the exposition conceptual enough so the core ideas can be clearly extrapolated.

\begin{rem}
    The second item is of relevance in \emph{strongly degenerate} systems since, as we will see in Section \ref{sec:proofmainthm}, implies the existence of two scattering maps defined globally on an annulus. This is crucial to overcome the fact that the inner dynamics on the cylinder is fully resonant (the identity map for the time-one map).

    We expect that understanding the global geometry of the intersection between the invariant manifolds of the cylinder is considerably more involved in less degenerate models. At the level of the Arnold diffusion mechanism, this might be however less relevant as in these systems one can also rely on the inner dynamics to move along the cylinder.
\end{rem}

We believe that the techniques developed in this paper are of high relevance to the study of Arnold diffusion also in properly degenerate a priori stable systems. The reason is that, for these systems, the KAM tori lying in the normally hyperbolic manifold associated to a simple resonance are ``almost resonant'' in the sense that the dynamics in one of the angles is much slower than for the other one. Therefore, we expect that the splitting between the invariant manifolds  of these tori can be analyzed by the techniques of the present paper. As already mentioned, these systems appear naturally in celestial mechanics. One interesting model to consider is the 3-body problem near mean motion resonances.

The general case, namely non-degenerate systems, seems at the moment out of reach. Our techniques could be used to obtain a first order approximation of the splitting potential (see Section \ref{sec:nonexactlagrangian}), associated to a pair of nonresonant (KAM) invariant tori,  in terms of an explicit fast oscillatory integral. However, the so-called problem of small exponents (see \cite{MR1841877,MR1478530}) makes it difficult to obtain a lower bound for this integral and, hence, for the splitting angles.

One important result commonly used in  the proofs of existence of Arnold diffusion is Moeckel's theorem \ref{thm:originalmoeckel}, which gives sufficient conditions for an iterated systems of two maps to have drifting orbits. Roughly speaking, the main condition in Moeckel's theorem is that the two maps do not share any invariant curve. But to check this condition in a given example can be extremely difficult. In this paper, we provide a new quantitative Theorem  \ref{thm:scattmapmain} which gives explicit checkable conditions ensuring that the requirements of Moeckel's theorem \ref{thm:originalmoeckel} are satisfied and, therefore, the two maps do not share any invariant curve. We think this result, which is independent of the rest of the paper and can be applied to a general setting, can be an important tool to show the existence of  Arnold diffusion in other models.

\subsection{Organization of the article}

 Section \ref{sec:proofmainthm} contains the core of the proof of Theorem \ref{thm:Maintheorem}. More concretely, Section \ref{sec: The intersection problem} renders the main ideas behind the proof of the first main ingredient: identifying a (topological) normally hyperbolic invariant manifold ``at infinity'' $\mathcal{P}_\infty$, and analyzing  the geometry of the intersection between its stable and unstable invariant manifolds. The proofs of the results in this section are postponed to Section \ref{sec:generatingfunctsinvmanifolds}. 
 
 Section \ref{sec:proofthm24} is devoted to the construction of two global \textit{scattering maps} on $\mathcal{P}_\infty$ and their asymptotic analysis. In particular, we check that they satisfy the conditions of Theorem \ref{thm:scattmapmain} and consequently we prove the non-existence of common invariant curves for the scattering maps. The rather technical proofs of the results in these sections are deferred to Sections \ref{sec:proofgeneratingscatt} and  \ref{sec:asymptformscattmap}.
 Section \ref{sec:abstractresultscattmaps}, which is independent of the rest of the paper,  deals with the proof of Theorem \ref{thm:scattmapmain}.

 Appendix \ref{sec:perturbregime}  contains some relevant information on the 2-body problem and Appendix \ref{sec:Melnikov} contains a detailed study of the perturbative potential and the associated Melnikov potential.

Througout the rest of the paper we fix a value  $\mu\in (0,1/2)$.

\subsection*{Acknowledgements}

The authors want to thank the referees for valuable suggestions and help to substantially improve  the final version of the
manuscript. This project has received funding from the European Research Council (ERC) under the European Union's Horizon 2020 research and innovation programme (grant agreement No 757802).  This work is part of the grant PID-2021-122954NB-100 funded by MCIN/AEI/10.13039/501100011033 and ``ERDF A way of making Europe''.   M.G. and T.M.S are supported by the Catalan Institution for Research and Advanced Studies via an ICREA Academia Prize 2019. This work is also supported by the Spanish State Research Agency, through the Severo Ochoa and Mar\'{i}a de Maeztu Program for Centers and Units of Excellence in R\&D (CEX2020-001084-M).



\section{Proof of the main theorem}\label{sec:proofmainthm}

We introduce the (exact symplectic)  change to polar coordinates $(r,y,\alpha,G,t,E)\mapsto(q,p,t,E)$ where $q=(r\cos\alpha,r\sin\alpha)$ and  $(y,G)$ are the conjugate momenta to $(r,\alpha)$. In this coordinate system, the RPE3BP is a Hamiltonian system on the (extended) phase space (rigourously one should exclude collisions, but, since our analysis is performed far from collisions, we abuse notation and we refer to $M_{\mathrm{pol}}$ as phase space) 
\begin{equation}\label{eq:phasespacepolar}
(r,\alpha,t,y,G,E)\in\mathbb{R}_+\times\mathbb{T}^2\times\mathbb{R}^3\equiv M_{\mathrm{pol}}
\end{equation}
with Hamiltonian function
\begin{equation}\label{eq: Hamiltonian in polar coordinates}
H_{\mathrm{pol}}(r,\alpha,t,y,G,E)=\frac{y^2}{2}+\frac{G^2}{2r^2}-V_{\mathrm{pol}}(r,\alpha,t)+E,\qquad\qquad  V_{\mathrm{pol}}(r,\alpha,t)=U(r\cos\alpha,r\sin\alpha,t),
\end{equation}
where $U$ is the potential in \eqref{eq:originalpotentiall}. Due to the fact that  $V_{\mathrm{pol}}(r,\alpha,t)\to0$ as $r\to\infty$, 
 there exists an invariant manifold ``at infinity''. To describe this manifold properly,  we make use of  McGehee's compactification of the phase space $r=2/x^2$. Then, it is an easy computation to check that in McGehee's coordinates $(x,\alpha,t,y,G,E)$ 
\begin{equation}\label{eq: the infinity manifold}
\mathcal{P}_\infty=\left\{(0,\alpha,t,0,G,0)\colon (\alpha,t)\in\mathbb{T}^2, G\in\mathbb{R}\right\}
\end{equation}
is an invariant submanifold contained in the zero energy level $\{H_{\mathrm{pol}}=0\}$. One may also check that $\mathcal P_\infty$ is foliated by periodic orbits with zero Floquet exponents. Therefore, the dynamics around $\mathcal P_\infty$, falls into the \emph{strongly degenerate} case according to the classification introduced in Section \ref{sec:degenerateHamiltonians} (with the extra degeneracy of having zero exponents which, as we will see, can be dealt with by means of classical results and will not introduce extra difficulties). As $\mathcal P_\infty$ is the union of periodic orbits, the following lemma is a direct consequence of Theorem 2.1 of \cite{MR3583476} (see also \cite{MR362403,MR2030148}).


\begin{lem}\label{lem:lemacilindro}
    Let $\mathcal P_\infty$ be the 3-dimensional invariant manifold introduced in \eqref{eq: the infinity manifold}. Then, the stable and unstable sets \begin{equation}\label{eq:invmanifoldspolar}
\begin{split}
W^{\mathrm{u}}(\mathcal{P}_\infty)=&\{ z\in \{H_{\mathrm{pol}}=0\} \colon \exists z'\in \mathcal{P}_\infty\ \text{ for which }\ \lim_{\tau\to -\infty}|\phi^\tau(z)-\phi^\tau(z')|=0\}\\
W^{\mathrm{s}}(\mathcal{P}_\infty)=&\{ z\in  \{H_{\mathrm{pol}}=0\} \colon \exists z'\in \mathcal{P}_\infty\ \text{ for which }\ \lim_{\tau\to\infty}|\phi^\tau(z)-\phi^\tau(z')|=0\}.\\
\end{split}
\end{equation}
    are 4-dimensional  immersed submanifolds, which are $C^\infty$ everywhere an real-analytic on the complement of $\{x=0\}$.
\end{lem}

It is also possible to show that a suitable version of the Lambda lemma holds near $\mathcal P_\infty$ (see \cite{MR3583476}). Together, these results show that the flow on a neighborhood of the invariant manifold $\mathcal{P}_\infty$ behaves, at the topological level, as the flow around a normally hyperbolic invariant cylinder. 

The invariant submanifold $\mathcal{P}_\infty$ together with its stable and unstable manifolds constitute the main players involved in the geometric mechanism, outlined in Section \ref{sec:aprioriunstablestable}, leading to the main result in Theorem \ref{thm:Maintheorem}. Theorems \ref{thm:intersectionmain}, \ref{thm:scattmapmain} and \ref{thm:Moeckel} below, together with a well known result by Moeckel \cite{MR1884898} (Theorem \ref{thm:originalmoeckel} below), constitute all the steps needed to prove the existence of this geometric mechanism. Although this mechanism was already implemented in \cite{MR3927089} for $\mu\ll 1$, extending this construction for fixed $\mu\in (0,1/2)$ is a major challenge.

In Theorem \ref{thm:intersectionmain}, we study the geometry of the intersection between the stable and unstable manifolds of the invariant submanifold $\mathcal{P}_\infty$. To that end, it is convenient to introduce the Poincar\'{e} map (recall that the Hamiltonian \eqref{eq: Hamiltonian in polar coordinates} is $2\pi$-periodic in $t$)
\begin{equation}\label{eq:poincaremap}
P:\{t=0\}\to \{t=2\pi\}
\end{equation}
induced by the flow of the Hamiltonian \eqref{eq: Hamiltonian in polar coordinates} on the section $\{t=0\}$ and denote by
\[
\mathcal{P}_\infty^*=\mathcal{P}_\infty\cap\{t=0\}\cong \mathbb{T}\times\mathbb{R}.
\]
It is important to remark that $\mathcal P^*_\infty$ is foliated by fixed points of the map $P$.
\begin{thm}\label{thm:intersectionmain}
Fix any $\mu\in(0,1/2)$, any  $G_*\gg 1$ and any $R>0$. Then, for sufficiently small eccentricity $0\leq\zeta<(G_*+R)^{-3}$, there exist, at least, two different real-analytic  submanifolds $\Gamma_\pm\subset W^{\mathrm u}(\mathcal{P}_\infty^*)\pitchfork W^{\mathrm s}(\mathcal{P}_\infty^*)$, which are diffeomorphic to the compact annulus $\mathbb{T}\times[0,1]$ and whose $G$-projection  covers the interval $[G_*,G_*+R]$.
\end{thm}
The proof of Theorem \ref{thm:intersectionmain} is outlined in Section \ref{sec: The intersection problem}. It is the most challenging step in our construction and constitutes the foremost source of novelties of the present work. The main difficulty is that, in order to work in a nearly integrable regime (without assuming $\mu$ small) we consider the singular perturbation framework obtained by restricting to the region of the phase space in which the angular momentum is large, namely $G\geq G_*$ for some $G_*\gg1$. Thus, we  analyze the portion of the invariant manifolds $W^{\mathrm{u,s}}(\mathcal{P}_\infty^*)$ which lies on this region.  We will show in Section \ref{sec:scaling} that, in this region of the phase space, the Hamiltonian is $\mathcal{O}(\mu G_*^{-4})$ (the precise power in $G_*$ is not relevant for the ongoing discussion) close to integrable. However, there exist  different time scales, whose ratio is proportional to $G_*^3$. This results in the fact that the angle (see \cite{MR802878}) satisfies
\[
\angle \left(T_{\Gamma_\pm}W^{\mathrm{u}}(\mathcal P^*_\infty),T_{\Gamma_\pm}W^{\mathrm{s}}(\mathcal P^*_\infty)\right)\leq \mu\exp(-cG_*^3),
\]
for some $c>0$. This makes rather difficult to check that the splitting of these invariant manifolds is actually non-zero and even more challenging, to say something about the geometry of the intersection. Indeed, any classical perturbation technique using  $1/G_*$ as parameter fails to detect the splitting between these manifolds.  Theorem \ref{thm:intersectionmain} is, nevertheless, much easier to establish if one also considers $\mu\ll1$ so one can make use of classical perturbation theory in this parameter. This was the approach used in \cite{MR3927089}, where Theorem \ref{thm:intersectionmain} was proved under the assumption $\mu\ll \exp(-cG^3_*)$, reminiscent of Arnold's trick in \cite{MR0163026}. In the present work, where $\mu\in(0,1/2)$ is fixed, we cannot rely on classical perturbation theory but instead use rather delicate singular perturbation theory techniques and exploit the symplectic features of the problem, in particular, the Hamilton-Jacobi formalism for Lagrangian submanifolds. While our proof is, for the sake of clarity, adapted to the problem at hand, we believe that our ideas are of interest in the more general framework of properly degenerate a priori stable nearly-integrable real-analytic Hamiltonian systems (see Section \ref{sec:degenerateHamiltonians}).

The homoclinic manifolds $\Gamma_\pm$   obtained in Theorem \ref{thm:intersectionmain} are composed of  heteroclinic points connecting (possibly different) points in $\mathcal{P}_\infty^*$. Although the $G$-projection of the collection of heteroclinic points is large, the jumps in the $G$ coordinate along each heteroclinic orbit are small (bounded by an inverse power of $G_*$). In order to complete the proof of Theorem \ref{thm:Maintheorem}, we look for a (possibly very large) chain of points in $\mathcal P_\infty^*$ connected by heteroclinic orbits along which the value of $G$ increases. 

To formalize the study of the outer dynamics associated to these heteroclinic connections, we make use of the so-called \textit{scattering maps} (see \cite{MR1737988,MR2184276,MR2383896}), which encode the dynamics along the heteroclinic orbits passing through $\Gamma_\pm$. In the present setting, they can be introduced  as follows. Denote by $\phi^\tau_{H_\mathrm{pol}}$ the time $\tau$ flow associated to the Hamiltonian \eqref{eq: Hamiltonian in polar coordinates}. Then, associated to each homoclinic channel $\Gamma_{\pm}$, we can define the \textit{backward wave map} 
\begin{equation}\label{eq:restrwavemaps}
\begin{aligned}
\Omega_\pm^{\mathrm{u}}:\Gamma_\pm &\to \mathcal{P}_\infty^*\\
z &\mapsto (\alpha_\pm^{\mathrm{u}},G_\pm^{\mathrm{u}})=\lim_{\tau\to -\infty} (\alpha\circ\phi^\tau_{H_\mathrm{pol}}(z), G \circ\phi^\tau_{H_\mathrm{pol}}(z))
\end{aligned}
\end{equation}
and the \textit{forward wave map}
\begin{equation}\label{eq:restrwavemapsford}
\begin{aligned}
\Omega_\pm^{\mathrm{s}}:\Gamma_\pm &\to \mathcal{P}_\infty^*\\
z &\mapsto (\alpha_\pm^{\mathrm{s}},G_\pm^{\mathrm{s}})=\lim_{\tau\to + \infty} (\alpha\circ\phi^\tau_{H_\mathrm{pol}}(z), G \circ\phi^\tau_{H_\mathrm{pol}}(z)),
\end{aligned}
\end{equation}
which are diffeomorphisms onto their images. Notice that $\alpha$ and $G$ are constants of motion in $\mathcal{P}_\infty$ and therefore, these limits are well defined. Finally,  the  \textit{scattering maps}, are  given by
\begin{equation}\label{eq:defnscatteringmaps}
\begin{array}{rcl}
\mathbb{P}_\pm =\Omega_{\pm}^{\mathrm{s}}\circ\left(\Omega^{\mathrm{u}}_\pm\right)^{-1}: \mathbb{A}_\pm &\longrightarrow &\mathcal{P}_\infty^*\\
(\alpha^{\mathrm{u}},G^{\mathrm{u}})&\longmapsto& (\alpha_\pm^{\mathrm{s}},G_\pm^{\mathrm{s}}).\
\end{array}
\end{equation}
where $\mathbb{A}_\pm=\Omega_\pm^{\mathrm{u}}(\Gamma_\pm)$ is their domain of definition. 

As it will be shown in Theorem \ref{thm:Moeckel}, both scattering maps share a common domain of definition $\mathbb A\subset\mathbb A_+\cap \mathbb A_-$ which is diffeomorphic to an annulus. Then, a suitable version of the Lambda Lemma (see Proposition \ref{prop:shadowing} below) reduces the study of the existence of a drifting  orbit, to the study of the existence of drifting orbits for the iterated function system defined by $\{\mathbb{P}_+,\mathbb{P}_-\}$ on $\mathbb{A}$. To obtain the latter we will rely on the next result by Moeckel   (see also \cite{MR2322183} and \cite{GideaM22} for more general statements), which gives a necessary and sufficient condition for the existence of drifting orbits of an iterated function system of maps of the annulus.

\begin{thm}[Theorem 1 in \cite{MR1884898}]\label{thm:originalmoeckel}
Let $\phi$ be an iterated function system constructed from two $\mathcal C^r$ twist maps $g_+,g_-$ of the annulus $A=\mathbb{T}\times [a,b]$ which fix the boundary of $A$. Then, $\phi$ does not drift from $a'>a$ to $b'<b$ if and only if $g_+$ and $g_-$ have a common essential invariant curve on $\mathbb{T}\times[a',b']\subset A$.
\end{thm}

\begin{rem}
    The assumption that the map fixes the boundary  is superfluous if one is only interested in orbits which drift from $a+2\delta$ to $b-2\delta$ (for a fixed positive $\delta>0$). Indeed one could modify the map on $A\setminus (\mathbb T\times([a,a+\delta]\cup[b-\delta,b]))$ so that it fixes the boundary without affecting the dynamics on the region $\mathbb{T}\times [a+\delta,b-\delta]$.
\end{rem}

It is in general very hard to check that two given maps do not share common essential invariant curves. In the problem at hand, this condition is actually not difficult to verify if one assumes that $\mu\ll 1$ is sufficiently small. Indeed, it is enough to consider a first order Taylor expansion in $\mu$ of the scattering maps $\mathbb{P}_\pm$ (see \cite{MR3927089}). However, checking the non-existence of common invariant curves is much more challenging for fixed $\mu\in(0,1/2)$, since the Taylor expansions in $1/G$ (notice that there is no other perturbative parameter) of the maps $\mathbb{P}_\pm$ coincide up to any order.

In Theorem \ref{thm:scattmapmain} below,  we give sufficient conditions for a pair of close to identity twist maps of the annulus to do not share common essential invariant curves. The key point of this result is that allows us to treat cases when the Taylor expansions of both maps in the perturbative parameter coincide.

Before giving the precise statement of Theorem \ref{thm:scattmapmain}, let us give a heuristic idea and identify the main quantities involved. We consider two  twist area preserving maps $g_\pm$ (see equation \eqref{eq:firstcondmapmoeckel}), which are $\varepsilon$-close to the identity  and $\delta(\varepsilon)$-close to integrable twist maps with frequency $\varepsilon\omega$ (and we assume $\delta(\varepsilon)\ll\varepsilon$). 

The fact that these 2-dimensional twist maps are $\varepsilon$-close to the identity allows us to deduce that, for any of them, there exist  one degree-of-freedom Hamiltonians $H_\pm$ such that their time-one maps are exponentially close 
(in fact $\mathcal{O}(\delta(\varepsilon)\exp(-c/\varepsilon))$- close) to the maps $g_\pm$.  

Using this we can provide estimates for the following to quantities: 
\begin{itemize}
\item An upper bound for the size of  the gaps between the KAM curves of $g_\pm$.  
%
\item A lower bound for the angle between the foliations given by the level sets of the Hamiltonians $H_+$ and $H_-$. 
\end{itemize}
Then, if we show that as $\varepsilon\to 0$,
\[
\frac{\text{size of the gaps between KAM curves}}{\text{angle between leaves of the foliations}}\to 0,
\]
there cannot exist common invariant curves for the maps $g_\pm$. This condition is exactly  condition \eqref{eq:suffcondmoeckel} below, which  only requires exponentially small transversality between the foliations of $H_\pm$.

In the statement of Theorem \ref{thm:scattmapmain} we make use of the following notation. Given a complex $\rho$-neighborhood $A_\rho$ of an annulus $A=\mathbb{T}\times[a,b]$ and $f:A_\rho\to\mathbb{C}$,  we write 
\[
f=\mathcal{O}_\rho (\varepsilon)
\]
if there exists $C>0$ independent of $\varepsilon$ such that, as $\varepsilon \to 0$ we have $|f|\leq C \varepsilon$ uniformly on $A_\rho$. We say that 
\[
f=o_\rho(\varepsilon)
\]
if 
\[
\lim_{\varepsilon\to 0} \frac{|f|}{\varepsilon}=0
\]
uniformly on $A_\rho$.
\begin{thm}\label{thm:scattmapmain}
Let $\rho>0$ and let $g_+,g_-$ be two real-analytic exact symplectic twist maps which depend also analytically on a parameter $\varepsilon$. Assume that there exists $\rho>0$ such that for any $\varepsilon$ sufficiently small the maps $g_\pm$ are defined in a $\rho$-neighborhood in $(\mathbb{C}/2\pi\mathbb{Z})\times\mathbb{C}$ of the annulus $A=\mathbb{T}\times[a,b]$ and  are of the form
\begin{equation}\label{eq:firstcondmapmoeckel}
g_\pm:(\alpha,G)\mapsto (\alpha+\varepsilon\  \omega(G;\varepsilon),\ G+\delta(\varepsilon)r(\alpha,G))+o_\rho(\delta(\varepsilon)),
\end{equation}
 for 
 \[
 \inf_{G\in[a,b]}|\omega(G;0)|\neq 0,\qquad\qquad r=\mathcal{O}_\rho(1),
 \]
 and $\delta(\varepsilon)>0$ satisfying $\lim_{\varepsilon\to 0} \delta (\varepsilon)/\varepsilon=0$. Let 
 \[
 f(\alpha,G)=(-\delta(\varepsilon)r(\alpha,G),\ \varepsilon\omega(G;\varepsilon))
 \]
 and assume that there exists $\eta(\varepsilon)>0$ and a vertical strip $\mathbb{I}=(\alpha_1,\alpha_2)\times[a,b]$ on which ($\langle\cdot,\cdot\rangle$ denotes the standard scalar product on $\mathbb{R}^2$)
\begin{equation}\label{eq:secondcondmapmoeckel}
\inf_{(\alpha,G)\in\mathbb{I}} |\langle f,\ g_+-g_-\rangle (\alpha,G)|\geq \eta(\varepsilon)>0
\end{equation}
and, for all $(\alpha,G)\in \mathbb I$,
\begin{equation}\label{eq:suffcondmoeckel2}
\lim_{\varepsilon\to 0} \frac{ \delta (\varepsilon) |(g_+-g_-)(\alpha,G)| }{\eta(\varepsilon)}\to 0.
\end{equation}
Then, if there exists $c>0$ independent of $\varepsilon$ such that  
\begin{equation}\label{eq:suffcondmoeckel}
\lim_{\varepsilon\to 0} \ \frac{\sqrt{\delta(\varepsilon)\exp(-c/\varepsilon)}}{
\varepsilon \inf_{G\in[a,b]} |\omega'(G;\varepsilon)|\  \eta(\varepsilon)}\to 0,
\end{equation}
the maps $g_+$ and $g_-$ share no common essential invariant curve on $A_{\varepsilon}=\mathbb{T}\times[a+\varepsilon^{1/4}, b-\varepsilon^{1/4}]$.
\end{thm}

   The proof of Theorem \ref{thm:scattmapmain} is carried out in Section \ref{sec:abstractresultscattmaps}. Recall that the main ingredient in the proof is the construction of two Hamiltonians $H_\pm$ whose time-one maps approximate the maps $g_\pm$ up to exponentially small remainder. Provided \eqref{eq:suffcondmoeckel2} is satisfied, we show that  $\langle f,g_+-g_-\rangle$ gives an asymptotic formula for the angle between the foliations $H_\pm=\text{const}$ bounded by bellow by $\eta (\varepsilon)$. If this angle is much bigger than the gaps between the KAM tori of the maps, whose size is upper bounded 
   \[
   \frac{\sqrt{\delta(\varepsilon)\exp(-c/\varepsilon)}}{\inf_{G\in[a,b]}|\varepsilon |\omega'(G,\varepsilon)|}
   \]
   then $g_\pm$ do not share common essential invariant curves. It is in this way that we obtain  condition \eqref{eq:suffcondmoeckel}.

 The main ideas involved in the construction of the Hamiltonians $H_\pm$  are the combination of an interpolation argument for close to integrable real-analytic exact symplectic maps due to Kuksin and P\"{o}schel \cite{MR1279392}  with a standard averaging technique for one fast-frequency systems.
 Our following result shows that the scattering maps $\mathbb{P}_\pm$ satisfy the hypotheses of Theorem \ref{thm:scattmapmain}. As a corollary of Theorems \ref{thm:originalmoeckel} and \ref{thm:scattmapmain}, we obtain a drifting orbit of the iterated function system of the scattering maps $\mathbb{P}_+$ and $\mathbb{P}_-$.

\begin{rem}
    The perturbative regime in which we analyze the restricted 3-body problem corresponds to looking at the region of the phase space in which the angular momentum $G$ is sufficiently large.  

    In Theorem \ref{thm:Moeckel} below we look at a compact but large piece of the cylinder $\mathcal P_\infty^*$ whose $G$-projection covers an interval $[G_*,G_*+R]$ with $G_*\gg 1$ and any $R>0$. Since $R$ might be much larger than $G_*$, in order to make use of Theorem \ref{thm:scattmapmain}, which was formulated in terms of a perturbative parameter $\varepsilon$, we split the interval $[G_*,G_*+R]$ into several intervals of the form $[G_0,G_0+1]\subset [G_*,G_*+R]$ and use  $\varepsilon=G_0^{-4}$ as perturbative parameter.
\end{rem}
 
\begin{thm}\label{thm:Moeckel}
Fix any $\mu\in(0,1/2)$ any $G_*\gg 1$ and any $R>0$.  Then, for  $0\leq \zeta\leq (G_*+R)^{-3}$ the scattering maps $\mathbb{P}_\pm$ in \eqref{eq:defnscatteringmaps} are well defined, real-analytic and exact symplectic on an annulus $\mathbb{A}\subset\mathbb{A}_+\cap \mathbb{A}_-\subset\mathcal{P}_\infty^*$ whose $G$-projection covers the interval $[G_*,G_*+R]$.

Moreover, there exists $\rho>0$ such that, for any $G_0\in [G_*,G_*+R]$, on the complex $\rho$-neighborhood of each horizontal strip $\mathbb{T}\times(G_0,G_0+1)\subset \mathbb{A}$,  the scattering maps $\mathbb{P}_\pm$ are twist maps of the form \eqref{eq:firstcondmapmoeckel} with $a=G_0, b=G_0+1$
\begin{equation}\label{eq:firstmoeckelcondscatt}
\varepsilon=G_0^{-4},\qquad\qquad \delta(\varepsilon)=\zeta\  \varepsilon^{5/4},\qquad\qquad |\omega(G,\varepsilon)|\gtrsim 1\qquad\qquad|\omega'(G;\varepsilon)|\gtrsim\varepsilon^{1/4},
\end{equation}
and satisfy \eqref{eq:suffcondmoeckel} and \eqref{eq:suffcondmoeckel2} on $\mathbb{I}=(\pi/8,\pi/4)\times(G_0,G_0+1)$ with $\eta(\varepsilon)$ in \eqref{eq:secondcondmapmoeckel} given by  
\begin{equation}\label{eq:secondmoeckelcondscatt}
\eta(\varepsilon(G_0))=\zeta\exp(-G_0^3/3).
\end{equation}
Therefore, for any $0<\zeta\leq (G_*+R)^{-3}$, there exists $N\in\mathbb{N}$ and a sequence
\[
\{ (i_k,z_k)\}_{1\leq k\leq N}\subset \left( \{+,-\} \times \mathcal{P}_\infty^* \right)^N\qquad\qquad z_{k+1}=\mathbb{P}_{i_k}(z_k)
\]
such that 
\[
\pi_G(z_N)-\pi_G(z_1)\geq R/2.
\]
\end{thm}

Note that in many settings the scattering maps are not globally defined and, in particular, they may have monodromy when the angle $\alpha$ makes a full turn. This is not the case in the present paper as stated in Theorem \ref{thm:Moeckel}.

Theorem \ref{thm:Moeckel} is proved in Section \ref{sec:proofthm24}. The most difficult part is to establish a lower bound for the difference between both scattering maps. Indeed,  they are associated to a homoclinic intersection with exponentially small transversality, so their difference is exponentially small. Key to establish a lower bound (actually, an asymptotic expression) are the geometric ideas developed in the proof of Theorem \ref{thm:intersectionmain}, which allow us to construct generating functions for the scattering maps whose asymptotics, as well as their difference, are well approximated by certain explicit integrals, the so-called  Melnikov potential (see \eqref{eq: DefinitionredMelnikovPotential}).

\begin{rem}\label{rem:transversalitygaps}
Notice that the lower bound \eqref{eq:secondmoeckelcondscatt} on the angle  $\gtrsim \exp(-G_0^3/3)$ between  the foliations associated to the scattering maps (see the discussion below Theorem \ref{thm:scattmapmain}) is of much larger size than the gaps$\lesssim \exp(-cG_0^{-4})$ between the KAM curves of each of the scattering maps and therefore condition \eqref{eq:suffcondmoeckel} is satisfied. The different exponents are obtained as follows: First, the exponent on \eqref{eq:secondmoeckelcondscatt} comes from the ratio between time scales which is of order $G_0^{-3}$. On the other hand, the exponent in the upper estimate for the KAM gaps is related to the distance from the scattering maps to the identity, which is of order $G_0^{-4}$.
\end{rem}

Finally, the proof of Theorem \ref{thm:Maintheorem}  is completed by standard shadowing results (see Figure  \ref{fig:shadowing}). Let $(\alpha,G)\in \mathcal{P}_\infty^*$ be a parabolic fixed point of the Poincar\'{e} map $P$ in \eqref{eq:poincaremap} and denote by $W_{\alpha,G}^{\mathrm{u},\mathrm{s}}$ its stable and unstable manifolds.  For a number $\delta>0$  and a point $p\in\{t=0\}$  denote by $B_\delta(p)\subset \{t=0\}$ the ball of radius $\delta$ centered at $p$ in the Poincar\'{e} section. The following  shadowing result for parabolic fixed points, proved in \cite{MR3583476} fits our purposes.

\begin{prop}[Proposition 2 in \cite{MR3583476}] \label{prop:shadowing}
Let $N\in\mathbb{N}\cup \{\infty\}$ and let $\{(\alpha_k,G_k)\}_{1\leq k\leq N}$ be a family of  fixed points  in $\mathcal{P}_\infty^*$  for the Poincar\'{e} map $P$ such that, for all $1\leq k\leq N$,   $W^{\mathrm{u}}_{\alpha_k,G_k}$ intersects transversally $W^{\mathrm{s}}(\mathcal{P}_\infty^*)$ at a point $p_k\in W^{\mathrm{s}}_{\alpha_{k+1},G_{k+1}}$. Then, for any sequence $\{\delta_k\}_{k\geq 1}$ with $\delta_k>0$, there exists a point $z\in B_{\delta_1}(\alpha_0,G_0)$ and two sequences  $\{n_k\}_{1\leq k\leq N}, \{\tilde{n}_k\}_{1\leq k\leq N}, \subset \mathbb{N}$ with $n_{k}< \tilde{n}_k<n_{k+1}< \tilde{n}_{k+1}$ such that $P^{n_k}(z)\in B_{\delta_k}( \alpha_k,G_k)$ and $P^{\tilde{n}_k}(z)\in B_{\delta_k}( p_k)$ for all $1\leq k\leq N$.

\end{prop}\vspace{0.4cm}

Let $\{z_k\}_{1\leq k\leq N}=\{(\alpha_k,G_k)\}_{1\leq k\leq N} \subset \mathcal{P}_\infty^*$ be the sequence of fixed points for the Poincar\'{e} map $P$ given in Theorem \ref{thm:Moeckel} and apply Proposition \ref{prop:shadowing} with $\delta_k>0$ small enough. The proof of Theorem \ref{thm:Maintheorem}  is complete.

\begin{figure}
\centering
\includegraphics[scale=0.5]{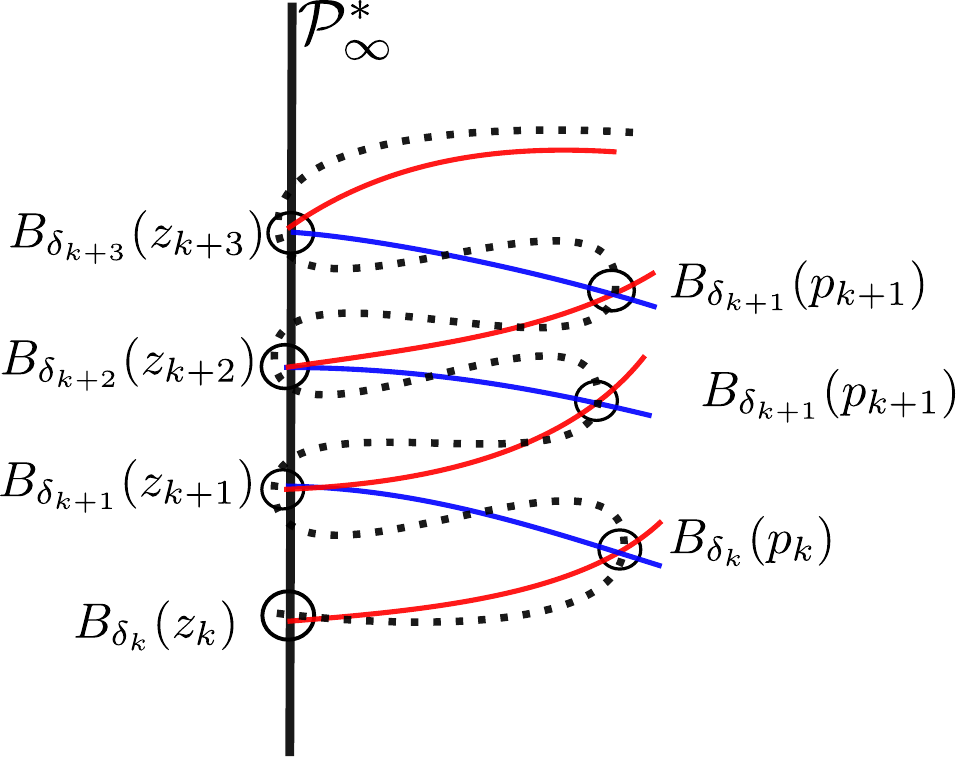}
\caption{A true orbit (dotted curve) of the RPE3BP which shadows the pseudo orbit $\{z_k\}_{1\leq k\leq N}$ obtained in Theorem \ref{thm:Moeckel}. The pseudo orbit is contained in a union of segments of stable and unstable manifolds manifolds associated to parabolic fixed points $\{z_k\}_{0\leq k\leq N}\subset \mathcal P_\infty^*$ of the Poincar\'{e} map $P$. }\label{fig:shadowing}
\end{figure}

\vspace{1cm}


\subsection{Proof of Theorem \ref{thm:intersectionmain}}\label{sec: The intersection problem}
The proof of Theorem \ref{thm:intersectionmain} concentrates most of the novelties of this work. Before entering into the details of our construction, we believe it is worth to explain conceptually why all the previous techniques available fail to  give a proof of Theorem \ref{thm:intersectionmain}. For that matter, consider a real-analytic Hamiltonian system of the form 
\begin{equation}\label{eq:modelhamiltonian}
H(q,p,\theta,I,t,E)=H_0(q,p,I)+\varepsilon^{-1/2} E+\mu H_1(q,p,\theta,I,t),
\end{equation}
where $(q,p)$, $(\theta,I)$ and $(t,E)$ are pairs of conjugated variables. 

\begin{rem}
We will see in Section \ref{sec:scaling} that, close to $\mathcal{P}_\infty$, the Hamiltonian of the RPE3BP is of the form \eqref{eq:modelhamiltonian} after a suitable scaling. Also \eqref{eq:modelhamiltonian} appears naturally as a (scaled) normal form for properly degenerate a priori stable Hamiltonians (see Section \ref{sec:degenerateHamiltonians}) close to a (simple) resonance.
\end{rem}
Suppose that 
\[
N=\{q=p=0\}
\]
is a  normally hyperbolic invariant manifold for $H_0+\varepsilon^{-1/2} E$ whose stable and unstable manifolds coincide along a homoclinic manifold $W^{\mathrm{h}}(N)$ and that, for some $q=q_*$ the section $\{q=q_*\}$ intersects transversally the manifold $W^{\mathrm{h}}(N)$. Moreover, assume that $\nabla H_1$ vanishes on $N$ and that the 3-dimensional submanifold
\begin{equation}\label{eq:Ncalig}
\mathcal N=N\cap\{E=0\}
\end{equation}
is contained in a single energy level of the Hamiltonian $H_0+\mu H_1$.

\begin{rem}
    The last assumption is highly non generic but happens in the strongly degenerate setting introduced in Section \ref{sec:degenerateHamiltonians}. Some explicit examples of this phenomenon are given by the invariant manifold for the degenerate Arnold model in Section \ref{sec:degenerateHamiltonians} or the invariant manifold $\mathcal{P}_\infty$ for the RPE3BP.
\end{rem}

A natural approach to study the existence of transverse intersections between the perturbed invariant manifolds $W_\mu^{\mathrm{u,s}} (N)$ is to consider the section $\Sigma=\{z+\lambda \nabla  H_0(z)\colon z\in W^{\mathrm{h}}(N)\cap\{q=q_*\},\ \lambda\in\mathbb{R}\}$ and look for graph parametrizations of  $W^{\mathrm{u,s}}_\mu(N)\cap \Sigma$ (in coordinates $(z,\lambda)$ on $\Sigma$) of the form $\Gamma^{\mathrm{u,s}}(z)=(z,\lambda^{\mathrm{u,s}}(z))$. Then,  $
\hat{d}(z)=\lambda^{\mathrm{u}}(z)-\lambda^{\mathrm{s}}(z)$, which is proportional to $d(z)= H_0(\Gamma^{\mathrm{u}}(z))- H_0(\Gamma^{\mathrm{u}}(z))
$, measures the distance between these manifolds. Classical first order perturbation theory, yields then the existence of an explicit function $M_\varepsilon(z)$ such that 
\begin{equation}\label{eq:poincaremelnikovexp}
d(z)=\mu M_\varepsilon(z)+\mathcal{O}(\mu^2).
\end{equation}
However, $M_\varepsilon$, given by a fast oscillatory integral (notice that $\dot t=1/\sqrt\varepsilon$), will have in general size $M_\varepsilon \sim \exp(-c_*/\sqrt\varepsilon)$ (for some $c_*>0$), and the expansion \eqref{eq:poincaremelnikovexp} becomes meaningless unless $\mu\ll \exp(-c_*/\sqrt\varepsilon)$. 

Taking advantage of the fast oscillation of the variable $t$, one can, nevertheless, obtain more accurate estimates for the $\mathcal{O}(\mu^2)$ terms in \eqref{eq:poincaremelnikovexp}. Indeed, standard averaging arguments (see Neishtadt \cite{MR802878}) show that there exists a $\mu\sqrt{\varepsilon}$-close to identity transformation $z=\Psi(\hat z)$, defined in a complex neighborhood of $W^{\mathrm{h}}(N)$, coinciding with the identity on $N$ (this is a consequence of the fact that $\nabla H_1=0$ on $N$), for which, for some $c\leq c_*$,
\[
H\circ\Psi(\hat q,\hat p,\hat \theta,\hat I, \hat t,\hat E)=K(\hat q,\hat p,\hat \theta,\hat I)+\varepsilon^{-1/2}\hat E+\mathcal{O}\left(\mu \sqrt{\varepsilon}\exp(-c/\sqrt{\varepsilon})\right).
\]
A simple counting dimension argument now shows that, for the flow of the (autonomous) Hamiltonian $K$, the stable and unstable manifolds of $N$  coincide along a homoclinic manifold which we denote by $W^{\mathrm{h}}_K(N)=W^{\mathrm{u,s}}_K(N)$. Indeed, the manifold $\mathcal{N}\cap\{t=0\}$ (recall the definition of $\mathcal{N}$ in \eqref{eq:Ncalig}), is a 2-dimensional invariant cylinder which is contained in a single energy level of the autonomous Hamiltonian $K$. Thus, their 3-dimensional stable and unstable manifolds cannot intersect transversally in this 3-dimensional energy level.

If we now build the section $\Sigma_K=\{\hat z+\hat\lambda \nabla K(\hat z)\colon \hat z\in W^{\mathrm{h}}_K(N)\cap\{\hat q=q_*\},\ \hat\lambda\in\mathbb{R}\}$ and look for graph parametrizations of  $W^{\mathrm{u,s}}_\mu(N)\cap \Sigma_K$ (in coordinates $(\hat z,\hat\lambda)$ on $\Sigma_K$) of the form $\hat \Gamma^{\mathrm{u,s}}(\hat z)=(\hat z,\hat\lambda^{\mathrm{u,s}}(\hat z))$, we observe that the distance between these manifolds, given by $\hat\lambda^{\mathrm{u}}(\hat z)-\hat\lambda^{\mathrm{s}}(\hat z)$ is bounded above by
\[
\widehat{d}(\hat z)=|K(\widehat\Gamma^{\mathrm{u}}(\hat z))-K(\widehat\Gamma^{\mathrm{u}}(\hat z))|\leq \mu\sqrt{\varepsilon}\exp(-c/\sqrt{\varepsilon}).
\]
Hence,
\begin{equation}\label{eq:poincaremelnikovexp2}
\widehat{d}(\hat z)=\mu M_\varepsilon(z)+\mathcal{O}\big(\mu\sqrt\varepsilon\exp(-c/\sqrt{\varepsilon})\big).
\end{equation}
Still, \eqref{eq:poincaremelnikovexp2} does not constitute an asymptotic expression unless we are able to show that $c=c_*$, a task that seems rather complicated if one does not perform a deeper analysis of the perturbed invariant manifolds and incorporates into the argument more geometric features of the problem.

Coming back to Arnold's original idea, instead of trying to study the invariant manifolds of the normally hyperbolic invariant manifold, we will \textit{focus} our attention on the \textit{invariant tori} inside it and analyze their stable and unstable manifolds. These invariant manifolds have codimension two inside each energy surface so, a priori, it seems much more complicated to  investigate the existence of transverse intersections between them. The key point behind this shift of approach is the fact that the invariant manifolds of the invariant tori are Lagrangian and, therefore, one can exploit more symplectic features of the problem to analyze them.

A very conceptual geometric framework to study the intersections between the invariant manifolds of these partially hyperbolic invariant tori has been developed by Lochak, Marco and Sauzin in \cite{MR1841877}  and \cite{MR1964346} (in a  more general setting than that of \eqref{eq:modelhamiltonian}). Their ideas rely on the Hamilton-Jacobi formalism for Lagrangian submanifolds: namely, the stable and unstable manifolds of these invariant tori can be parametrized as graphs of the differential of certain scalar functions $S^{\mathrm{u,s}}$ over the configuration space. Given a homoclinic orbit $\gamma$ (which can, in general, be obtained by variational methods) to a given torus,  the authors derive upper bounds for the splitting angles between the tangent spaces of the stable and unstable manifolds at $\gamma$, in terms of the Hessian matrix of the difference $\Delta S=S^{\mathrm{u}}-S^{\mathrm{s}}$.

From this connection, one of the main observations in \cite{MR1964346} is that the upper bounds for the splitting angles strongly depend on the arithmetic properties of the frequency vectors of the invariant torus under consideration. 
Roughly speaking, the upper bounds of both angles are  exponentially small for irrational tori but, 
for resonant tori,  there are exponentially small splitting angles (associated to the  fast dynamics) and polynomially small splitting angles (associated to the slow resonant dynamics).

However, notice that, given a homoclinic orbit $\gamma$ to a torus,  to deduce the existence of heteroclinic orbits to sufficiently close partially hyperbolic invariant  tori, one would need a lower bound on the splitting angles. 

The quantitative study of these angles for irrational tori is, in general, very complicated, and has only been established in very particular models and for very particular frequency vectors \cite{MR1478530}. In these works, the authors consider strongly irrational (Diophantine) frequency vectors since these are the ones associated to KAM tori and see that both angles have exponentially small lower bounds.

But, as our system is strongly degenerate, the  
invariant cylinder $\mathcal{P}_\infty$ is foliated by strongly resonant tori. In consequence, as pointed above, the splitting is heavily anisotropic (see the discussion after Theorem \ref{thm:mainthmcriticalpoints}).
In this case, by exploiting the Hamilton-Jacobi formalism together with the analysis of complex extensions of the invariant manifolds, we are able to obtain a lower bound for the splitting angles at homoclinic orbits. 

Although from this lower bound one could indeed deduce the existence of transverse heteroclinic intersections between sufficiently close tori in $\mathcal{P}_\infty$, this approach would yield a statement much weaker than Theorem \ref{thm:intersectionmain}. Indeed, this indirect approach only guarantees that the homoclinic manifolds contain exponentially thin vertical strips instead of an annulus (see Remark \ref{rem:ift}). 

Instead of taking this indirect approach, the main novelty towards the proof of Theorem \ref{thm:intersectionmain} is to look directly at the existence of heteroclinic orbits between different resonant tori in $\mathcal{P}_\infty$. 

\begin{rem}\label{rem:ift}
At a homoclinic orbit to a given torus, one can define two splitting angles. Loosely speaking, they measure the transversality (at the homoclinic orbit) in the direction tangent to the cylinder and the transversality in the direction of the unperturbed energy. The size of the transversality in the direction tangent to the cylinder, which in the problem at hand is polynomial, gives an idea of the distance between the tori that can be connected by heteroclinic orbits. Let us explain why the indirect approach described above (to deduce the existence of heteroclinics from the existence of homoclinics) only allows us to connect tori which are exponentially close.

As the invariant manifolds of invariant tori have codimension two inside the fixed energy level, to construct heteroclinic orbits between a pair of invariant tori one needs to measure the splitting between their manifolds along two independent directions. In order to do this using perturbative techniques, it is necessary to locate which is the direction of exponentially small splitting for the given pair of tori. Then, one can define the distance between the manifolds along this direction and another (suitable) independent direction. 

Once these two directions are identified, one can prove that the  Melnikov vector is the leading term in the approximation of the distance in these two directions. Finally, zeros of the distance  correspond to heteroclinic orbits between the pair of invariant tori.

The reason why the indirect method described above fails to connect tori separated at polynomial distances is that the suitable directions  for measuring the splitting depend heavily on the pair of tori considered. 
Namely, in the indirect method we always work with  the directions adapted to measure the splitting of the invariant manifolds of the same torus. However, these directions change if the distance between the tori is not exponentially close. 
The key of the direct method proposed in this work is to understand which are the suitable directions, found as a new system of coordinates, to measure the splitting according to the pair of given invariant tori. 
\end{rem}

For the sake of clarity, from now on, the discussion is adapted to the problem at hand, since, developing a general theory of splitting for invariant manifolds is way beyond the scope of this paper. We however believe that our techniques can be of interest to analyze  more general models including \eqref{eq:modelhamiltonian}, in the framework of \emph{properly degenerate} a priori stable Hamiltonians. A brief discussion  can be found in Section \ref{sec:degenerateHamiltonians}.

We now present the proof of Theorem \ref{thm:intersectionmain}. It is divided in several steps. First, we introduce a suitable scaling in Section \ref{sec:scaling}, which puts the system in the form \eqref{eq:modelhamiltonian}. Then we introduce the Hamilton-Jacobi formalism in Section \ref{sec:HJformalism}, where we obtain Lagrangian graph parametrizations of the stable and unstable manifolds of the invariant tori contained in $\mathcal{P}_\infty$. Finally, in Section \ref{sec:nonexactlagrangian}, we deduce the proof of Theorem \ref{thm:intersectionmain} from the existence of manifolds of critical points of a certain scalar function related to the Lagrangian graph parametrizations of the invariant manifolds.

In the remaining of Section \ref{sec: The intersection problem}  it is convenient to come back to the  original polar coordinates $(r,\alpha,t,y,G,E)$ (instead of using McGehee's coordinates). We abuse notation and denote by $\mathcal{P}_\infty$ the invariant submanifold in \eqref{eq: the infinity manifold} described in polar coordinates.

\subsubsection{Scaling in a neighborhood of an invariant torus}\label{sec:scaling} 
We fix an invariant torus 
\begin{equation}\label{def:torus}
\mathcal{T}_{G_0}=\mathcal{P}_\infty\cap\{G=G_0\}\qquad \text{with}\qquad G_0\gg 1.
\end{equation}
In order to look for heteroclinic orbits between pairs of tori contained in a neighborhood of $\mathcal{T}_{G_0}$, it is to convenient introduce the conformally symplectic scaling
\begin{equation}\label{eq:scaling}
\eta_{G_0}:(\tilde r,\tilde \alpha,t,\tilde y,\tilde G,\tilde E)\mapsto (G_0^2\tilde r, \tilde \alpha, t, G_0^{-1}\tilde y,G_0\tilde G, G_0\tilde E).
\end{equation}
After time reparametrization (multiplying by $G_0^3)$, the scaled system is Hamiltonian with respect to the area form $\mathrm{d}\tilde y\wedge \mathrm d\tilde r+\mathrm{d}\tilde G\wedge \mathrm d\tilde \alpha+\mathrm{d}\tilde E\wedge \mathrm d t$ and the Hamiltonian 
\begin{equation}\label{eq:scaledhamiltonian}
\widetilde H(\tilde r,\tilde \alpha,t,\tilde y,\tilde G,\tilde E;G_0)=H_0(\tilde r,\tilde y,\tilde G)+G_0^3 \tilde E+ \widetilde V(\tilde r,\tilde\alpha,t;G_0)
\end{equation}
with 
\[
H_0(\tilde r,\tilde y,\tilde G)=\frac{\tilde y^2}{2}+\frac{\tilde G^2}{2\tilde r^2}-\frac{1}{\tilde r}
\]
 being the (integrable) Hamiltonian of the 2-body problem and, as we will see in Lemma \ref{lem:boundspotential} (see also Appendix \ref{sec:Melnikov}), 
 \[
 \widetilde V(\tilde r,\tilde\alpha,t;G_0)= G_0^2V_{\mathrm{pol}}(G_0^2\tilde r, \tilde \alpha, t)-\frac{1}{\tilde r}=\mathcal{O}(\mu G_0^{-4}).
 \]
One can observe that the Hamiltonian \eqref{eq:scaledhamiltonian} has the same structure as the one in \eqref{eq:modelhamiltonian} where now $G_0^{-3}$ plays the role of $\varepsilon^{-1/2}$ and $\mu G_0^{-4}$ plays the role of $\mu$.

\begin{rem}\label{rem:scalingremark}
    If we do not reparametrize the time we obtain the Hamiltonian 
    \[
    H\circ\phi_{G_0}= E + G_0^{-3} H_0+ \mathcal{O}(\mu G_0^{-7}).
    \]
    Then, for large values of $G_0$, one can consider the integrable Hamiltonian $\widetilde H_0=E+G_0^{-3} H_0$  as a perturbation of the Liouville-Arnold integrable Hamiltonian $E$ (see Section \ref{sec:aprioriunstablestable}).
\end{rem}
\subsubsection{The Hamilton-Jacobi formalism}\label{sec:HJformalism}
In this section we explain how to use the Hamilton-Jacobi formalism to obtain Lagrangian graph parametrizations of the invariant manifolds associated to any invariant torus sufficiently close to $\mathcal{T}_{G_0}$ (which in the scaled coordinates corresponds to $\widetilde G=1$). One of the features of our approach is that, uniformly, for all the tori in some neighborhood of $\mathcal{T}_{G_0}$, we  treat their stable/unstable manifolds as small perturbations of the (same) unperturbed homoclinic manifold to $\mathcal{T}_{G_0}$. As we will see in Section \ref{sec:nonexactlagrangian}, this choice is extremely convenient for comparing the parametrizations of the invariant manifolds associated to different invariant tori in a neighborhood of $\mathcal{T}_{G_0}$.

Let us now make precise the discussion above. We consider a torus $\mathcal{T}_{G^\star}$ contained in  a small neighborhood of $\mathcal{T}_{G_0}$. Due to their Lagrangian character, (part of) its stable and unstable manifolds can be parameterized in terms of generating functions which are solutions to a Hamilton-Jacobi equation. 

\begin{rem}
    The global unperturbed homoclinic manifold does not admit a graph parametrization in the original $(r,\alpha,t,y,G,E)$ coordinates (neither in the scaled ones). It is for this reason that we introduce two generating functions $T_{\mathrm{h}}^{\mathrm{u,s}}$, each of them associated to a piece of the global homoclinic manifold. Later, we introduce a new coordinate system \eqref{eq:changebasis} in which the unperturbed homoclinic manifold does admit a global parametrization.
\end{rem}

Since, for large $G_0$, the Hamiltonian $\widetilde H$ in \eqref{eq:scaledhamiltonian} is $\mathcal{O}(G_0^{-4})$ close to $H_0+G_0^{3}\widetilde E$, we write $\tilde z=(\tilde r,\tilde \alpha, t)$ and look for a generating function (of the invariant manifolds of $\mathcal T_{G^\star}$ in scaled coordinates) of the form 
$   T^{\mathrm{u,s}}_{\mathrm{h}}(\tilde z)+\widetilde T^{\mathrm{u,s}}(\tilde z;G^\star)$, 
where $T^{\mathrm{u,s}}_{\mathrm{h}}=T^{\mathrm{u,s}}_{\mathrm{h}}(\tilde r)$ are the (unique) solutions to the unperturbed Hamilton-Jacobi equation 
\[
H_0(\tilde r,\partial_{\tilde r}T^{\mathrm{u,s}}_{\mathrm{h}},1)=0,
\]
satisfying $\partial_{\tilde r} T^{\mathrm{u}}_{\mathrm{h}}<0$ and $\partial_{\tilde r} T^{\mathrm{s}}_{\mathrm{h}}>0$. Then, $\widetilde T^{\mathrm{u,s}}$ are uniquely determined by the Hamilton-Jacobi equation
\begin{equation}\label{eq:firstHJ}
\widetilde H\big(\tilde z,\ \tilde{\boldsymbol\delta} +\mathrm{d} (T^{\mathrm{u,s}}_{\mathrm{h}}+\widetilde T^{\mathrm{u,s}})(\tilde z;G^\star)\big)=0, \qquad\qquad\lim_{r\to\infty} \mathrm{d} \widetilde T^{\mathrm{u,s}}(\tilde z;G^\star)=0
\end{equation}
where $\tilde{\boldsymbol\delta} =(0,\tilde \delta,0)=(0,G^\star/G_0,0)$ accounts for the cohomology class of the torus $\mathcal{T}_{G^\star}$ when working in the scaled coordinates \eqref{eq:scaling}. Notice that $\widetilde T^{\mathrm{u,s}}$ yield graph parametrizations of the local stable/unstable manifolds of the torus $\mathcal{T}_{G^\star}$ that are of the form 
\[
\tilde y=\partial_{\tilde r}(T^{\mathrm{u,s}}_{\mathrm{h}}+ \widetilde T^{\mathrm{u,s}})(\tilde r,\tilde\alpha,t), \quad\qquad \tilde G=\tilde\delta+\partial_{\tilde \alpha}(T^{\mathrm{u,s}}_{\mathrm{h}}+ \widetilde T^{\mathrm{u,s}})(\tilde r,\tilde\alpha,t), \quad\qquad \tilde E=\partial_{t} (T^{\mathrm{u,s}}_{\mathrm{h}}+ \widetilde T^{\mathrm{u,s}})(\tilde r,\tilde\alpha,t).
\]

Instead of studying directly  solutions to \eqref{eq:firstHJ}, following Sauzin (see \cite{MR1841877}), it is convenient to make an extra change of variables. The idea of this change of variables, is to introduce new coordinates, for which the unperturbed dynamics in the homoclinic manifold are given by a linear translation on the base. To that end, let 
\begin{equation}\label{eq:changebasis}
\tilde r=r_{\mathrm{h}}(u),\qquad\qquad\tilde\alpha=\beta+\alpha_{\mathrm{h}}(u),
\end{equation}
where $r_{\mathrm{h}}(u),\alpha_{\mathrm{h}}(u)$ are the time parametrization of the unperturbed homoclinic manifold to the torus $\mathcal T_{G_0}$ (which in the scaled coordinates corresponds to $\tilde G=1$) (see Appendix \ref{sec:perturbregime}). 
Then, we consider the change of variables $\phi_{\mathrm{h}}:(u,\beta,t)\mapsto (\tilde r,\tilde\alpha,t)$ in the base, and denote by 
\begin{equation}\label{eq:Mathieutransf}
\Phi_{\mathrm{h}}:(u,\beta,t,Y,J,E)\mapsto (\tilde r,\tilde \alpha,t,\tilde y,\tilde G,\tilde E)
\end{equation}
its associated Mathieu transformation composed with the translation $\tilde G=1+J$ (see Section \ref{sec:continuationunstable}). The coordinate transformation \eqref{eq:Mathieutransf} is well defined except at $u=0$ since $y_{\mathrm{h}}(0)=0$ (see Lemma \ref{lem:uperturbedhomoclinic}). In particular, for any $u_0>0$, it is well defined on any domain of the form $(u,\beta,t)\in (-\infty,-u_0)\times\mathbb{T}^2$, or of the form $(u,\beta,t)\in (u_0,\infty)\times\mathbb{T}^2$ so we will be able to obtain parametrizations of the local stable and unstable manifolds in these  coordinates.

In the new coordinate system, the unperturbed homoclinic manifold can be entirely parametrized as a graph. Namely, there exists a function $T_{\mathrm{h}}:\mathbb{R}\to\mathbb{R}$ such that it coincides with $T^{\mathrm{u}}_{\mathrm{h}}\circ\phi_{\mathrm{h}}$ for $u\leq 0$ and with $T^{\mathrm{s}}_{\mathrm{h}}\circ\phi_{\mathrm{h}}$ for $u\geq 0$. In view of this, the generating functions which parametrize the (local) stable and unstable manifolds of the torus $\mathcal{T}_{G^\star}$, can be written in the form $
T_{\mathrm{h}}(u)+ T^{\mathrm{u,s}}(z,G^\star)$ (here $z=(u,\beta,t)$),
where $T^{\mathrm{u,s}}$ are uniquely determined by the Hamilton-Jacobi equation 
\begin{equation}\label{eq:HJfinal}
H\big(z,\ \boldsymbol\delta +\mathrm{d} (T_{\mathrm{h}}+ T^{\mathrm{u,s}})(z;G^\star)\big)=0,\qquad\qquad\text{with}\qquad\qquad H=\widetilde H\circ\Phi_{\mathrm{h}},
\end{equation}
$\boldsymbol \delta=(0,\delta,0)$ with $\delta=\tilde G^\star-1=(G^\star-G_0)/G_0$ and the associated asymptotic conditions at infinity
\begin{equation}\label{eq:bcHJ}
\lim_{u\to -\infty}\mathrm{d} T^{\mathrm{u}}(u,\beta,t)=0,\qquad\qquad \lim_{u\to +\infty}\mathrm{d} T^{\mathrm{s}}(u,\beta,t)=0.
\end{equation}
\begin{rem}
To check that $T^{\mathrm{u,s}}$ satisfy the asymptotic conditions \eqref{eq:bcHJ}, we have used that $\mathrm{d}\alpha_h/\mathrm{d}u$ vanishes for $u\to\pm\infty$, (see Lemma \ref{lem:uperturbedhomoclinic}).
\end{rem}

The upshot  of introducing the change of variables $\phi_h$, is that the partial differential equation \eqref{eq:HJfinal} now reads
\[
(1+A^{\mathrm{u,s}})\partial_u T^{\mathrm{u,s}}+ B^{\mathrm{u,s}} \partial_\beta T^{\mathrm{u,s}} +G_0^3 \partial_t T^{\mathrm{u,s}}-V=0
\]
for certain (explicit) $A^{\mathrm{u,s}},B^{\mathrm{u,s}}=\mathcal{O}(|\mathrm{d} T^{\mathrm{u,s}}|, |\delta|)$ and 
\begin{equation}\label{eq:dfnperturbativepotential}
V(u,\beta,t;G_0)=\widetilde V( r_{\mathrm{h}}(u),\beta+\alpha_h(u),t;G_0),
\end{equation}
where $\widetilde V$ is defined in \eqref{eq:scaledhamiltonian}. That is, the partial differential operator is, up to non-linearities and small terms, of constant coefficients. On the other hand, since $V=\mathcal{O}(G_0^{-4})$ one expects the non-linear terms to be small. In this setting, it is not difficult to prove that given $u_0>0$, for sufficiently large $G_0$, the initial value problems \eqref{eq:HJfinal}, \eqref{eq:bcHJ}  admit unique solutions defined on domains of the form $(u,\beta,t)\in (-\infty,-u_0)\times\mathbb{T}^2$ and $(u,\beta,t)\in (u_0,\infty)\times\mathbb{T}^2$ respectively. 
The existence of solutions to these initial value problems implies the existence of parametrizations of the local unstable
\begin{equation}\label{eq:stableparam}
\begin{aligned}
\mathcal{W}^{\mathrm{u}}_{G^\star}:(-\infty,-u_0)\times\mathbb{T}^2&\longrightarrow W^{\mathrm{u}}(\mathcal{T}_{G^\star})\\
z&\longmapsto (z,\boldsymbol\delta+\mathrm{d}(T_{\mathrm{h}}+T^{\mathrm{u}})(z))
\end{aligned}
\end{equation}
and of the local stable manifolds
\begin{equation}\label{eq:unstableparam}
\begin{aligned}
\mathcal{W}^{\mathrm{s}}_{G^\star}:(u_0,\infty)\times\mathbb{T}^2&\longrightarrow W^{\mathrm{s}}(\mathcal{T}_{G^\star})\\
z&\longmapsto (z,\boldsymbol\delta+\mathrm{d}(T_{\mathrm{h}}+T^{\mathrm{s}})(z)).
\end{aligned}
\end{equation}
In the following proposition we show that, as long as the torus $\mathcal{T}_{G^{\star}}$ is sufficiently close to $\mathcal{T}_{G_0}$, the generating functions $T^{\mathrm{u,s}}$ in  \eqref{eq:HJfinal} can be continued, so that the parameterizations \eqref{eq:stableparam}, \eqref{eq:unstableparam}, can be defined over a common domain. To make precise what we mean by sufficiently close we introduce the interval
\begin{equation}\label{eq:Gintervaltori}
\mathbb{G}(G_0)=(G_0-\zeta G_0^{-4}, G_0+\zeta G_0^{-4}).
\end{equation}

\begin{prop}\label{prop:extensiondecaygenfunctions}
 Let $G_0$ be large enough, let $0\leq \zeta\leq G_0^{-3}$ and  $G^{\star}\in\mathbb{G}(G_0)$. Then, the functions $T^{\mathrm{u,s}}$ satisfying \eqref{eq:HJfinal} and \eqref{eq:bcHJ} admit a unique analytic continuation to certain domains of the form $(u,\beta,t)\in R^{\mathrm{u,s}}\times\mathbb{T}^2$ where $R^{\mathrm{u,s}}\subset \mathbb{R}$ are such that $R^{\mathrm{u}}\cap R^{\mathrm{s}}$ is a non-empty open interval. Moreover,
\[
| \mathrm d {T}^{\mathrm{u,s}}(u,\beta,t)| \lesssim G_0^{-4} \qquad\qquad \forall (u,\beta,t)\in R^{\mathrm{u,s}}\times\mathbb{T}^2
\]
and 
\[
T^{\mathrm{u}}(u,\beta,t)=\mathcal{O}(|u|^{-1/3})\quad\text{as}\ u\to-\infty\qquad\text{and}\qquad T^{\mathrm{s}}(u,\beta,t)=\mathcal{O}(|u|^{-1/3})\quad\text{as}\ u\to +\infty.
\]
\end{prop}

\begin{rem}\label{rem:dominios}
Ideally, one would try to extend the unstable parameterization to  $R^{\mathrm{u}}=(-\infty,\tilde{u}_0]$ and the stable one to $R^{\mathrm{s}}=[-\tilde{u}_0,\infty)$ for some $\tilde{u}_0>0$ so $0\in R^{\mathrm{u}}\cap R^{\mathrm{s}}$. However, for technical reasons, we are not able to define the parameterizations \eqref{eq:stableparam} at $u=0$. Yet, we can extend  $T^{\mathrm{u}}$ to a complex domain $R^{\mathrm{u}}$ (which does not contain the point $u=0$) and such that $R^{\mathrm{u}}\cap R^{\mathrm{s}}$ is a non-empty open interval (see Section \ref{sec:Extensionflow}).
\end{rem}

 The existence of real-analytic solutions to the Hamilton-Jacobi equation describing parameterizations of the invariant manifolds of partially hyperbolic invariant tori (in perturbative settings) has been considered in a number of works (see, for example, \cite{ MR1841877} and the references therein). However, as far as we know, due to the appearance of certain unbounded operator, in all previous works solutions to this equation were constructed in an implicit way. Namely, the approach was to prove the existence of vector parameterizations of the invariant manifolds and, a posteriori, to deduce the existence of solutions to the associated Hamilton-Jacobi equation. The main novelty in the proof of Proposition \ref{prop:extensiondecaygenfunctions}, presented in Section \ref{sec:generatingfunctsinvmanifolds}, is that  we are able to prove directly the existence of solutions to \eqref{eq:HJfinal} making use of a Newton-like iterative construction.

\subsubsection{A non-exact Lagrangian intersection problem}\label{sec:nonexactlagrangian}
We are now ready to study the existence of heteroclinic orbits connecting (possibly different) tori in a neighborhood of $\mathcal{T}_{G_0}$ and obtain a proof of Theorem \ref{thm:intersectionmain}.

Consider two tori $\mathcal T_{G^{\mathrm{u}}}, \mathcal T_{G^{\mathrm{s}}}\subset \mathcal{P}_\infty$ (see \eqref{def:torus}) contained in a sufficiently small neighborhood of $\mathcal{T}_{G_0}$. More concretely $G^{\mathrm{u}}, G^{\mathrm{s}} \in \mathbb{G}(G_0)$ where $\mathbb{G}(G_0)$  is defined in \eqref{eq:Gintervaltori}. Our goal is to show that the unstable manifold of the torus $\mathcal T_{G^{\mathrm{u}}}$ intersects transversally the stable manifold of the torus $\mathcal T_{G^{\mathrm{s}}}$. Observe that Proposition \ref{prop:extensiondecaygenfunctions} implies the existence of parametrizations of these manifolds  $W^{\mathrm{u}}(\mathcal T_{G^{\mathrm{u}}})$ and $W^{\mathrm{s}}(\mathcal T_{G^{\mathrm{s}}})$ of the form \eqref{eq:stableparam} and \eqref{eq:unstableparam}, which are defined on a common domain. Hence, proving the existence of transverse heteroclinic connections between these tori boils down to showing the existence of non-degenerate solutions of the equation
\begin{equation}\label{eq:zerostransverse}
(\boldsymbol{\delta}^{\mathrm{u}}-\boldsymbol{\delta}^{\mathrm{s}})+\mathrm{d} (T^{\mathrm{u}}-T^{\mathrm{s}})(z;G^{\mathrm{u}},G^{\mathrm{s}})=0.
\end{equation}
Now, notice that, if we define
\begin{equation}\label{eq:defngeneratingfunctmanif}
 S^{\mathrm{u,s}}(z;G^{\mathrm{u,s}})=\langle\boldsymbol\delta^{\mathrm{u,s}},z\rangle+(T_h+T^{\mathrm{u,s}})(z;G^{\mathrm{u,s}}),\qquad  z=(u,\beta,t),
 \end{equation}
 and let
 \begin{equation}\label{eq:diffgenfunct}
 \Delta S(z;G^{\mathrm{u}},G^{\mathrm{s}})=S^{\mathrm{u}}(z;G^{\mathrm{u}})-S^{\mathrm{s}}(z;G^{\mathrm{s}}),
 \end{equation}
 then, the existence of non-degenerate solutions to \eqref{eq:zerostransverse}, is equivalent to the existence of non-degenerate critical points of $z\mapsto \Delta S$. 

The function $\Delta S$ in \eqref{eq:diffgenfunct} generalizes the notion of splitting potential, which was first introduced in \cite{MR1274762} (see also \cite{MR1766491}) for the homoclinic case, to the heteroclinic case. However, we point out that, for the heteroclinic case, i.e. $\boldsymbol\delta^{\mathrm{u}}\neq \boldsymbol\delta^{\mathrm{s}}$, $\Delta S$ is not $2\pi$- periodic in $\beta$, what reflects the non-exact nature of the problem and excludes the possibility (at least in a straightforward manner) of applying topological methods such as Ljusternik-Schnirelman theory to prove the existence of critical points of \eqref{eq:diffgenfunct}, as is usually done in the homoclinic case $\boldsymbol\delta^{\mathrm{u}}=\boldsymbol\delta^{\mathrm{s}}$ (see \cite{MR1274762}).

In Theorem \ref{thm:mainthmcriticalpoints} below we establish the existence of two manifolds of critical points for the function $\Delta S$ defined in \eqref{eq:diffgenfunct}. The main ingredient, contained in Proposition \ref{prop:approxgenfunctionbyMelnikov}, is the approximation of $\Delta S$ by the  \textit{Melnikov potential} defined in \eqref{eq:defnMelnikovPotential}. 
Roughly speaking, the proof of Proposition \ref{prop:approxgenfunctionbyMelnikov} is based on the observation that $\Delta S$, defined in \eqref{eq:diffgenfunct}, belongs to the kernel of a partial differential operator which is close to the linear partial differential operator with constant coefficients $\mathcal L= \partial_u+G_0^3 \partial_t$, which encodes the dynamics along the unperturbed homoclinic \eqref{eq:changebasis} (see Section \ref{sec:differencegenfunctionsmanifolds} and, in particular, Lemma  \ref{lem:exponsmallness}). To state the result in Proposition \ref{prop:approxgenfunctionbyMelnikov}, for $\rho>0$ we denote by $\mathbb{G}_\rho(G_0)$ the complex $\rho$-neighborhood of the interval $\mathbb{G}(G_0)$ in \eqref{eq:Gintervaltori}.

\begin{rem}
    The use of complex domains in the variable $G$ is needed to provide quantitative estimates for the complex extension of the scattering maps in Section \ref{sec:proofthm24}.
\end{rem}

\begin{prop}\label{prop:approxgenfunctionbyMelnikov}
Let $\Delta S (u,\beta,t;G^{\mathrm{u}},G^{\mathrm{s}})$ be the function defined in \eqref{eq:diffgenfunct} and let $0<v_1<v_2$ be two fixed real numbers. There exists $\rho>0$  such that, for $G_0$ large enough, and $0\leq \zeta< G_0^{-3}$, then, for any $G^{\mathrm{u}},G^{\mathrm{s}}\in \mathbb{G}_\rho(G_0)$,  there exist an analytic (real-analytic if $G^{\mathrm{u}},G^{\mathrm{s}}\in\mathbb{R}$) close to the identity local change of variables 
\[
\begin{aligned}
\Phi(\cdot;\ G^{\mathrm{u}},G^{\mathrm{s}}):[v_1,v_2]\times \mathbb{T}^2_\rho&\longrightarrow &\mathbb{C}\times \mathbb{T}^2_{2\rho}\\
(v,\theta,t)&\longmapsto& (u,\beta,t)
\end{aligned}
\]
and an analytic (real-analytic if $G^{\mathrm{u}},G^{\mathrm{s}}\in\mathbb{R}$)  function 
\[
\begin{aligned}
\Delta \mathcal{S}(\cdot;G^{\mathrm{u}},G^{\mathrm{s}}):\mathbb{T}^2_\rho&\longrightarrow &\mathbb{C}\\
(\sigma,\theta)&\longmapsto& \Delta \mathcal{S}(\sigma,\theta;G^{\mathrm{u}},G^{\mathrm{s}}),
\end{aligned}
\]
such that 
\begin{equation}\label{eq:deltascali}
\Delta \mathcal{S} (t-G_0^3v,\theta;G^{\mathrm{u}},G^{\mathrm{s}})=\Delta S \circ \Phi (v,\theta,t;G^{\mathrm{u}},G^{\mathrm{s}}).
\end{equation}
Moreover, if we define the Melnikov potential
\begin{equation}\label{eq:defnMelnikovPotential}
L(\sigma,\theta; G_0)=\int_{\mathbb{R}} V(s,\theta,\sigma +G_0^3 s;G_0) \mathrm{d}s,
\end{equation} 
where $V$ is defined in \eqref{eq:dfnperturbativepotential}, there exists a constant  $C>0$ independent of $G_0$ and $\zeta$ such that the following estimates are satisfied
\begin{itemize}
\item $\displaystyle\left|\Delta \mathcal{S}(\sigma,\theta)-(\delta^{\mathrm{u}}-\delta^{\mathrm{s}})\theta -L(\sigma,\theta)\right| \leq G_0^{-7}$.
\item $\displaystyle\left|\Delta \mathcal{S}^{[l]}(\theta)-L^{[l]}(\theta)\right|\lesssim (CG_0)^{-4+3|l|/2} \exp(- |l|G_0^3)/3),$
 \end{itemize}
 (here $h^{[l]}$ denotes the $l$-th Fourier coefficient of a $2\pi$-periodic function $\sigma\mapsto h(\sigma)$).
 \end{prop}

Proposition \ref{prop:approxgenfunctionbyMelnikov} is proved in Section \ref{sec:generatingfunctsinvmanifolds}, where we perform the analytic continuation of the stable and unstable generating functions $S^{\mathrm{u},\mathrm{s}}$ in \eqref{eq:defngeneratingfunctmanif} up to a common domain where we can study their difference $\Delta S=S^{\mathrm{u}}-S^{\mathrm{s}}$. The core of Proposition \ref{prop:approxgenfunctionbyMelnikov} is, using that $\Delta \mathcal{S}$, defined in \eqref{eq:deltascali} is $2\pi$-periodic in $\sigma$, obtain an approximation of its harmonics  in terms of the Melnikov potential $L$ in \eqref{eq:defnMelnikovPotential}. Since, as is proven in Section \ref{sec:generatingfunctsinvmanifolds}, $L$ possesses  non-degenerate critical points, a quantitative application of the implicit function theorem yields the  next theorem. Its proof  is deferred to Section \ref{sec:generatingfunctsinvmanifolds}.

\begin{thm}\label{thm:mainthmcriticalpoints}
There exists $\rho_*>0$ such that, for $G_0$ large enough, and $0\leq \zeta\leq G_0^{-3}$, then, for all $G^{\mathrm{u}}\in\mathbb{G}_{\rho_*}(G_0)$, there exist two real analytic functions
\[
(\theta,G^{\mathrm{u}}) \mapsto ( \sigma_\pm(\theta,G^{\mathrm{u}}), \tilde{G}^{\mathrm{s}}_\pm (\theta,G^{\mathrm{u}}))
\]
such that 
\[
\partial_\sigma \Delta \mathcal{S}(\sigma_\pm(\theta,G^{\mathrm{u}}),\theta; G^{\mathrm{u}},\tilde{G}^{\mathrm{s}}_\pm(\theta,G^{\mathrm{u}}))=0 \qquad\qquad \partial_\theta \Delta \mathcal{S}(\sigma_\pm(\theta,G^{\mathrm{u}}),\theta; G^{\mathrm{u}},\tilde{G}^{\mathrm{s}}_\pm(\theta,G^{\mathrm{u}}))=0.
\]
Moreover, the determinant of the Hessian matrix of the function $(\sigma,\theta)\mapsto \Delta \mathcal{S}(\sigma,\theta;G^{\mathrm{u}},G^{\mathrm{s}})$ evaluated at $(\sigma,\theta; G^{\mathrm{u}},G^{\mathrm{s}})=(\sigma_\pm(\theta,G^{\mathrm{u}}),\theta;G^{\mathrm{u}},\tilde{G}^{\mathrm{s}}_\pm (\theta,G^{\mathrm{u}}))$ is different from zero for all $(\theta,G^{\mathrm{u}})\in \mathbb{T}_{\rho_*}\times \mathbb{G}_{\rho_*}(G_0)$.
\end{thm}

Before analyzing the consequences of  Theorem \ref{thm:mainthmcriticalpoints} it  is worth pointing out some remarks. The first one is that,  in the proof of Proposition \ref{prop:approxgenfunctionbyMelnikov}, carried out in Section  \ref{sec:generatingfunctsinvmanifolds}, we will see that 
\[
\max_{\theta\in\mathbb{T}}|\partial_\sigma \Delta \mathcal{S}(\sigma,\theta;G^{\mathrm{u}},G^{\mathrm{s}})|\sim G_0^{-1/2} \exp(-G_0^3/3)\qquad\qquad \max_{\theta\in\mathbb{T}}|\partial_\theta \Delta \mathcal{S}(\sigma,\theta;G^{\mathrm{u}},G^{\mathrm{s}})|\sim \zeta\  G_0^{-5}.
\] 
That is, as was anticipated in Section \ref{sec: The intersection problem}, the splitting between the stable and unstable manifolds of nearby tori is strongly anisotropic. Namely, the splitting is exponentially small in the direction conjugated to the fast angle $t$, while it is only polynomially small in the direction conjugated to the slow angle $\theta$.

\begin{rem}
    The splitting coordinates $(v,\theta,t)$ are close to the original coordinates $(u,\beta,t)$. However, the change of coordinates from one system to the other depends non trivially on the pair $G^{\mathrm{u}},G^{\mathrm{s}}$.
    \end{rem}

\subsubsection{Completion of the proof of Theorem \ref{thm:intersectionmain}}
The manifolds of critical points of $\Delta\mathcal{S}$ (and therefore of $\Delta S$) obtained in Theorem \ref{thm:mainthmcriticalpoints} are of the form $(\theta,G^{\mathrm{u}})\in\mathbb{T}\times \mathbb{G}(G_0)$ with $\mathbb G(G_0)$  of the form \eqref{eq:Gintervaltori}, in particular, they are diffeomorphic to an annulus. The proof of Theorem \ref{thm:intersectionmain} is complete since, given any $R$, if $0\leq \zeta \leq (G_{*}+R)^{-3}$ we can cover the interval $(G_*, G_*+ R)$ by a sufficiently large number of overlapping intervals $\mathbb G (G_i)$ in order to  construct the homoclinic manifolds $\Gamma_\pm$.

\subsection{Proof of Theorem \ref{thm:Moeckel}}\label{sec:proofthm24}
We now give the proof of Theorem \ref{thm:Moeckel}. It is divided in several steps. First, in Section \ref{sec:constrscattmaps}, we obtain a suitable parameterization of the homoclinic manifolds $\Gamma_\pm$ obtained in Theorem \ref{thm:intersectionmain}. To verify \eqref{eq:firstmoeckelcondscatt} in Theorem \ref{thm:Moeckel}, in Section \ref{sec:qualitativpropertiesscatt} we obtain an asymptotic formula for the scattering maps and check that the leading order term, given by an explicit integral (the reduced Melnikov potential) has the desired form. 
In Section \ref{sec: Analysis of the scattering maps} we construct a suitable generating function for the scattering maps.  Finally in Section \ref{sec:proofthm24}, using these generating functions,  we check that using the choice of parameters \eqref{eq:firstmoeckelcondscatt} and \eqref{eq:secondmoeckelcondscatt} in Theorem \ref{thm:Moeckel} one can show that the hypotheses \eqref{eq:secondcondmapmoeckel}, \eqref{eq:suffcondmoeckel2} and \eqref{eq:suffcondmoeckel} in  Theorem \ref{thm:scattmapmain} hold and therefore one can apply this theorem to  complete the proof of Theorem \ref{thm:Moeckel}.
%
The key point in our construction of the generating functions (of the scattering maps) is that we relate them to the splitting potential $\Delta S$ in \eqref{eq:diffgenfunct}. This is crucial to obtain an asymptotic formula for the difference between the scattering maps.

\subsubsection{The homoclinic manifolds $\Gamma_\pm$}\label{sec:constrscattmaps}
From now on we fix $G_*\gg 1$, $R>0$ and assume that $0\leq\zeta \leq (G_*+R)^{-3}$ so that Theorem \ref{thm:intersectionmain} ensures the existence of two transverse homoclinic manifolds $\Gamma_\pm$ to $\mathcal P_\infty^*$, diffeomorphic to an annulus and whose $G$ projection covers the interval $[G_*,G_*+R]$. As explained in Section \ref{sec:proofmainthm}, the existence of these homoclinic manifolds ensures the existence of the scattering maps $\mathbb{P}_\pm$ in \eqref{eq:defnscatteringmaps}.

In order to describe the scattering maps $\mathbb P_\pm$, we turn our attention to Theorem \ref{thm:mainthmcriticalpoints}, where the homoclinic manifolds $\Gamma_\pm$ are described as manifolds  of non-degenerate critical points of the function $(\sigma,\theta)\mapsto \Delta \mathcal{S}$.  The key idea behind the proof of Theorem \ref{thm:mainthmcriticalpoints} has been the construction of a bespoke coordinate system for the analysis of each intersection problem:  notice that the change of variables  $\Phi$ introduced for studying the intersection between the invariant manifolds $W^{\mathrm{u}}_{G^{\mathrm{u}}}$ and $ W^{\mathrm{s}}_{G^{\mathrm{s}}}$ depends  on the actions $G^{\mathrm{u},\mathrm{s}}$. Therefore, Theorem \ref{thm:mainthmcriticalpoints} describes the set of heteroclinic orbits using different coordinate systems. Still, in order to analyze the properties of the scattering maps, we need a unified description of the asymptotic dynamics along the families of heteroclinic orbits.

The easy solution to this annoyance is to obtain a parametrization of  the homoclinic channels $\Gamma_\pm$ in the original polar coordinates \eqref{eq:phasespacepolar}. First, we notice that, in Section \ref{sec:scaling} we artificially introduced a perturbative parameter $G_0$. From now on, in order to get rid of this parameter, we will set $G_0=G^{\mathrm{u}}$ and will consider  $G^{\mathrm{u}}\gg 1$.  Now, let $\Phi$ be the change of variables of Theorem \ref{thm:mainthmcriticalpoints}, let $\sigma_\pm(\theta,G^{\mathrm{u}})$ and  $\tilde{G}^{\mathrm{s}}_\pm(\theta,G^{\mathrm{u}})$ be the functions obtained in that theorem (after setting $G_0=G^{\mathrm{u}}$) and for $t=0$, define (recall that $\sigma=t-(G^{\mathrm{u}})^3$v)
\begin{equation}\label{eq:definitionPhi+-}
\Phi_\pm(\theta,G^{\mathrm{u}})= \Phi (-(G^{\mathrm{u}})^{-3}\sigma_\pm(\theta,G^{\mathrm{u}}),\theta,0;G^{\mathrm{u}}, \tilde{G}^{\mathrm{s}}_\pm(\theta,G^{\mathrm{u}})).
\end{equation}
Then, the homoclinic manifolds $\Gamma_\pm\subset M_{\mathrm{pol}}$ (where $M_{\mathrm{pol}}$ is the phase space in polar coordinates) can be parametrized as follows
\begin{equation}\label{eq:paramhomoclinicmanifold}
\Gamma_\pm =\bigg\{(r,\alpha,0,y,G,E)=\eta_{G^{\mathrm{u}}}\circ\Phi_{\mathrm{h}}\circ \mathcal{W}_{G^{\mathrm{u}}}^{\mathrm{u}}\circ \Phi_\pm(\theta,G^{\mathrm{u}})=\eta_{G^{\mathrm{u}}}\circ\Phi_{\mathrm{h}}\circ \mathcal{W}_{G^{\mathrm{s}}}^{\mathrm{s}}\circ \Phi_\pm(\theta,G^{\mathrm{u}}),\ (\theta,G^{\mathrm{u}})\in \mathbb{T}\times [G_*,G_*+R]\} \bigg\}.
\end{equation}
with $\eta_{G^{\mathrm{u}}}$ being the scaling in \eqref{eq:scaling}, $\Phi_{\mathrm h}$ being the change of coordinates in  \eqref{eq:Mathieutransf}, $\mathcal W^{\mathrm{u,s}}_{G^{\mathrm{u,s}}}$ being of the form \eqref {eq:stableparam}, \eqref{eq:unstableparam} and defined over the common domain in Proposition \ref{prop:extensiondecaygenfunctions}.

We notice at this point that, our construction of the homoclinic channels gives much more information about the dynamics of the scattering map in the action component $G$ than in the angle component $\alpha$. Namely, using the parametrization \eqref{eq:paramhomoclinicmanifold} of the homoclinic manifold $\Gamma_\pm$ and writing $x=x_\pm(\theta,G^{\mathrm{u}})$ for a point $x_\pm\in\Gamma_\pm$, the wave maps satisfy
\begin{equation}\label{eq:usefulexpressionswavemaps}
\begin{split}
\Omega^{\mathrm{u}}_\pm(x_\pm(\theta,G^{\mathrm{u}}))= &(\alpha^{\mathrm{u}}_\pm(\theta,G^{\mathrm{u}}), G^{\mathrm{u}})=\lim_{\tau\to -\infty} (\alpha\circ\phi^\tau_{H_\mathrm{pol}}(x_\pm (\theta,G^{\mathrm{u}})),G^{\mathrm{u}})\\
\Omega^{\mathrm{s}}_\pm(x_\pm(\theta,G^{\mathrm{u}}))=& (\varphi_\pm^{\mathrm{s}}(\theta,G^{\mathrm{u}}),G_\pm^{\mathrm{s}}(\theta,G^{\mathrm{u}}))=\lim_{\tau\to + \infty} (\alpha\circ\phi^\tau_{H_\mathrm{pol}}(x_\pm(\theta,G^{\mathrm{u}})), \tilde{G}^{\mathrm{s}}_\pm(\theta,G^{\mathrm{u}}))\\
\end{split}
\end{equation}
so, up to composing with the close to identity transformation $(\Omega_\pm^{\mathrm{u}})^{-1}$, the projection of the scattering map in the direction of the action $G$ is given by the function $\tilde{G}^{\mathrm{s}}_\pm$ obtained in Theorem \ref{thm:mainthmcriticalpoints} and which is determined implicitely in terms of $\Delta \mathcal{S}=\Delta S\circ \Phi$ by the system of equations
\begin{equation}\label{eq:implicitscattmap}
\partial_\theta \Delta \mathcal{S}(\sigma_\pm(\theta,G^{\mathrm{u}}),\theta; G^{\mathrm{u}},\tilde{G}^{\mathrm{s}}_\pm(\theta,G^{\mathrm{u}}))=0\qquad\qquad
\partial_\theta \Delta \mathcal{S}(\sigma_\pm(\theta,G^{\mathrm{u}}),\theta; G^{\mathrm{u}},\tilde{G}^{\mathrm{s}}_\pm(\theta,G^{\mathrm{u}}))=0.
\end{equation}
However, the existence of a direct link between the generating functions which parametrize the invariant manifolds of the tori in $\mathcal{T}_G\subset \mathcal{P}_\infty$, and the angular component of the wave maps, and consequently of the scattering maps, is not clear at the moment.

If one is only interested in obtaining an asymptotic formula for the scattering maps, expressions \eqref{eq:usefulexpressionswavemaps} and \eqref{eq:implicitscattmap}, along with the verification of certain non-degeneracy properties of the  Melnikov function $L$ defined in \eqref{eq:defnMelnikovPotential} and a standard symplecticity argument, are enough. This is the content of Section \ref{sec:qualitativpropertiesscatt}.

However, proving an asymptotic formula for the difference between these maps is much more demanding since, as already explained, the Taylor expansion (in $1/G$) of both maps coincides up to any order. In Section \ref{sec: Analysis of the scattering maps} we establish a direct relationship between two quatities:
the difference  of the  generating functions  associated to the invariant manifolds of a pair of invariant tori $\mathcal{T}_{G^{\mathrm{u}}},\mathcal{T}_{G^{\mathrm{s}}}$, defined as $\Delta S (\cdot\ ; G^{\mathrm{u}},G^{\mathrm{s}})$  in \eqref{eq:diffgenfunct}, and the angular dynamics (i.e. the projection of the dynamics along the angle $\alpha$) along the heteroclinic orbit in $\Gamma_\pm$ which connects the tori $\mathcal{T}_{G^{\mathrm{u}}},\mathcal{T}_{G^{\mathrm{s}}}$. This connection will prove to be crucial in Section \ref{sec:transversalityscattmaps} to obtain asymptotic formulas for the difference between the scattering maps.

\subsubsection{Qualitative and asymptotic properties of the scattering maps}\label{sec:qualitativpropertiesscatt}

In this section we show that \eqref{eq:secondmoeckelcondscatt} in Theorem \ref{thm:Moeckel} holds (notice that this condition is part of the ingredients needed to verify \eqref{eq:secondcondmapmoeckel}). 
The link established between the scattering maps $\mathbb{P}_\pm$  and the difference $\Delta S$ between the generating functions associated to the invariant manifolds of pairs of invariant tori provides very rich information about the qualitative and quantitative properties of $\mathbb{P}_\pm$. This is the content of Theorem \ref{thm:mainthmscattmaps} in which we sum up the qualitative properties and state a global asymptotic formula for $\mathbb{P}_\pm$ in terms of the \textit{reduced  Melnikov potentials} (here we define $\tilde{\sigma}_+(\alpha)=\alpha$, and $\tilde{\sigma}_-(\alpha)=\alpha+\pi$)
\begin{equation}\label{eq: DefinitionredMelnikovPotential}
\mathcal{L}_\pm(\alpha, G)=\int_{\mathbb{R}} G\  V(s,\alpha, \tilde{\sigma}_\pm(\alpha) +G^{3}s;G) \mathrm{d}s
\end{equation}
where $V(u,\beta,t;G)$ is the potential introduced in \eqref{eq:dfnperturbativepotential} (see Remark \ref{rem:defnmelnikovGfactor}). For our construction, it will also be important to introduce the reduced Melnikov potential associated to the circular problem
\begin{equation}\label{eq: DefinitionredMelnikovPotentialCircular}
\mathcal{L}_{\pm,\mathrm{circ}}(G)=\int_{\mathbb{R}} G\   V_{\mathrm{circ}}(s, \tilde{\sigma}_\pm(\alpha)-\alpha+G^{3}s;G) \mathrm{d}s
\end{equation}
where $V_{\mathrm{circ}}(u,t-\beta;G)=V|_{\zeta=0}(u,\beta,t;G)$.
Then, in Theorem \ref{thm:mainthmscattmaps2}, we establish an asymptotic formula for the difference between the scattering maps $\mathbb{P}_+$ and $\mathbb{P}_-$.

\begin{rem}\label{rem:defnmelnikovGfactor}
    The appearance of the factor $G$ in the integrand of the expression \eqref{eq: DefinitionredMelnikovPotential}
 for $\mathcal{L}_\pm$ is just a consequence of how we have defined the perturbative potential $V$ in \eqref{eq:dfnperturbativepotential}. Indeed, given the (time parametrization of the)  homoclinic orbit 
    \[
    r(u)=G^2r_{\mathrm{h}}(u;G),\qquad\qquad\alpha(u)=\alpha+\alpha_{\mathrm{h}}(u;G),\qquad\qquad t(u)=t+u
    \]
    making use of Lemma \ref{lem:uperturbedhomoclinic} we have that 
    \begin{align*}
    \int_{\mathbb{R}} \bigg(V_{\mathrm{pol}}(G^2r_{\mathrm{h}}(u;G),\alpha+&\alpha_{\mathrm{h}}(u;G),t+u)-\frac{1}{G^2r_{\mathrm{h}}(u;G)} \bigg)\mathrm{d}u\\
    =&\int_{\mathbb{R}} \left( V_{\mathrm{pol}}(G^2r_{\mathrm{h}}(G^{-3}u;1),\alpha+\alpha_{\mathrm{h}}(G^{-3}u;1),t+u)-\frac{1}{G^2r_{\mathrm{h}}(G^{-3}u;1)}\right)\mathrm{d}u\\
    =&\int_{\mathbb{R}}G^3\widetilde V(\tilde{r}_{\mathrm{h}}(u),\alpha+\alpha_{\mathrm{h}}(u),t+G^3u;G)\mathrm{d}u\\
=&\int_{\mathbb{R}}G\ V(u,\alpha,t+G^3u;G)\mathrm{d}u.\\
    \end{align*}
\end{rem}

\begin{thm}\label{thm:mainthmscattmaps}
Let $\mathbb{A}=\mathbb{T}\times[G_*,G_*+R]\subset \mathcal{P}_\infty^*$.  
Then, the scattering maps $\mathbb{P}_\pm: \mathbb{A} \to \mathcal{P}_\infty^*$ defined in \eqref{eq:defnscatteringmaps} are exact symplectic and real-analytic. Moreover, there exists $\rho_*>0$ such that the maps $\mathbb{P}_\pm$ admit an analytic extension to $\mathbb{A}_{\rho_*}$ and for all $(\alpha,G)\in \mathbb{A}_{\rho_*}$
\begin{equation}\label{eq:estimatesmaintermscattmap}
\mathbb{P}_\pm=(\mathrm{Id}+\mathcal{J}\nabla \mathcal{L}_\pm) +\left(\mathcal{O}(|G|^{-7}), \mathcal{O}(\zeta |G|^{-11/2}) \right),
\end{equation}
where $\top$ denotes transpose and $\mathcal{L}_\pm$  has been defined in \eqref{eq: DefinitionredMelnikovPotential} and $\mathcal{J}=\begin{pmatrix}
0&-1\\
1&0
\end{pmatrix}$. Moreover, the vectors $\mathcal J\nabla \mathcal{L}_\pm$ are of the form
\[
\mathcal J\nabla \mathcal{L}_\pm=\left( \omega(G) +\mathcal{O}(|G|^{-7}),\  \zeta \ r(\alpha,G) +\mathcal{O} (\zeta |G|^{-7})\right)
\]
with 
\[
\omega(G)= -\mu(1-\mu)\frac{3\pi}{2G^4},\qquad\qquad r(\alpha,G)=\mu(1-\mu)(1-2\mu) \frac{15\pi }{8 G^5}\sin\alpha.
\]
\end{thm}

This result is proved in Section \ref{sec:asymptformscattmap}. Notice that \eqref{eq:estimatesmaintermscattmap} does not guarantee that \eqref{eq:firstmoeckelcondscatt} in Theorem \ref{thm:Moeckel} holds with $\varepsilon =G_*^{-4}$ and $\delta(\varepsilon)=\zeta\ \varepsilon^{5/4}$ since, a priori, there might be error terms in the approximation \eqref{eq:estimatesmaintermscattmap}  which do not contain $\zeta$. 
In order to take care of this subtlety we study the scattering maps for $\zeta=0$. 
To that end, we introduce $\mathbb{A}_{\mathrm{circ}}=\mathbb{T}\times[G_*,\infty)\subset \mathcal{P}_\infty$ and denote by
\begin{equation}\label{eq:scatteringmapcircular}
\mathbb{P}_{\pm,\mathrm{circ}}:\mathbb{A}_{\mathrm{circ}} \to \mathcal{P}_{\infty}^*
\end{equation}
the scattering map \eqref{eq:defnscatteringmaps} associated to the case $\zeta=0$, which corresponds to the circular problem (RPC3BP). 
The following result is an immediate corollary of Theorem \ref{thm:mainthmscattmaps}.

\begin{lem}\label{lem:lemmascattmapcirc}
The scattering map $\mathbb{P}_{\pm,\mathrm{circ}}$, associated to the circular case $\zeta=0$,  is of the form
\[
\mathbb{P}_{\pm,\mathrm{circ}}(\alpha,G)=(\alpha+\omega_{\mathrm{circ}}(G), G)
\]
Moreover, for all $(\alpha,G)\in \mathbb A_{\mathrm{circ}}^* $, we have 
\[
\omega_{\mathrm{circ}}(G)=\partial_{G} \mathcal{L}_{\pm,\mathrm{circ}} (G)+ \mathcal{O}(|G|^{-7})=\omega(G)+\mathcal{O}(|G|^{-7}),
\]
where $\mathcal{L}_{\pm,\mathrm{circ}}$ has been defined in \eqref{eq: DefinitionredMelnikovPotentialCircular}.
\end{lem}

\begin{rem}\label{rem:integrabilitycircular}
The integrability of the scattering map of the circular problem ($\zeta=0$) is a consequence of the conservation of the Jacobi constant (see \cite{MR3455155,MR3583476}). 
\end{rem}

Combining Theorem \ref{thm:mainthmscattmaps} and Lemma \ref{lem:lemmascattmapcirc}, a standard application of Schwarz's lemma shows that
\[
\mathbb{P}_\pm(\alpha,G)=(\alpha+\omega_{\mathrm{circ}}(G),\  G+\zeta\  r(\alpha,G))+\mathcal{O}(\zeta |G|^{-7}).
\]
from where \eqref{eq:firstmoeckelcondscatt} is immediate.

\subsubsection{A generating function for the scattering maps}\label{sec: Analysis of the scattering maps}

It is indeed quite natural to expect a direct relationship between the family of generating functions $S^{\mathrm{u}},S^{\mathrm{s}}$ in \eqref{eq:defngeneratingfunctmanif} and the scattering maps in \eqref{eq:defnscatteringmaps}. However, as far as the authors know, this connection had only  been established up to first order using the so called Melnikov potential (see \cite{MR2383896}). In Theorem  \ref{prop:generatingfunctscattmap} below we show how  $S^{\mathrm{u}}, S^{\mathrm{s}}$ \emph{completely determine} the scattering maps. 

Even if in \eqref{eq:estimatesmaintermscattmap} we have already obtained an asymptotic expression for the scattering maps, this asymptotic expression is not enough to distinguish them. Actually their difference is exponentially small in $1/G^3$. In order to detect this difference it is key to relate the scattering maps to the generating functions $S^{\mathrm{u,s}}$ not only up to first order, so we do capture exponentially small effects (see also Remark \ref{rem:verticalstriprestriction}).

To do so, we first need to look at the manifolds of critical points of the function $(\sigma,\theta)\mapsto \Delta \mathcal{S}(\sigma,\theta;G^{\mathrm{u}},G^{\mathrm{s}})$ in a different way from that in Theorem \ref{thm:mainthmcriticalpoints}. This is the content of the following proposition, which will be proved together with Theorem \ref{thm:mainthmcriticalpoints} in Section \ref{sec:generatingfunctsinvmanifolds}.

\begin{prop}\label{prop:sndthmcriticalpoints}
Let $\Delta \mathcal{S}$ be the function defined in Theorem \ref{thm:mainthmcriticalpoints}. Then, there exists a constant $c>0$ (depending only on $\mu$) such that, for every pair of actions
\begin{equation}\label{eq:domainactions}
(G^{\mathrm{u}},G^{\mathrm{s}})\in \mathcal{R}_G\equiv \left\{ (G^{\mathrm{u}},G^{\mathrm{s}})\in (G_*,G_*+R)\times(G_*,G_*+R) \colon  |G^{\mathrm{s}}-G^{\mathrm{u}}|< c\zeta/(G^{\mathrm{u}})^{5}\right \},
\end{equation}
one can find real-analytic  functions
\[
(G^{\mathrm{u}},G^{\mathrm{s}})\mapsto (\hat{\sigma}_\pm(G^{\mathrm{u}},G^{\mathrm{s}}),\hat{\theta}_\pm (G^{\mathrm{u}},G^{\mathrm{s}})),
\]
such that 
\[
\partial_\sigma \Delta \mathcal{S}(\hat{\sigma}_\pm(G^{\mathrm{u}}, G^{\mathrm{s}}),\hat{\theta}_\pm (G^{\mathrm{u}},G^{\mathrm{s}}); G^{\mathrm{u}},G^{\mathrm{s}})=0 \qquad\qquad \partial_\theta \Delta \mathcal{S}(\hat{\sigma}_\pm(G^{\mathrm{u}},G^{\mathrm{s}}),\hat{\theta}_\pm (G^{\mathrm{u}},G^{\mathrm{s}}); G^{\mathrm{u}},G^{\mathrm{s}})=0.
\]
\end{prop}

Proposition  \ref{prop:sndthmcriticalpoints} provides, in some sense, a more natural way to look for the critical points of the function $\Delta \mathcal{S}$ than the one in Theorem \ref{thm:mainthmcriticalpoints}: We fix a sufficiently close (but not necessarily exponentially close) pair of actions $(G^{\mathrm{u}},G^{\mathrm{s}})$ and look at the values of the angles $(\sigma,\theta)$ for which there exists a critical point of $(\sigma,\theta)\mapsto \Delta \mathcal{S}$. Next theorem gives the connection between the generating functions associated to the invariant manifolds and the scattering maps.

\begin{thm}\label{prop:generatingfunctscattmap}
Let $(G^{\mathrm{u}},G^{\mathrm{s}})\in\mathcal{R}_G$  where $\mathcal{R}_G$ is the domain defined in \eqref{eq:domainactions}, let  $\hat{\sigma}_\pm(G^{\mathrm{u}},G^{\mathrm{s}}),\hat{\theta}_\pm (G^{\mathrm{u}},G^{\mathrm{s}})$ be the functions obtained in Proposition \ref{prop:sndthmcriticalpoints} and define
\begin{equation}\label{eq:difngenfunctactions}
\mathtt{S}_\pm(G^{\mathrm{u}},G^{\mathrm{s}})= G^{\mathrm{u}}\Delta \mathcal{S}(\hat{\sigma}_\pm(G^{\mathrm{u}}, G^{\mathrm{s}}),\hat{\theta}_\pm (G^{\mathrm{u}},G^{\mathrm{s}}); G^{\mathrm{u}},G^{\mathrm{s}}).
\end{equation} 
 Then, for all $(G^{\mathrm{u}},G^{\mathrm{s}})\in\mathcal{R}_G$, the angles
\begin{equation}\label{eq:generatingfunct}
\alpha_\pm^{\mathrm{u}}(G^{\mathrm{u}},G^{\mathrm{s}})=\partial_{G^{\mathrm{u}}} \mathtt{S}_\pm (G^{\mathrm{u}},G^{\mathrm{s}}) \qquad\qquad \alpha_\pm^{\mathrm{s}}(G^{\mathrm{u}},G^{\mathrm{s}})=-\partial_{G^{\mathrm{s}}} \mathtt{S}_\pm(G^{\mathrm{u}},G^{\mathrm{s}})
\end{equation}
satisfy 
\[
\mathbb{P}_\pm (\alpha^{\mathrm{u}}_\pm(G^{\mathrm{u}},G^{\mathrm{s}}),G^{\mathrm{u}})= (\alpha^{\mathrm{s}}_\pm(G^{\mathrm{u}},G^{\mathrm{s}}),G^{\mathrm{s}}).
\]
Namely, $\mathtt{S}_\pm$ is a generating function for the scattering map $\mathbb{P}_\pm$ defined in \eqref{eq:defnscatteringmaps}.
\end{thm}

The rather slow decay of parabolic motions and the fact that the parametrizations \eqref{eq:stableparam} and \eqref{eq:unstableparam} are not defined at $u=0$ introduce certain technicalities in the  proof of Theorem \ref{prop:generatingfunctscattmap}. For this reason, the proof is deferred to Section \ref{sec:proofgeneratingscatt}.

\subsubsection{The difference and transversality between the scattering maps}\label{sec:transversalityscattmaps}

We finally provide an asymptotic formula for the difference  $\mathbb{P}_+-\mathbb{P}_-$ and verify \eqref{eq:secondmoeckelcondscatt} in Theorem \ref{thm:Moeckel}. The proof relies on proving an asymptotic formula for the difference between the scattering maps (see Theorem \ref{thm:mainthmscattmaps2}) and the verification of certain non-degeneracy condition of the leading order terms in this asymptotic formula (see Lemma \ref{lem:nondegdiffmelnikov}).

With the intention of clarifying the statement of Theorem \ref{thm:mainthmscattmaps2}, the recalling of some notation is in order. Let  $\Phi_\pm$ be the maps defined in \eqref{eq:definitionPhi+-}, let $\Omega_\pm^{\mathrm{u}}$ be the wave maps defined in \eqref{eq:restrwavemaps},  denote by $\Xi_\pm$ be the maps
\begin{equation}\label{eq:mapsXi+-}
(G^{\mathrm{u}}, G^{\mathrm{s}})\mapsto \Xi_\pm(G^{\mathrm{u}}, G^{\mathrm{s}})= (\hat{\theta}_\pm(G^{\mathrm{u}},G^{\mathrm{s}}),G^{\mathrm{u}})
\end{equation}
obtained  in  Proposition  \ref{prop:sndthmcriticalpoints},  let $\mathtt{S}_\pm$  be the generating functions obtained in Proposition \ref{prop:generatingfunctscattmap}  and consider the function $\tilde{G}^{\mathrm{s}}_\pm(\theta,G^{\mathrm{u}})$ obtained in Theorem \ref{thm:mainthmcriticalpoints}. Define also the vertical strip
\begin{equation}\label{eq:verticalstrip}
\mathbb{I}= \{\pi/8\leq \alpha^{\mathrm{u}}\leq \pi/4\}\times [G_*,G_*+R].
\end{equation}

\begin{thm}\label{thm:mainthmscattmaps2}
The restriction  $\mathbb{P}_\pm|_{\mathbb{I}}:\mathbb{I}\longrightarrow \mathcal{P}_\infty^*$ of the scattering maps $\mathbb{P}_\pm$ to $\mathbb{I}$ can be computed as
\[
(\alpha^{\mathrm{u}},G^{\mathrm{u}}) \longmapsto (\alpha^{\mathrm{u}}- (\partial_{G^{\mathrm{u}}} \mathtt{S}_\pm+ \partial_{G^{\mathrm{s}}} \mathtt{S}_\pm)\circ(\Omega_\pm^{\mathrm{u}}\circ \Phi_{\mathrm{h}}\circ\Phi_\pm\circ \Xi_\pm)^{-1}(\alpha^{\mathrm{u}},G^{\mathrm{u}}),\ \ \tilde{G}^{\mathrm{s}}_\pm \circ (\Omega_\pm^{\mathrm{u}}\circ \Phi_{\mathrm{h}}\circ \Phi_\pm)^{-1}(\alpha^{\mathrm{u}},G^{\mathrm{u}})).
\]
Moreover, for all $(\alpha,G)\in\mathbb{I}$,
\begin{equation}\label{eq:estimatesdifferencescattmaps}
 \mathbb{P}_+-\mathbb{P}_- =\mathcal{J} \nabla (\mathcal{L}_+-\mathcal{L}_-) + \exp(-G^3/3)\left(\mathcal{O}( G^{-1/2}),\ \mathcal{O}(\zeta\ G^{-5/2} ) \right),
\end{equation}
where $\mathcal L_\pm$ have been defined in \eqref{eq: DefinitionredMelnikovPotential}.
\end{thm}
The proof of Theorem is postponed until Section \ref{sec:asymptformscattmap}.
\begin{rem}
Notice that to state Theorem \ref{thm:mainthmscattmaps2}, we have considered the vertical strip $\mathbb{I}\subset\mathbb{A}$. This is due to the fact that the maps 
\[
(G^{\mathrm{u}},G^{\mathrm{s}})\to \Omega_\pm^{\mathrm{u}}\circ \Xi_\pm (G^{\mathrm{u}},G^{\mathrm{s}})= (\alpha^{\mathrm{u}}_\pm(G^{\mathrm{u}},G^{\mathrm{s}}), G^{\mathrm{u}}) 
\]
are not invertible everywhere on $\mathbb{A}$. However, it is easy to check from Theorem \ref{thm:mainthmscattmaps} that $\mathbb{I}\subset \mathrm{Dom}\ ( \Omega_\pm^{\mathrm{u}}\circ \Xi_\pm)^{-1}$. This will be enough for our purposes.
\end{rem}

\begin{rem}\label{rem:verticalstriprestriction}
We point out that \eqref{eq:estimatesdifferencescattmaps} does not mean that $\mathbb{P}_\pm$ are approximated by $\mathcal{L}_\pm$ up to an exponentially small remainder. This is a subtle point in our argument: as we saw in Theorem \ref{thm:mainthmscattmaps} (see \eqref{eq:estimatesmaintermscattmap}), there are non-exponentially small, i.e. polynomially small, errors in the approximation of $\mathbb{P}_\pm$ by $\mathcal{L}_\pm$. What we prove in Theorem \ref{thm:mainthmscattmaps2} is that these errors \emph{are the same for both approximations} of $\mathbb{P}_+$ and $\mathbb{P}_-$. 
\end{rem}

Using the estimates in Theorem \ref{thm:mainthmscattmaps}, together with the asymptotic expression \eqref{eq:estimatesdifferencescattmaps}, for all $(\alpha,G)\in \mathbb{I}$ we have
\[
\begin{split}
\langle \mathcal{J} (\mathbb{P}_+- \mathrm{Id}), (\mathbb{P}_--\mathbb{P}_+) \rangle=& ( \partial_\alpha \mathcal{L}_+ +\mathcal{O}(\zeta G^{-7})) (\partial_G (\mathcal{L}_+-\mathcal{L}_-)+ \mathcal{O}( G^{1/2} \exp(-G^3/3)))\\
&+( \partial_G \mathcal{L}_+ +\mathcal{O}(G^{-7}))(-\partial_\alpha (\mathcal{L}_+-\mathcal{L}_-)+ \mathcal{O}(\zeta G^{-3/2}\exp(-G^3/3)))\\
=&\{ \mathcal{L}_+,\mathcal{L}_- \}+ \mathcal{O}(\zeta G^{-9/2}\exp(-G^3/3)).
\end{split}
\]
In view of the asymptotic expansion above, the proof of Theorem \ref{thm:Moeckel}  follows after checking the following non-degeneracy condition between the reduced Melnikov potentials, proved in  Appendix \ref{sec:Melnikov}. 
\begin{lem}\label{lem:nondegdiffmelnikov}
Let $\mathcal{L}_\pm$ be the reduced Melnikov potentials defined in \eqref{eq: DefinitionredMelnikovPotential}.Then, for all $\alpha\in\mathbb{T}$ and  all $G\in\mathbb{R}$ sufficiently large, their Poisson bracket $\{ \mathcal L_+,\mathcal L_-\}$ admits the asymptotic expansion
\[
\{ \mathcal L_+,\mathcal L_-\}(\alpha,G)= -\mu^2(1-\mu)^2 9 (2\pi)^{3/2} \zeta \ G^{-5/2} \exp(-G^3/3)\ (\sin\alpha+\mathcal{O}(G^{-1})),
\]
where $\mu\in (0,1/2)$ is the mass ratio between the primaries and $\zeta\in (0,1)$ is the eccentricity of the primaries orbit.
\end{lem}
Lemma \ref{lem:nondegdiffmelnikov} implies that $\eta(\varepsilon)$ (which was defined in \eqref{eq:secondcondmapmoeckel} in Theorem \ref{thm:scattmapmain} and is related to the transversality between the level curves of the Hamiltonians which approximate the scattering maps)  satisfies the estimate \eqref{eq:secondmoeckelcondscatt}. This estimate together with the choices in \eqref{eq:firstmoeckelcondscatt} complete the proof of Theorem \ref{thm:Moeckel}.


\begin{figure}\label{fig:transversalityinvariantcurves}
\centering
\includegraphics[scale=0.9]{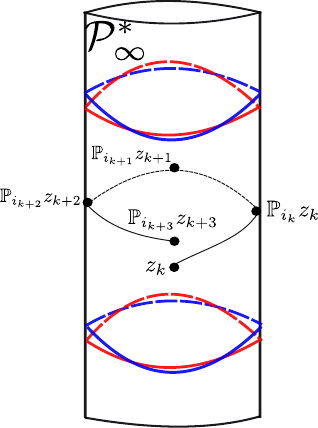}
\caption{The invariant curves of the map $\mathbb{P}_+$ (in red) intersect transversally the invariant curves of the map $\mathbb{P}_-$ (in blue). We also sketch a segment of a diffusive orbit for the iterated function system generated by the maps $\mathbb{P}_\pm$.}
\end{figure}


\section{The generating functions of the invariant manifolds}\label{sec:generatingfunctsinvmanifolds}

In this section we provide the proof of  Propositions \ref{prop:extensiondecaygenfunctions} and \ref{prop:approxgenfunctionbyMelnikov} and show how the latter readily implies Theorem \ref{thm:mainthmcriticalpoints}.  First we explain the change of variables $\Phi_{\mathrm{h}}$ in \eqref{eq:Mathieutransf} and show the existence of real analytic solutions $T^{\mathrm{u,s}}(u,\beta,t;G^{\mathrm{u}},G^{\mathrm{s}},\zeta)$ to the Hamilton-Jacobi equation \eqref{eq:HJfinal} with asymptotic initial conditions given by \eqref{eq:bcHJ} on certain complex domains of the form $D^{\mathrm{u},\mathrm{s}}\times \mathbb{T}^2$ (see Remark \ref{eq:remarkdomainsdef} for a precise definition of these domains), which satisfy
\[
D^{\mathrm{u}}\cap D^{\mathrm{s}}\neq \emptyset,\qquad\qquad \text{and}\qquad\qquad  \left( (-\infty,-u_0]\cup [u_1,u_2]\right) \times \mathbb{T}^2 \subset D^{\mathrm{u}}\times\mathbb{T}^2,\qquad [u_0,\infty)\times \mathbb{T}^2 \subset D^{\mathrm{s}}\times \mathbb{T}^2
\]
for some real values $u_0<u_1<u_2$.  This is the content of Sections \ref{sec:continuationunstable} and \ref{sec:Extensionflow}. Then, in Section \ref{sec:differencegenfunctionsmanifolds} we study the difference 
$\Delta S$ introduced in \eqref{eq:diffgenfunct}
on the complex domain $(D^{\mathrm{u}}\cap D^{\mathrm{s}})\times\mathbb{T}^2$ and show that $\Delta S$ is approximated uniformly in $(D^{\mathrm{u}}\cap D^{\mathrm{s}})\times\mathbb{T}^2$ by
\[
\Delta S\sim \langle \boldsymbol\delta^{\mathrm{u}}-\boldsymbol\delta ^{\mathrm{s}},z\rangle+ \tilde{L}
\]
where $\tilde{L}$ is the \textit{Melnikov potential} defined by
\begin{equation}\label{eq:defnMelnikovPot}
\tilde{L}(u,\beta,t; G_{0},\zeta)=\int_{\mathbb{R}} V(s, \beta, t-G_{0}^3u +G_{0}^3 s;G_{0},\zeta) \mathrm{d}s.
\end{equation}

\begin{rem}
The function $\tilde{L}$ satisfies 
\[
\tilde{L}(u,\beta,t; G_{0},\zeta)=L(t-G_{0}^3 u,\beta;G_{0},\zeta)
\]
where $L(\sigma,\beta;G_{0},\zeta)$ was defined in \eqref{eq:defnMelnikovPotential}. The introduction of \eqref{eq:defnMelnikovPot} is just a matter of convenience for the forthcoming sections.
\end{rem}

Finally, we prove that the existence of nondegenerate critical points of the function 
\[
z\mapsto  \langle \boldsymbol\delta^{\mathrm{u}}-\boldsymbol\delta ^{\mathrm{s}},z\rangle+ \tilde{L}(z;G_0,\zeta)
\]
implies the existence of nondegenerate critical points of the function $z\mapsto \Delta S$.

\subsection{From the circular to the elliptic problem}

As pointed out in the introduction, for $\zeta=0$ and $\mu>0$, which corresponds to the circular problem (RPC3BP), the system is already non-integrable  since there exist transverse intersections between the stable and unstable manifolds of all the tori $\mathcal{T}_G\subset \mathcal{P}_\infty$ with $G$ sufficiently large (see \cite{MR3455155}). However, for $\zeta=0$, due to the conservation of the Jacobi constant, there do not exist heteroclinic connections between different tori $\mathcal{T}_G,\mathcal{T}_{G'}\subset \mathcal{P}_\infty$. In Theorem \ref{thm:mainthmcriticalpoints} we prove that for $\zeta>0$ there do exist heteroclinic connections between sufficiently close $\mathcal{T}_G,\mathcal{T}_{G'}\subset\mathcal{P}_\infty$. As explained at the beginning of  Section \ref{sec:generatingfunctsinvmanifolds} this result will be proved by approximating the difference $\Delta S$ by the Melnikov potential $L$. In this approximation there are  errors coming from the circular part of the perturbation and errors exclusive of the elliptic part. For this reason, in order to obtain asymptotic formulas for the scattering maps associated to the aforementioned heteroclinic intersections, in the case $\mu,\zeta>0$, it is necessary to keep track of the $\zeta$ dependent part in the generating functions $T^{\mathrm{u,s}}$. To that end, we denote by (see \cite{MR3455155})
\begin{equation}\label{eq:generatingcircular}
T_{\mathrm{circ}}^{\mathrm{u,s}}(u,t-\beta;G_{0})=T^{\mathrm{u,s}}(u,\beta,t;G_{0},G_{0},0),
\end{equation}
the generating functions associated to the invariant manifolds of the invariant torus $\mathcal{T}_{G_{0}}\subset \mathcal{P}_\infty$
for the circular problem ($\zeta=0$), let 
\begin{equation}\label{eq:potentialcircular}
V_{\mathrm{circ}}(u,t-\beta;G_{0})=V(u,\beta,t;G_{0},0)
\end{equation}
and introduce the Melnikov potential associated to the circular problem
\begin{equation}\label{eq:Melnikovcircular}
\tilde{L}_{\mathrm{circ}}(u,t-\beta;G_{0})=\tilde{L}(u,\beta,t;G_{0},0).
\end{equation}

\subsection{Analytic continuation of the unstable generating function}\label{sec:continuationunstable}

We define the change of variables $\Phi_{\mathrm{h}}:(u,\beta,t,Y,J,E)\mapsto(\tilde r,\tilde\alpha,t,\tilde y,\tilde G,\tilde E)$ given by 
\[
\tilde r=r_{\mathrm{h}}(u),\qquad \tilde\alpha=\beta+\beta_{\mathrm{h}}(u),\qquad\tilde y=y_{\mathrm{h}}+y^{-1}_{\mathrm{h}}(u)\left(Y-r_{\mathrm{h}}^{-2}J\right),\qquad\tilde G=1+J,
\]
which is the composition of the Mathieu transformation associated to the change $\phi_{\mathrm{h}}$ on the base defined in \eqref{eq:changebasis} and a translation in the $\tilde G$ coordinate. It is now an straightforward computation to check that the Hamiltonian $H=\widetilde H\circ\Phi$ satisfies 
\begin{equation}\label{eq:FinaHamiltonian}
    H(u,\beta,t,Y,J,E)=Y+G_0^3 E+\frac{1}{2y_{\mathrm{h}}^2}(Y-r_{\mathrm{h}}^{-2}J)Y-\frac{1}{2y_{\mathrm{h}}^2r_{\mathrm{h}}^2}(Y-2r_{\mathrm{h}}^{-1}J)J
\end{equation}
Consider now the domain (see Figure \ref{fig:DominisOuter2} and Remark \ref{rem:dominios})
\begin{equation}\label{eq:dominios1}
D^{\mathrm{u}}_\kappa=\{ u\in \mathbb{C}\colon |\operatorname{Im} u|\leq 1/3- \kappa G_{0}^{-3}-\tan\beta_1 \operatorname{Re}u,\ |\operatorname{Im} u|\geq 1/6+\kappa G_{0}^{-3}-\tan\beta_2 \operatorname{Re}u \},
\end{equation}
where $\beta_1,\beta_2\in (0,\pi/2)$, $\beta_1<\beta_2$ and $\kappa>0$ is a given constant. It is clear that for $G_{0}$ large enough $D^{\mathrm{u}}_\kappa$ is non-empty. The role of the parameter $\kappa$ is to  shrink the domain $D^{\mathrm{u}}_\kappa$ when, in Sections \ref{sec:Extensionflow} and  \ref{sec:differencegenfunctionsmanifolds},  we introduce close to identity changes of variables and make use of Cauchy estimates.

 In this section we prove the existence of positive constants $\kappa, \rho$ and $\sigma$  such that for any pair $G^{\mathrm{u}},G^{\mathrm{s}}\in \mathbb{G}_{\rho}(G_0)$, where $\mathbb{G}_{\rho}(G_0)$ was introduced in \eqref{eq:Gintervaltori}, there exists a unique real analytic solution to the Hamilton-Jacobi equation \eqref{eq:HJfinal}, which, in view of the expression \eqref{eq:FinaHamiltonian} for $H$, reads
\begin{equation}\label{eq:firstHJeq}
H(z, \boldsymbol\delta^{\mathrm{u}}+\mathrm{d}T^{\mathrm{u}})=(1+A^{\mathrm{u}}(z,T^{\mathrm{u}}))\partial_u T^{\mathrm{u}}+ B^{\mathrm{u}}(z,T^{\mathrm{u}})\partial_\beta T^{\mathrm{u}}+ G_{0}^3 \partial_t T^{\mathrm{u}}+(\delta^{\mathrm{u}})^2y_{\mathrm{h}}^{-2}r_{\mathrm{h}}^{-3}-V(z)=0
\end{equation}
with asymptotic condition $\lim_{\operatorname{Re}u\to -\infty} T^{\mathrm{u}}=0$ in the complex domain $(u,\beta,t)\in D^{\mathrm{u}}_\kappa\times\mathbb{T}_\rho\times\mathbb{T}_\sigma$ and where
\begin{equation}\label{eq:definitionAusBus}
A^{\mathrm{u}}= \frac{1}{2y_{\mathrm{h}}^{2}}  \left(\partial_u T^{\mathrm{u}}- r_{\mathrm{h}}^{-2}(2\delta^{\mathrm{u}}+\partial_\beta T^{\mathrm{u}})\right) \qquad\qquad    B^{\mathrm{u}}=- \frac{1}{2y_{\mathrm{h}}^{2}r_{\mathrm{h}}^{2}}\big(\partial_u T^{\mathrm{u}}-2r_{\mathrm{h}}^{-1}(2\delta^{\mathrm{u}}+ \partial_\beta T^{\mathrm{u}})\big),
\end{equation}
The existence of $T^{\mathrm{s}}$ solving the corresponding Hamilton-Jacobi equation on $D^{\mathrm{s}}\times\mathbb{T}_\rho\times\mathbb{T}_\sigma$ with $D^{\mathrm{s}}_\kappa=\{u\in\mathbb{C}\colon -u\in D^{\mathrm{u}}_\kappa\}$ is obtained by a completely analogous argument.

\begin{figure}\label{fig:DominisOuter2}
\centering
\includegraphics[scale=0.50]{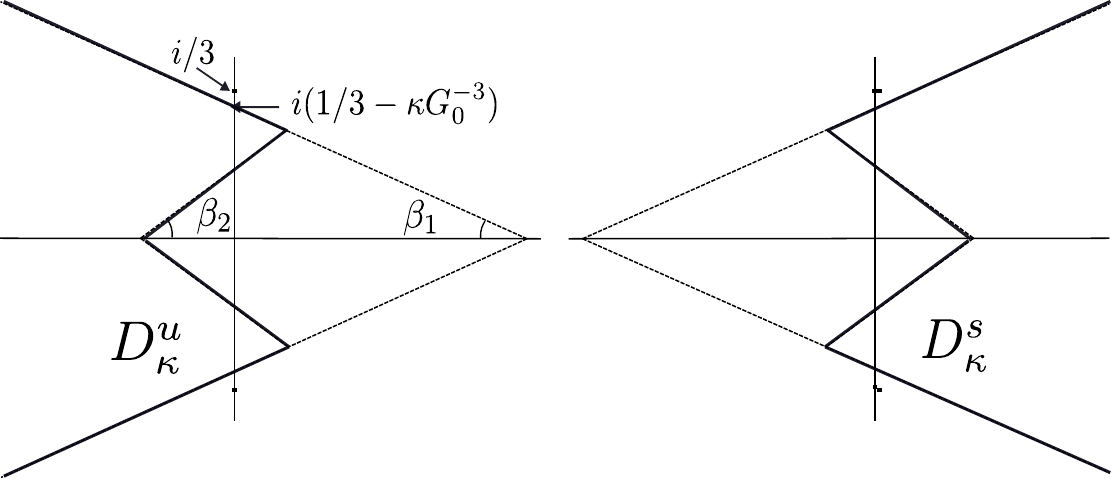}
\caption{The complex domains $D^{\mathrm{u}}_\kappa$  and $D^{\mathrm{s}}_\kappa$.}
\end{figure}

\begin{rem}
The use of different widths for the strips in the angles $\beta$ and $t$ is only a technical issue. The solution to the Hamilton-Jacobi equation \eqref{eq:firstHJeq} will be obtained by means of a Newton method in which the size of the strip $\mathbb{T}_\rho$ for the angle $\beta$ is reduced at each iteration while  the size of the strip $\mathbb{T}_\sigma$ for the $t$ variable can be kept constant. For this reason, in all the forthcoming notation, we omit the dependence on $\sigma$ and only emphasize the dependence on $\rho$.

The width of the strip of analyticity in the angle $\beta$ is taken to be the same than the width of the complex neighborhood for the parameter $G_{0}$. This is an arbitrary choice to avoid introducing more notation.
\end{rem}\vspace{0.3cm}

Let $\eta,\nu$ be positive real constants. We now introduce the family of Banach spaces of sequences of  analytic functions in which we will look for solutions to \eqref{eq:firstHJeq}
\begin{equation}
\mathcal{Z}_{\eta,\nu,\rho}=\left\{ h=\{h^{[l]}\}_{l\in\mathbb{Z}}\colon\   h^{[l]}: D^{\mathrm{u}}_\kappa \times \mathbb{T}_\rho\to \mathbb{C}\  \text{is analytic for all } l\in \mathbb{Z}\ \text{and}\ \lVert h\rVert_{\eta,\nu,\rho}<\infty\right\}
\end{equation}
where $\lVert \cdot \rVert_{\eta,\nu,\rho}$ is the Fourier sup norm
\[
\lVert h \rVert_{\eta,\nu,\rho}=\sum_{l\in\mathbb{Z}} \lVert h^{[l]} \rVert_{\eta,\nu,\rho,l} e^{|l| \sigma}
\]
defined by 
\[
\begin{split}
 \lVert h \rVert_{\eta,\nu,\rho,l}=&\sup_{(u,\beta)\in (D^{\mathrm{u}}_\kappa\cap\{\operatorname{Re}(u)\leq -1\})\times \mathbb{T}_\rho} \left| u^\eta  h^{[l]} (u,\beta)\right|\\
&+\sup_{(u,\beta)\in (D^{\mathrm{u}}_\kappa\cap\{\operatorname{Re}(u)\geq -2\})\times \mathbb{T}_\rho} \left| (u-i/3)^{\nu+l/2} (u+i/3)^{\nu-l/2} h^{[l]} (u,\beta)\right|.
\end{split}
\]
It will also be convenient for us to introduce the Banach spaces
\begin{equation}\label{eq:unstableBanachspace}
\mathcal{X}_{\eta,\nu,\rho}=\left\{ h=\{h^{[l]}\}_{l\in\mathbb{Z}}\ \colon\  h^{[l]}: D^{\mathrm{u}}_\kappa \times \mathbb{T}_\rho\to \mathbb{C}\  \text{is analytic for all } l\in \mathbb{Z}\ \text{and}\  \llbracket h\rrbracket_{\eta,\nu,\rho}<\infty\right\}.
\end{equation}
where 
\begin{equation}\label{eq:bracketnorm}
\llbracket h \rrbracket_{\eta,\nu,\rho}=\lVert  h \rVert_{\eta,\nu,\rho}+\lVert  \partial_u h \rVert_{\eta+1,\nu+1,\rho}.
\end{equation}

\begin{rem}\label{rem:formalseries}
It is straightforward to check that the elements of $\mathcal{Z}_{\eta,\nu,\rho}$ and $\mathcal{X}_{\eta,\nu,\rho}$ can be identified with Fourier series
\[
h(u,\beta,t)=\sum_{l\in\mathbb{Z}} h^{[l]}(u,\beta) e^{ilt}
\]
which, for a given $u\in D^{\mathrm{u}}$, converge on the strip
\[
\mathbb{T}_\sigma (u)=\left\{ t\in \mathbb{C}/ 2\pi\mathbb{Z}\colon \left|\operatorname{Im}(t)-\frac{1}{2} \left(\ln( |u-i/3|)- \ln(|u+i/3|)\right)\right|\leq \sigma\right\}.
\]
That is, they yield well defined functions for $(u,\beta)\in D^{\mathrm{u}}_\kappa \times \mathbb{T}_\rho$ and $t\in \mathbb{T}_\sigma (u)$. Alternatively one can think of the elements of $\mathcal{X}_{\eta,\nu,\rho}$ as formal Fourier series on the strip $\mathbb{T}_\sigma$ (see \cite{MR3455155} and  \cite{https://doi.org/10.48550/arxiv.2207.14351}). 
 \end{rem}

\begin{rem}\label{rem:realanaliticitybanachspace}
In the case $G^{\mathrm{u}},G^{\mathrm{s}}\in\mathbb{R}$, one can replace analytic by real analytic in the definition of $\mathcal{Z}_{\eta,\nu,\rho}$ and $\mathcal{X}_{\eta,\nu,\rho}$.
\end{rem}\vspace{0.3cm}

In the following lemma we list some properties of the spaces $\mathcal{X}_{\eta,\nu,\rho}$ which will be useful. The proof is straightforward.

\begin{lem}\label{lem:techlemmaspaces}
The following statements hold:
\begin{itemize}
\item (Graded algebra property) For any $g\in \mathcal{X}_{\eta,\nu,\rho}$ and $f\in\mathcal{X}_{\eta',\nu',\rho}$, their product satisfies $gf\in\mathcal{X}_{\eta+\eta',\nu+\nu',\rho}$.
\item Let $h\in \mathcal{X}_{\eta,\nu,\rho}$. Then, for $\eta'<\eta$ and $\nu'<\nu$ we have $h\in\mathcal{X}_{\eta',\nu',\rho}$ and
\[
\lVert h\rVert_{\eta',\nu',\rho} \leq  G_{0}^{3(\nu-\nu')} \lVert h\rVert_{\eta,\nu,\rho}
\]
\item Let $h\in \mathcal{X}_{\eta,\nu,\rho}$. Then for any $0<\delta<\rho$ we have that $\partial_\beta h\in \mathcal{X}_{\eta,\nu,\rho-\delta}$ and
\[
\lVert \partial_\beta h\rVert_{\eta,\nu,\rho-\delta} \leq \delta ^{-1} \lVert h\rVert_{\eta,\nu,\rho}.
\]
\end{itemize}
\end{lem}

We also state the following  lemma, which will be useful to deal with compositions in the angular variable $\beta$. The proof can be found in \cite{MR3455155}.

\begin{lem}\label{lem:compositionangularvariabletechnical}
Let $h\in \mathcal{X}_{\eta,\nu,\rho}$ and let $g_i\in\mathcal{X}_{0,0,\rho'}$ with $\rho>\rho'$, $i=1,2,$ and 
\[
\lVert g_i\rVert_{0,0,\rho'}\leq \frac{\rho-\rho'}{2}.
\]
 Write $h\circ(\mathrm{Id}+g_i)(u,\beta,t)=h(u,\beta+g_i(u,\beta,t),t)$. Then, $h\circ(\mathrm{Id}+g_i)\in\mathcal{X}_{\eta,\nu,\rho'}$ with 
\[
\lVert h\circ (\mathrm{Id}+g_i) \rVert_{\eta,\nu,\rho'}\lesssim \lVert h \rVert_{\eta,\nu,\rho}.
\]
Moreover, for $f= h\circ (\mathrm{Id}+g_2)- h\circ (\mathrm{Id}+g_1)$ we have 
\[
\lVert f \rVert_{\eta,\nu,\rho'}\lesssim (\rho-\rho')^{-1} \lVert h \rVert_{\eta,\nu,\rho} \lVert g_2-g_1\rVert_{0,0,\rho'}.
\]
\end{lem}

The choice of the functional space for solving \eqref{eq:firstHJeq} is motivated by the following result proved in  Appendix \ref{sec:Melnikov}.
\begin{lem}\label{lem:boundspotential}
Fix $\kappa>0$ and $\sigma>0$. Then, there exist $\rho_0>0$ such that, for $G_0$ large enough and $0\leq \zeta\leq  G_0^{-3}$ the perturbative potential $V(v,\beta,t;G_{0})$ defined in \eqref{eq:dfnperturbativepotential} satisfies $V\in \mathcal{X}_{2,3/2,\rho_0}$. Moreover
\[
\lVert V \rVert_{2,3/2,\rho_0}\lesssim G_{0}^{-4}
\]
and
\[
\lVert V-V_{\mathrm{circ}} \rVert_{2,3/2,\rho_0}\lesssim \zeta G_{0}^{-4}
\]
where $V_{\mathrm{circ}}$ was defined in \eqref{eq:potentialcircular}.
\end{lem}

We now state the main result in this section.

\begin{thm}\label{thm:stableparam}
Let $\kappa,\sigma>0$ and $\rho_0>0$ as in Lemma \ref{lem:boundspotential}. Then, there exist $\rho\in (0,\rho_0)$ such that, for $G_0$ large enough, $0\leq \zeta\leq G_0^{-3}$, and any pair $G^{\mathrm{u}},G^{\mathrm{s}}\in\mathbb{G}_
\rho(G_0)$,   there exists  $T^{\mathrm{u}}\in \mathcal{X}_{1/3,1/2,\rho}$ solution to the Hamilton-Jacobi equation \eqref{eq:firstHJeq} such that 
\[
\lVert T^{\mathrm{u}} \rVert_{1/3,1/2,\rho}\lesssim G_{0}^{-4} \qquad\qquad\text{and}\qquad\qquad \lVert T^{\mathrm{u}}- L^{\mathrm{u}}\rVert_{1/3,1,\rho}\lesssim G_{0}^{-8},
\]
where $L^{\mathrm{u}}$ is the unstable half Melnikov potential
\begin{equation}\label{eq:halfstableMelnikov}
L^{\mathrm{u}} (u,\beta,t; G_{0},\zeta)= \int_{-\infty}^0 V(u+s,\beta, t+G_{0}^3s ; G_{0},\zeta) \mathrm{d}s.
\end{equation}
Moreover, 
\[
\lVert T^{\mathrm{u}}-T^{\mathrm{u}}_{\mathrm{circ}}-(L^{\mathrm{u}}-L^{\mathrm{u}}_{\mathrm{circ}})\rVert_{1/3,1,\rho} \lesssim \zeta G_{0}^{-8},
\]
where $T^{\mathrm{u}}_{\mathrm{circ}}$ is defined in \eqref{eq:generatingcircular} and $L^{\mathrm{u}}_{\mathrm{circ}}(u,t-\beta;G_{0})=L^{\mathrm{u}} (u,\beta,t; G_{0},0)$.
\end{thm}

The proof of Theorem \ref{thm:stableparam} will be accomplished by a Newton iterative scheme. That is, we obtain $T^{\mathrm{u}}$ as the limit of an iterative process $T^{\mathrm{u}}=\lim_{n\to \infty} T_n$ where $T_0=0$ and the $n$-th step is obtained as the solution to the linear equation
\begin{equation}\label{eq:linearizedHJ}
 H(T_{n-1})+DH(T_ {n-1})[T_{n}-T_{n-1}]=0
\end{equation}
where, by abuse of notation we have written (and will write in the forthcoming sections)
\[
H(T_n)=H(z, \boldsymbol\delta^{\mathrm{u}}+\mathrm{d} T_n(z))
\]
to refer to \eqref{eq:HJfinal}. One can check that the linearized operator $DH(T)[\cdot]$ reads
\begin{equation}\label{eq:linearoperator}
\begin{split}
DH(T)[\cdot]= &\left(1+ y_{\mathrm{h}}^{-2}\left(\partial_u T-r_{\mathrm{h}}^{-2}(\delta^{\mathrm{u}}+\partial_\beta T)\right)\right) \partial_u [\cdot] \\
&-r_{\mathrm{h}}^{-2}y_{\mathrm{h}}^{-2}\left(\partial_u T-2r_{\mathrm{h}}^{-1}(\delta^{\mathrm{u}}+\partial_\beta T)\right) \partial_\beta [\cdot] +G_{0}^{3} \partial_t [\cdot],
\end{split}
\end{equation}
where we recall that  
\[
\delta^{\mathrm{u}}(G^{\mathrm{u}}-G_0)/G_0.
\]
Since $H$ is quadratic in $\nabla T$  the second differential of $H$ is a bilinear operator and the error we accomplish at the step $n$ is 
\begin{equation}\label{eq:errorstepn}
H(T_n)= D^2 H \left[\Delta T_{n}, \Delta T_{n}\right]= y_{\mathrm{h}}^{-2}\left(\left( \partial_u \Delta T_{n} \right)^2 -2r_{\mathrm{h}}^{-2} \partial_u \Delta T_{n}\partial_\beta \Delta T_{n} +  2r_{\mathrm{h}}^{-1} \left(\partial_\beta \Delta T_{n}\right)^2  \right)
\end{equation}
where we have introduced the notation $\Delta T_{n}= T_n-T_{n-1}$.

In the proof of Theorem \ref{thm:stableparam} we treat $DH(T)[\cdot]$ as a small perturbation of the constant coefficients linear operator 
\begin{equation}\label{eq:linoperatorconst}
\mathcal{L}[\cdot]=(\partial_u+G_{0}^3\partial_t)[\cdot].
\end{equation}
The next technical lemma, proved in \cite{MR2726524},  shows the existence of a right inverse for $\mathcal{L}$ on the functional space $\mathcal{X}_{\eta,\nu,\rho}$ with $\eta>1$.

\begin{lem}\label{lem:techlemoperator}
Let $\mathcal{L}$ be the operator defined in \eqref{eq:linoperatorconst}. Then, for any $\eta>1$ there exists an operator $\mathcal{G}:\mathcal{X}_{\eta,\nu,\rho}\to \mathcal{X}_{\eta-1,\nu,\rho}$, given by 
\begin{equation}
\mathcal{G}(h) (u,\beta,t)=\int_{-\infty}^0 h(u+s,\beta,t+G_{0}^3 s) \mathrm{d}s,
\end{equation}
 such that $\mathcal{L}\circ\mathcal{G}=\mathrm{Id}:\mathcal{X}_{\eta,\nu,\rho}\to \mathcal{X}_{\eta,\nu,\rho}$. Moreover, for any $h\in \mathcal{X}_{\eta,\nu,\rho}$ with $\eta,\nu>1$ the following estimates hold
\[
\lVert \mathcal{G}(h) \rVert_{\eta-1,\nu-1,\rho}\lesssim \lVert h\rVert_{\eta,\nu,\rho}\qquad\qquad \text{and}\qquad\qquad \lVert \partial_u \mathcal{G}(h) \rVert_{\eta,\nu,\rho}\lesssim \lVert h\rVert_{\eta,\nu,\rho}.
\]
\end{lem}

\subsubsection{First step of the Newton scheme}\label{sec:firststepiterative}
The iterative scheme proposed above defines the function $T_1$ as the solution to the linearized equation
\begin{equation}\label{eq:firststepnewton}
H(0)+DH(0)[T_1]=0.
\end{equation}
Instead, it will be convenient to modify the first step of the iterative process and define $T_1$ as the solution to 
\begin{equation}\label{eq:firststepnewton2}
\mathcal{L} T_1=-H(0),
\end{equation}
where $\mathcal{L}$ is the constant coefficients linear operator defined in \eqref {eq:linoperatorconst}. Using Lemma \ref{lem:techlemoperator} we can rewrite \eqref{eq:firststepnewton} as 
\begin{equation}\label{eq:T_1}
T_1=-\mathcal{G}(H(0)).
\end{equation}
The properties of the unperturbed homoclinic stated in Lemma  \ref{lem:uperturbedhomoclinic}, Lemma  \ref{lem:boundspotential} for the potential $V$ and the hypothesis $|\delta^{\mathrm{u}}_\beta|\lesssim \zeta G_{0}^{-5}$  imply that (here $\rho_0$ is the constant given in Lemma \ref{lem:boundspotential})
\begin{equation}\label{eq:defeps0}
\lVert H(0)\rVert_{4/3,3/2,\rho_0}= \lVert r_{\mathrm{h}}^{-3}y_{\mathrm{h}}^{-2} (\delta^{\mathrm{u}}_\beta)^2 -V \rVert_{4/3,3/2,\rho_0}\lesssim G_{0}^{-4}\equiv  \varepsilon_0.
\end{equation}
Therefore,  it follows from Lemma \ref{lem:techlemoperator} that $T_1\in \mathcal{X}_{1/3,1/2,\rho_0}$ with
\begin{equation}\label{eq:estimateS0}
\llbracket T_1 \rrbracket_{1/3,1/2,\rho_0} \lesssim \varepsilon_0.
\end{equation}
The error in this first approximation is given by
\[
H(T_1)= D^2 H[T_1,T_1]+ \left( DH(T_1)-\mathcal{L}\right)[T_1].
\] 
Using  Lemma \ref{lem:uperturbedhomoclinic}, Lemma \ref{lem:techlemmaspaces}, and the expressions \eqref{eq:linearoperator} and \eqref{eq:errorstepn}, we obtain that for $0<\delta_0<\rho_0$
\[
\begin{split}
\lVert  D^2 H[T_1,T_1] \rVert_{4/3,3/2,\rho_0-\delta_0}\lesssim G_{0}^{3/2} & \left(\lVert \partial_u T_1 \rVert^2_{4/3,3/2,\rho_0}+  \lVert \partial_u T_1 \rVert_{4/3,3/2,\rho_0} \lVert \partial_\beta T_1\rVert_{1/3,1/2,\rho_0-\delta_0} \right.\\
+&\left.  G_{0}^{-3/2} \lVert \partial_\beta T_1\rVert^2_{4/3,3/2,\rho_0-\delta_0} \right)
\end{split}
\]
and 
\[
\lVert  \left( DH(T_1)-\mathcal{L}\right)[T_1]\rVert_{4/,3/2,\rho_0-\delta_0} \lesssim|\delta^{\mathrm{u}}_\beta| \lVert \partial_\beta T_1 \rVert_{1/3,1/2,\rho_0-\delta_0}
\]
Then, it follows from the estimate \eqref{eq:estimateS0} for $T_1$, the hypothesis $|\delta^{\mathrm{u}}_\beta|\lesssim \zeta G_{0}^{-5}$ and the third and fourth items in Lemma \ref{lem:techlemmaspaces}, that
\[
\lVert H(T_1) \rVert_{4/3,3/2,\rho_0-\delta_0}\lesssim G_{0}^{3/2}\varepsilon_0^2\left(1+ \delta_0^{-1}+ G_{0}^{-3/2} \delta_0^{-2} \right)\lesssim G_{0}^{3/2} \varepsilon_0^2 \delta_0^{-2},
\]
where $\varepsilon_0$ was defined in \eqref{eq:defeps0}. We now take 
\[
\delta_0\equiv \varepsilon_0^{1/16}\qquad\text{and define}\qquad \rho_1= \rho_0-\delta_0.
\]
Therefore,
\[
\lVert H(T_1) \rVert_{4/3,3/2,\rho_1}\lesssim G_{0}^{3/2} \varepsilon_0^{1/4} \varepsilon_0^{3/2}\leq \varepsilon_0^{3/2}\equiv \varepsilon_1.
\]


\subsubsection{The iterative argument}\label{sec:iterativeargument}
Throughout this section, the symbol $a\lesssim b$ means that there exists $C>0$ which does not depend on the step $n$ and $G_{0}$ such that $a\leq Cb$.

Through the Newton iteration scheme, at the step $n+1$, we have to solve the linearized equation
\[
H(T_n)+DH(T_n)[\varDelta T_{n+1}]=0\qquad\qquad \varDelta T_{n+1}=T_{n+1}-T_n.
\]
For that, we have to invert the linear operator $DH(T_n)[\cdot]$ defined in \eqref{eq:linearoperator}. Since this operator has a non-zero coefficient multiplying $\partial_\beta$  we must find a change of variables 
\begin{equation}\label{eq:changeangleconstantcoeff}
(u,\beta,t) =\Psi_{n+1}(u,\varphi,t)\equiv(u,\varphi+\psi_{n+1}(u,\varphi,t),t),
\end{equation}
in which the linearized operator does not involve partial derivatives with respect to the new angle $\varphi$. Define 
\begin{equation}\label{eq:AnBn}
A_n=\frac{1}{2y_{\mathrm{h}}^{2}}   \left(\partial_u T_n-r_{\mathrm{h}}^{-2}(2\delta^{\mathrm{u}}+ \partial_\beta T_n)\right) \qquad\qquad    B_n= -\frac{1}{2y_{\mathrm{h}}^{2}r_{\mathrm{h}}^2}\big( \partial_u T_n-2r_{\mathrm{h}}^{-1} \big(2\delta^{\mathrm{u}}+\partial_\beta T_n\big)\big)
\end{equation}
Then, one can check that, if one considers a change of variables \eqref{eq:changeangleconstantcoeff} with $\psi_{n+1}$ solving (here $\mathcal{L}$ is the  operator \eqref{eq:linoperatorconst})
\[
 \mathcal{L}\psi_{n+1}=B_{n}\circ \Psi_{n+1}-(A_{n}\circ \Psi_{n+1} ) \partial_v \psi_{n+1},
\]
the following equations determining an unknown function $h$ are equivalent
\begin{equation}\label{eq:equivalencelineareqs}
DH(T_{n})[h]+H(T_{n})=0\quad\quad\Longleftrightarrow \quad\quad (1+A_{n}\circ\Psi_{n+1})\partial_u (h\circ\Psi_{n+1})+ G_{0}^3 \partial_t (h\circ\Psi_{n+1}) +H(T_{n})\circ\Psi_{n+1}=0.
\end{equation}
Indeed, if we denote by $\Delta\widetilde{T}_{n+1}$ the solution of the second equation in \eqref{eq:equivalencelineareqs}, then 
\[
\Delta T_{n+1}= \Delta\widetilde{T}_{n+1}\circ \Psi_{n+1}^{-1}
\]
solves the first equation in \eqref{eq:equivalencelineareqs}. As a consequence, we define 
\[
T_{n+1}=T_n+\Delta\widetilde{T}_{n+1}\circ \Psi_{n+1}^{-1}
\]
and
\[
\widetilde{T}_{n+1}=T_{n+1}\circ \Psi_{n+1}=\widetilde{T}_n\circ\Psi_n^{-1}\circ\Psi_{n+1}+\Delta \tilde{T}_{n+1}.
\]
In order to state the inductive hypothesis, we define now the constants 
\begin{equation}
\varepsilon_n= \varepsilon_{n-1}^{3/2}\qquad\qquad \delta_n= \varepsilon_n^{1/16}\qquad\qquad \rho_n= \rho_{n-1}-2\delta_{n-1}.
\end{equation}

Notice that it follows (taking $\varepsilon_0$ small enough) from this definition that $\rho_n\geq\rho_0/2\ \  \forall n\in\mathbb{N}\cup\{\infty\}$. Suppose that:
\begin{itemize}
\item (H1) There exists a family of functions $\{\widetilde{T}_i \}_{1\leq i\leq n}\subset \mathcal{X}_{1/3,1/2,\rho_{i-1}}$ and a family of close to identity maps $\Psi_1=\mathrm{Id}$ and $\left\{\Psi_i\right\}_{2\leq i\leq n}$ with  $\Psi_i=\mathrm{Id}+\psi_i$,  $\psi_i \in \mathcal{X}_{2/3,1/2,\rho_{i-1}}$, such that 
\[
\mathcal{L}\psi_{i+1}=B_{i}\circ \Psi_{i+1}-(A_{i}\circ \Psi_{i+1}) \partial_u \psi_{i+1},
\]
where now $A_i,B_i$ are written in terms of $\tilde{T}_i$ 
\[
\begin{split}
A_i= &\frac{1}{2y_{\mathrm{h}}^{2}} \left(\partial_u (\widetilde{T}_i\circ \Psi_i^{-1} )-r_{\mathrm{h}}^{-2} (2\delta^{\mathrm{u}}+ \partial_\beta (\widetilde{T}_i\circ \Psi_i^{-1})) \right)\\
B_i= -&\frac{1}{2r_{\mathrm{h}}^{2}y_{\mathrm{h}}^{2}}\big(\partial_u (\widetilde{T}_i\circ \Psi_i^{-1} )-2r_{\mathrm{h}}^{-1} (2\delta^{\mathrm{u}}+  \partial_\beta (\widetilde{T}_i\circ \Psi_i^{-1}) )\big).\\
\end{split}
\]
\item (H2) The functions $\psi_i$ satisfy (see Remark \ref{rem:compositioninductivemaps} below)
\[
\llbracket \psi_{i+1}-\psi_i \rrbracket_{2/3,1/2,\rho_{i}} \lesssim  \delta_{i-1}^{15}.
\]
\item (H3) The functions $\widetilde{T}_i$ satisfy
\[
 \llbracket \widetilde{T}_{i+1}-\widetilde{T}_{i}\circ \Psi_{i}^{-1}\circ \Psi_{i+1}\rrbracket_{1/3,1/2,\rho_{i}}\lesssim \varepsilon_{i},
\]
and 
\[
\lVert H(\widetilde{T}_{i+1}\circ \Psi_{i+1}^{-1})\circ \Psi_{i+1} \rVert_{4/3,3/2,\rho_{i+1}}\lesssim \varepsilon_{i+1}.
\]
\end{itemize}

\begin{rem}\label{rem:compositioninductivemaps}
Hypothesis (H2) can be rephrased as 
\[
\llbracket \Psi_{i}^{-1}\circ \Psi_{i+1}-\mathrm{Id} \rrbracket_{2/3,1/2,\rho_i}\lesssim \delta_{i-1}^{15}.
\]
\end{rem}\vspace{0.3cm}

\noindent We claim that, under these hypotheses, there exists a map $\Psi_{n+1}=\mathrm{Id}+\psi_{n+1}$, $\psi_{n+1}\in\mathcal{X}_{2/3,1/2,\rho_{n}}$ solving 
\[
\mathcal{L}\psi_{n+1}=B_{n}\circ \Psi_{n+1}-(A_n\circ \Psi_{n+1}) \partial_u \psi_{n+1}
\]
with 
\[
\llbracket \psi_{n+1}-\psi_n\rrbracket_{2/3,1/2,\rho_n}\lesssim \delta_{n-1}^{15}
\]
and $\widetilde{T}_{n+1}\in \mathcal{X}_{1/3,1/2,\rho_{n}}$ such that 
\[
 \llbracket \widetilde{T}_{n+1}-\widetilde{T}_n\circ \Psi_{n}^{-1}\circ \Psi_{n+1}\rrbracket_{1/3,1/2,\rho_{n}}\lesssim \varepsilon_n 
\]
for which
\[
\lVert H(\widetilde{T}_{n+1}\circ \Psi_{n+1}^{-1}) \circ \Psi_{n+1} \rVert_{4/3,3/2,\rho_{n+1}}\lesssim \varepsilon_{n+1}.
\]
The first step towards the proof of the inductive claim is to look for the change of variables  $\Psi_{n+1}$.

\begin{lem}\label{lem:straighteninglinoperator}
Assume that (H1), (H2) and (H3) hold for all $1\leq i\leq n$ . Then, there exists $\Psi_{n+1}=\mathrm{Id}+\psi_{n+1}$, $\psi_{n+1}\in \mathcal{X}_{2/3,1/2,\rho_{n}}$ such that 
\[
\mathcal{L}\psi_{n+1}=B_n\circ \Psi_{n+1}-(A_n\circ\Psi_{n+1} ) \partial_v \psi_{n+1}
\]
with $\llbracket \Psi_{n+1}-\Psi_n \rrbracket_{2/3,1/2,\rho_{n}}\lesssim  \delta^{15}_{n-1}$.
\end{lem}

\begin{proof}
Throughout the proof we will use the first part of Lemma \ref{lem:compositionangularvariabletechnical}, which deals with compositions in the angular variable, without mentioning. We also define $\tilde{\rho}_{n}=\rho_{n-1}-\delta_{n-1}$ to avoid lengthy notation. Since $\mathcal{L}$ is linear, we can write 
\[
\begin{split}
\mathcal{L}(\psi_{n+1}-\psi_n)=&B_{n}\circ \Psi_{n+1}- B_{n-1}\circ\Psi_n-(A_n\circ\Psi_{n+1} ) \partial_u \psi_{n+1}+(A_{n-1}\circ\Psi_{n} ) \partial_u \psi_{n}\\
=&B_{n}\circ \Psi_{n+1}-B_{n}\circ \Psi_{n}+( B_{n}- B_{n-1})\circ\Psi_n-((A_n-A_{n-1})\circ\Psi_{n+1} ) \partial_u \psi_{n+1}\\
&-(A_{n-1}\circ\Psi_{n+1} -A_{n-1}\circ\Psi_{n}) \partial_u \psi_{n+1}-(A_{n-1}\circ\Psi_n) \partial_u (\psi_{n+1}-\psi_n)
\end{split}
\]
which, by the mean value theorem,  can be rewritten as the fixed point equation 
\[
 \Delta \psi_{n+1} =\mathcal{G} ( F(\Delta \psi_{n+1}))
\]
in a Banach space $\mathcal{X}_{\eta,\nu,\rho}$ for suitable $\eta,\nu,\rho$ to be chosen, where $\Delta \psi_{n+1} =\psi_{n+1}-\psi_n$, $\mathcal{G}$ is the operator introduced  in Lemma \ref{lem:techlemoperator} and 
\[
\begin{split}
F (\Delta \psi_{n+1})=& \Delta \psi_{n+1} \int_{0}^{1} \partial_{\beta} B_{n}\circ\left(\mathrm{Id}+s\Delta \psi_{n+1} \right)\mathrm{d}s+( B_{n}- B_{n-1})\circ\Psi_n\\
&-((A_n-A_{n-1})\circ (\Psi_{n}+\Delta \psi_{n+1}) ) \partial_u ( \psi_n+ \Delta\psi_{n+1}) \\ 
&-\Delta \psi_{n+1} \partial_u (\psi_{n}+ \Delta\psi_{n+1}) \int_0^1 \partial_\beta A_n \circ \left( \mathrm{Id}+s\Delta \psi_{n+1}\right) \mathrm{d}s -(A_{n-1}\circ\Psi_n) \partial_u \Delta \psi_{n+1}.
\end{split}
\]
We obtain $\Delta \psi_{n+1}$ by an standard application of the fixed point theorem for Banach spaces. To that end we first bound the term
\[
F(0)=(B_n-B_{n-1})\circ\Psi_n-(A_n-A_{n-1})\circ \Psi_n \partial_u \psi_{n}.
\]
We observe that 
\[
\begin{split}
\partial_u \left(\widetilde{T}_n\circ \Psi_n^{-1}- \widetilde{T}_{n-1}\circ \Psi_{n-1}^{-1}\right)=& \partial_u \left(\widetilde{T}_n- \widetilde{T}_{n-1}\circ \Psi_{n-1}^{-1}\circ\Psi_n\right)\circ \Psi_n^{-1}\\
&+ \partial_\varphi\left(\widetilde{T}_n- \widetilde{T}_{n-1}\circ \Psi_{n-1}^{-1}\circ\Psi_n\right)\circ \Psi_n^{-1}\ \  \partial_u \Psi_n^{-1}\\
\end{split}
\]
and
\[
\begin{split}
\partial_\beta \left(\widetilde{T}_n\circ \Psi_n^{-1}- \widetilde{T}_{n-1}\circ \Psi_{n-1}^{-1}\right)=& \partial_\varphi \left(\widetilde{T}_n- \widetilde{T}_{n-1}\circ \Psi_{n-1}^{-1}\circ\Psi_n\right)\circ \Psi_n^{-1} \ \ \partial_\beta \Psi_n^{-1}.\\
\end{split}
\]
Therefore, taking into account that (for the case $n=1$ notice that $\psi_1=0$)
\[
\llbracket \psi_n\rrbracket_{2/3,1/2,\rho_{n-1}}\leq \llbracket \psi_2 \rrbracket_{2/3,1/2,\rho_{n-1}}+\sum_{i=3}^{n} \llbracket \psi_i -\psi_{i-1}\rrbracket_{2/3,1/2,\rho_{n-1}}\lesssim \llbracket \psi_2\rrbracket_{2/3,1/2,\rho_{n-1}}\lesssim \delta_0^{15},
\]
 it is easy to show that the inductive hypothesis implies
\[
\begin{split}
\left \lVert y_{\mathrm{h}}^{-2}r_{\mathrm{h}}^{-2}  \partial_u \left(\widetilde{T}_n\circ \Psi_n^{-1}- \widetilde{T}_{n-1}\circ \Psi_{n-1}^{-1}\right) \circ\Psi_n \right\rVert_{5/3,3/2, \tilde{\rho}_n}\lesssim & \llbracket \widetilde{T}_{n}-\widetilde{T}_{n-1}\circ \Psi_{n-1}^{-1}\circ \Psi_{n}\rrbracket_{1/3,1/2,\rho_{n-1}}\lesssim \varepsilon_{n-1}\\
\end{split}
\]
and
\[
\begin{split}
\left\lVert y_{\mathrm{h}}^{-2} r_{\mathrm{h}}^{-3} \partial_\beta \left(\widetilde{T}_n\circ \Psi_n^{-1}- \widetilde{T}_{n-1}\circ \Psi_{n-1}^{-1}\right) \circ\Psi_n \right\rVert_{5/3,3/2, \tilde{\rho}_n}\lesssim &\delta_{n-1}^{-1} \llbracket  \widetilde{T}_{n}-\widetilde{T}_{n-1}\circ \Psi_{n-1}^{-1}\circ\Psi_n\rrbracket_{1/3,1/2,\rho_{n-1}}\\
\lesssim & \varepsilon_{n-1}\delta_{n-1}^{-1}.
\end{split}
\]
Thus,  from the definition of $B_n$ in \eqref{eq:AnBn},
\[
\lVert ( B_{n}- B_{n-1})\circ\Psi_n\rVert_{5/3,3/2,\tilde{\rho}_n} \lesssim \varepsilon_{n-1}\delta_{n-1}^{-1}= \delta_{n-1}^{15}.
\]
Taking into account the definition of $A_n$ in \eqref{eq:AnBn}, a similar computation shows that 
\[
\begin{split}
\lVert ( A_{n}-A_{n-1})\circ\Psi_n \partial_u \psi_n \rVert_{5/3,3/2,\tilde{\rho}_n} \lesssim& \lVert ( A_{n}- A_{n-1})\circ\Psi_n \rVert_{2/3,1/2,\tilde{\rho}_n} \lVert \partial_u \psi_n \rVert_{1,1,\rho_n} \\ 
 \lesssim &  G_{0}^{3/2} \lVert ( A_{n}- A_{n-1})\circ\Psi_n \rVert_{2/3,1/2,\tilde{\rho}_n}  \llbracket \psi_n \rrbracket_{2/3,1/2,\rho_n} \\ 
\lesssim & G_{0}^{3/2} \varepsilon_{n-1}\delta_{n-1}^{-1}\llbracket \psi_n \rrbracket_{2/3,1/2,\rho_n}\lesssim G_{0}^{-9/4} \delta_{n-1}^{15}.
\end{split}
\]
Therefore,
\[
\begin{split}
\lVert F(0) \rVert_{5/3,3/2, \tilde{\rho}_{n}}= &\lVert (B_{n}-B_{n-1})\circ\Psi_n -(A_n-A_{n-1})\circ \Psi_n \partial_u \psi_n\rVert_{5/3,3/2,\rho_n}\\
\lesssim & \varepsilon_{n-1}\delta_{n-1}^{-1} =  \delta_{n-1}^{15},
\end{split}
\]
and it follows from Lemma \ref{lem:techlemoperator} that 
\[
\lVert \mathcal{G}(F(0))\rVert_{2/3, 1/2,\tilde{\rho}_n }\lesssim \delta_{n-1}^{15}.
\]
We notice that, since $A_n,B_n$ depend linearly on $T_n$,
\[
\begin{split}
 \lVert \partial_\beta A_{n} \rVert_{2/3,1/2,\tilde{\rho}_n-\delta_{n-1}}\leq & \lVert \partial_\beta A_1\rVert_{2/3,1/2,\tilde{\rho}_1-\delta_0}+ \sum_{i=2}^{n}  \lVert \partial_\beta  (A_{i}-A_{i-1})\rVert_{2/3,1/2,\tilde{\rho}_i-\delta_{i-1}}\\
\leq & \delta_0^{-1} \lVert \partial_\beta A_1\rVert_{2/3,1/2,\tilde{\rho}_1}+  \sum_{i=2}^{n} \delta_i^{-1} \lVert A_i-A_{i-1} \rVert_{2/3,1/2,\tilde{\rho}_i}  \\
\lesssim &\varepsilon_0 \delta_0^{-1}+ \sum_{i=1}^{n-1} \varepsilon_i\delta_i^{-1}  \lesssim \varepsilon_0 \delta_0^{-1}=G_{0}^{-15/4},
\end{split}
\]
and the same computation shows that 
\[
\lVert \partial_\beta B_{n} \rVert_{5/3,3/2,\tilde{\rho}_n-\delta_{n-1}} \lesssim \varepsilon_0 \delta_0^{-1}=G_{0}^{-15/4}.
\]
Take now any $\Delta\psi,\Delta\psi^*\in B(\delta_{n-1}^{15}) \subset \mathcal{X}_{2/3,1/2,\rho_{n}}$. From the fundamental theorem of calculus, it follows that 
\[
\begin{split}
F(\Delta \psi^*)-F(\Delta \psi)= & (\Delta\psi^*-\Delta\psi)\int_{0}^1 \partial_\beta B_n\circ\left(\mathrm{Id}+s(\Delta \psi^*-\Delta\psi) \right)\mathrm{d}s -A_n\circ(\Psi_n+\Delta\psi) \partial_u (\Delta\psi^*-\Delta\psi)\\
&-(\Delta\psi^*-\Delta\psi) \partial_u  (\psi_n+\Delta\psi^*) \int_{0}^1 \partial_\beta A_n\circ\left(\mathrm{Id}+s(\Delta \psi^*-\Delta\psi) \right).
\end{split}
\]
Using the previous estimates, Lemma \ref{lem:techlemmaspaces} and the second part of Lemma \ref{lem:compositionangularvariabletechnical}, we obtain that (recall that $\rho_n= \rho_{n-1}-2\delta_{n-1}=\tilde{\rho}_n-\delta_{n-1}$)
\[
\begin{split}
\lVert  (\Delta\psi^*-\Delta\psi)\int_{0}^1 \partial_\beta B_n\circ\left(\mathrm{Id}+s(\Delta \psi^*-\Delta\psi) \right)\mathrm{d}s\rVert_{5/3,3/2,\rho_n}\lesssim&G_{0}^{3/2} \lVert \partial_\beta B_n \rVert_{5/3,3/2,\rho_n} \lVert \Delta \psi-\Delta \psi^*\rVert_{2/3,1/2,\rho_n}\\
\lesssim &G_{0}^{-9/4}\lVert \Delta \psi-\Delta \psi^*\rVert_{2/3,1/2,\rho_n}.\\
\end{split}
\]
Similar computations show that 
\[
\lVert  F(\psi)-F(\psi^*)  \rVert_{5/3,3/2,\rho_{n}}\\
\lesssim  G_{0}^{-9/4}\lVert \psi-\psi^* \rVert_{2/3,1/2,\rho_{n}}.
\]
Finally, from Lemma \ref{lem:techlemoperator},
\[
\begin{split}
\lVert \mathcal{G}\left( F(\psi)-F(\psi^*) \right) \rVert_{2/3,1/2,\rho_{n}}\lesssim& \lVert  F(\psi)-F(\psi^*)  \rVert_{5/3,3/2,\rho_{n}}\\
\lesssim &  G_{0}^{-9/4}\lVert \psi-\psi^* \rVert_{2/3,1/2,\rho_{n}}.\\
\end{split}
\]
Then, the proof of the lemma follows from a direct application of the fixed point theorem in the ball of radius $C \delta^{15}_{n-1}$ (for some large enough $C$) centered at the origin of the Banach space $\mathcal{X}_{2/3,1/2,\rho_{n}}$.
\end{proof}

We now complete the proof of the inductive claim for $\tilde{T}_{n+1}$.
\begin{prop}
The equation 
\begin{equation}\label{eq:lineareqprooffinal}
(1+A_n\circ\Psi_{n+1})\partial_u (\Delta \widetilde{T}_{n+1})+ G_{0}^3 \partial_t (\Delta \widetilde{T}_{n+1}) +H(T_n)\circ\Psi_{n+1}=0
\end{equation}
admits a unique solution $\Delta \widetilde{T}_{n+1}\in \mathcal{X}_{1/3,1/2,\rho_{n}}$ such that 
\[
\lVert \Delta \widetilde{T}_{n+1} \rVert_{1/3,1/2,\rho_n}\lesssim \varepsilon_n.
\]
Moreover, the function
\[
\widetilde{T}_{n+1}= T_{n}\circ \Psi_n^{-1}\circ \Psi_{n+1} +\Delta \widetilde{T}_{n+1}
\]
satisfies
\[
\lVert H(\widetilde{T}_{n+1}\circ \Psi_{n+1}^{-1})\circ \Psi_{n+1} \rVert_{1/3,1/2,\rho_{n+1}}\lesssim \varepsilon_{n+1}.
\]
\end{prop}

\begin{proof}
Again, throughout the proof we will use the first part of Lemma \ref{lem:compositionangularvariabletechnical}, which deals with compositions in the angular variable, without mentioning. We rewrite \eqref{eq:lineareqprooffinal} as the affine fixed point equation for $\Delta \tilde{T}_{n+1}$
\[
\Delta \widetilde{T}_{n+1}=- \mathcal{G}\left(H(\widetilde{T}_n\circ \Psi_n^{-1})\circ\Psi_{n+1}- (A_n\circ\Psi_{n+1}) \partial_u (\Delta \widetilde{T}_{n+1})\right)
\] 
where $\mathcal{G}$ is the operator  introduced  in Lemma \ref{lem:techlemoperator}. The existence of a fixed point $\Delta \widetilde{T}_{n+1}\in \mathcal{X}_{1/3,1/2,\rho_n}$ with 
\[
\llbracket \Delta \widetilde{T}_{n+1} \rrbracket_{1/3,1/2,\rho_n}\lesssim \varepsilon_n
\]
is easily completed using the properties of $\mathcal{G}$ in Lemma \ref{lem:techlemoperator} and  the estimates
\[
\lVert ( H(\widetilde{T}_n\circ \Psi_{n}^{-1}))\circ \Psi_{n+1} \rVert_{4/3,3/2,\rho_{n}}\lesssim \varepsilon_n\quad\quad \lVert A_n\circ\Psi_{n+1} \rVert_{0,0,\rho_n}\lesssim G_{0}^{3/2}\lVert A_n\circ\Psi_{n+1} \rVert_{2/3,1/2,\rho_n}\lesssim  G_{0}^{-9/4},
\]
which are obtained from the inductive hypothesis after writing
\[
\begin{split}
(H(\widetilde{T}_n\circ \Psi_{n}^{-1}))\circ \Psi_{n+1}=&(H(\widetilde{T}_n\circ \Psi_{n}^{-1}))\circ \Psi_{n}+\left((H(\widetilde{T}_n\circ \Psi_{n}^{-1}))\circ \Psi_{n+1}-(H(\widetilde{T}_n\circ \Psi_{n}^{-1}))\circ \Psi_{n}\right),
\end{split}
\]
and using the estimate for $\llbracket \psi_{n+1}-\psi_n\rrbracket_{1/3,1/2,\rho_{n}}$ given in Lemma \ref{lem:straighteninglinoperator}. In order to prove the estimate for the error, it follows from our construction that 
\[
H(\widetilde{T}_{n+1}\circ\Psi_{n+1}^{-1})= D^{2} H\left[ \Delta \widetilde{T}_{n+1}\circ \Psi_{n+1}^{-1},\Delta 
\widetilde{T}_{n+1}\circ \Psi_{n+1}^{-1} \right].
\]
The proof is completed in a straightforward manner from expression \eqref{eq:errorstepn}, the estimate for $\llbracket \Delta \tilde{T}_{n+1} \rrbracket_{1/3,1/2,\rho_n}$ and the estimate for $\llbracket \psi_{n+1}-\psi_n\rrbracket_{1/3,1/2,\rho_{n}}$ given in Lemma \ref{lem:straighteninglinoperator}.
\end{proof}\vspace{0.3cm}

We can now conclude the proof of Theorem \ref{thm:stableparam}.

\begin{proof}[Proof of Theorem \ref{thm:stableparam}]
Notice that the function $T_1$ obtained in \eqref{eq:T_1}  and the map $\Psi_1=\mathrm{Id}$ satisfy the inductive hypothesis assumed at the beginning of Section \ref{sec:iterativeargument}. Therefore, for all $n\in\mathbb{N}$, we can find maps $\Psi_n=\mathrm{Id}+\psi_n$ with $\psi_{n} \in \mathcal{X}_{1/3,1/2,\rho_{n-1}}$ satisfying 
\[
\llbracket \psi_{n+1}-\psi_n\rrbracket_{1/3,1/2,\rho_{n}}\lesssim \delta_{n-1}^{15}
\]
 and functions  $\widetilde{T}_n\in\mathcal{X}_{1/3,1/2,\rho_{n-1}}$ such that 
\[
\Vert \widetilde{T}_{n+1}-\widetilde{T}_n\circ \Psi_{n}^{-1}\circ \Psi_{n+1} \rVert_{1/3,1/2,\rho_n}\lesssim \varepsilon_n \qquad\qquad\left\lVert( H(\widetilde{T}_n \circ \Psi_n^{-1})) \circ \Psi_{n}\right\rVert_{4/3,3/2,\rho_n}\lesssim \varepsilon_{n+1}=\varepsilon_0^{(3/2)^{n+1}}.
\]
Then, $\Psi_n$ converges uniformly on $\mathcal{X}_{1/3,1/2,\rho_0/4}$  to an analytic change of coordinates
\begin{equation}\label{eq:finalstraightening}
\Psi_\infty =\mathrm{Id}+\psi_\infty\qquad\qquad \llbracket \psi_{\infty}\rrbracket_{1/3,1/2,\rho_{n}}\lesssim \delta_0^{15}=G_{0}^{-15/4}
\end{equation}
and  the sequence $\left\{T_n\right\}_{n\in\mathbb{N}}$ defined by 
\[
T_n=\widetilde{T}_{n}\circ \Psi_n^{-1}
\]
converges uniformly to an analytic function   $T\in\mathcal{X}_{1/3,1/2,\rho_0/4}$ such that
\[
\lVert T \rVert_{1/3,1/2,\rho_0/4}\lesssim \lVert T_1 \rVert_{1/3,1/2,\rho/4} +\sum_{n=1}^{\infty} \lVert T_{n+1}-T_n \rVert_{1/3,1/2,\rho/4}\lesssim \varepsilon_0 
\]
and
\[
\lVert H(T)\rVert_{4/3,3/2,\rho_0/4} =\lim_{n\to\infty}\lVert H(T_n)\rVert _{4/3,3/2,\rho_0/4}=0.
\]
This  proves the existence of a solution $T\in\mathcal{X}_{1/3,1/2,\rho_0/4}$ to the Hamilton-Jacobi equation \eqref{eq:firstHJeq}. Moreover, recalling the definition of the half Melnikov potential $L^{\mathrm{u}}$ in  \eqref{eq:halfstableMelnikov},  we have
\[
T^{\mathrm{u}}-L^{\mathrm{u}}= \mathcal{G} \left( \frac{1}{2 y_{\mathrm{h}}^2 } (\partial_u T^{\mathrm{u}}-r_{\mathrm{h}}^{-2} \partial_\beta T^{\mathrm{u}})^2+ \frac{1}{2r_{\mathrm{h}}^2 }((\delta^{\mathrm{u}}_\beta)^2+(\partial_\beta T^{\mathrm{u}})^2)\right)
\]
Therefore, using that $|\delta^{\mathrm{u}}_\beta|\lesssim \zeta G_{0}^{-5}$ and $\lVert T^{\mathrm{u}}\rVert_{1/3,1/2,\rho_0/4}\leq G_0^{-4}$, one easily obtains that 
\[
\lVert T^{\mathrm{u}}-L^{\mathrm{u}} \rVert_{1/3,1,\rho_0/8}\lesssim G_{0}^{-8}.
\]
We set $\rho=\rho_0/8$. Now we prove the estimate for the difference $T^{\mathrm{u}}-T^{\mathrm{u}}_{\mathrm{circ}}$.  The function $T_{\mathrm{circ}}^{\mathrm{u}}$ satisfies (compare \eqref{eq:firstHJeq})
\[
(1+A_{\mathrm{circ}}^{\mathrm{u}})\partial_u  T_{\mathrm{circ}}^{\mathrm{u}}+ B_{\mathrm{circ}}^{\mathrm{u}} \partial_\beta T_{\mathrm{circ}}^{\mathrm{u}}+ G_{0}^3\partial_t T_{\mathrm{circ}}^{\mathrm{u}}- V_{\mathrm{circ}}=0,
\]
with 
\begin{equation}\label{eq:defnABcircular}
A_{\mathrm{circ}}^{\mathrm{u}}=\frac{1}{2y_{\mathrm{h}}^2} \left(\partial_u T_{\mathrm{circ}}^{\mathrm{u}}-r_{\mathrm{h}}^{-2}\partial_\beta T_{\mathrm{circ}}^{\mathrm{u}} \right)\qquad\qquad B_{\mathrm{circ}}^{\mathrm{u}}=-\frac{1}{2y_{\mathrm{h}}^2 r_{\mathrm{h}}^2} \left(\partial_u T_{\mathrm{circ}}^{\mathrm{u}}-2r_{\mathrm{h}}^{-1}\partial_\beta T_{\mathrm{circ}}^{\mathrm{u}} \right).
\end{equation}
Using that (see Lemma \ref{lem:boundspotential}) $\lVert V-V_{\mathrm{circ}}\rVert_{2,3/2,\rho_0}\lesssim \zeta G_{0}^{-4}$  and, by hypothesis, $|\delta^{\mathrm{u}}_\beta|\lesssim \zeta  G_{0}^{-5}$, one easily obtains that 
\[
\lVert T^{\mathrm{u}}-T_{\mathrm{circ}}^{\mathrm{u}}\rVert_{1/3,1/2,\rho}\lesssim \zeta G_{0}^{-4}.
\]
This estimate implies that 
\[
\lVert A ^{\mathrm{u}}-A_{\mathrm{circ}}^{\mathrm{u}} \rVert_{2/3,1/2,\rho}\lesssim \zeta G_{0}^{-4}\qquad\qquad \lVert B^{\mathrm{u}}- B_{\mathrm{circ}}^{\mathrm{u}} \rVert_{2,3/2,\rho}\lesssim \zeta G_{0}^{-4},
\]
and we obtain that 
\[
\lVert T^{\mathrm{u}}-T^{\mathrm{u}}_{\mathrm{circ}}-(L^{\mathrm{u}}-L^{\mathrm{u}}_{\mathrm{circ}})\rVert_{1/3,1,\rho} \lesssim \zeta G_{0}^{-8},
\]
as was to be shown.
\end{proof}


\subsection{Extension of the parametrization by the flow}\label{sec:Extensionflow}

Theorem \ref{thm:stableparam} provides the existence of  a Lagrangian graph parametrization $\mathcal{W}^{\mathrm{u}}$ of the form \eqref{eq:stableparam} of the unstable manifold of the invariant torus $\mathcal{T}_{G^{\mathrm{u}}}$,   on the domain $(u,\beta,t)\in  D^{\mathrm{u}}_\kappa\times \mathbb{T}_\rho\times \mathbb{T}_\sigma$. As already discussed in Remark \ref{rem:dominios} (see also Section \ref{sec:differencegenfunctionsmanifolds} below), to study the difference between $W^{\mathrm{u}}_{G^{\mathrm{u}}}$ and $W^{\mathrm{s}}_{G^{\mathrm{s}}}$ we need to extend their parametrizations  to a common domain containing a subset of the real line. This is (at least in a direct manner) not possible using the parametrizations 
\eqref{eq:stableparam} since $y_{\mathrm{h}}(0)=0$.

\begin{figure}\label{fig:DominisFlux-Tilde}
\centering
\includegraphics[scale=0.60]{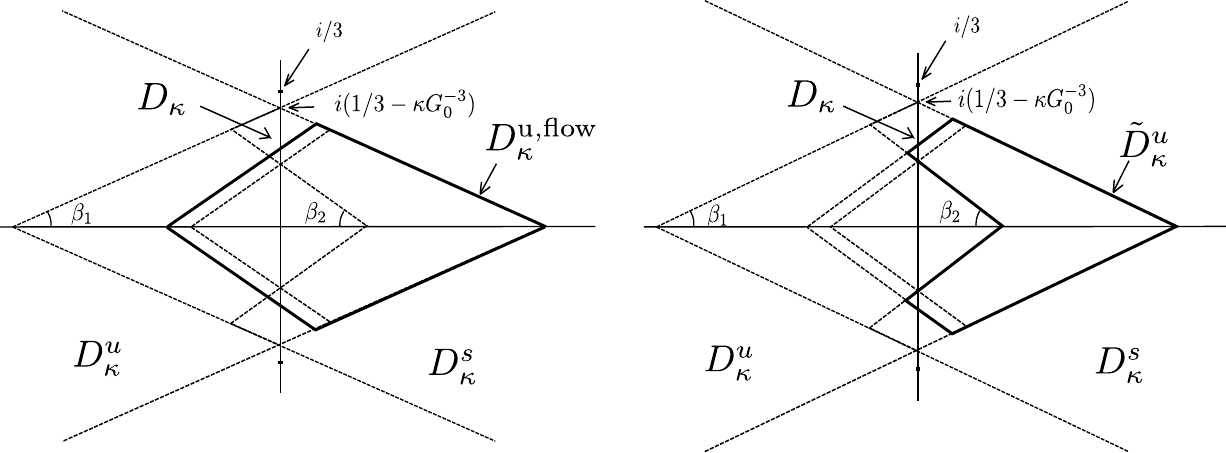}
\caption{The domains  $ D_{\kappa_1}^{\mathrm{u},\mathrm{flow}}$ and  $\tilde{D}_{\kappa_2}^{\mathrm{u}}$ defined in \eqref{eq:flowdomain} and \eqref{eq:tildedomain}.}
\end{figure}

We sketch the simple solution to this technical issue. The details can be found in \cite{MR3455155}: since in polar coordinates $(r,\alpha,t,y,G,E)$ the vector field associated to the Hamiltonian  \eqref{eq: Hamiltonian in polar coordinates} is not singular (except at $r=0$), we look for a different parametrization $\mathtt{W}^{\mathrm{u}}(\tilde{u},\tilde{\beta},t)$ of the unstable manifold in polar coordinates
\begin{equation}\label{eq:parametrizationflow}
\begin{split}
\mathtt{W}^{\mathrm{u}}(\tilde{u},\tilde{\beta},t)=\big\{ (r,\alpha,t,y,G,E)=&(G_{0}^{2}r_{\mathrm{h}}(\tilde{u})+R_{\mathrm{flow}}(\tilde{u},\tilde{\beta},t),\tilde{\beta}+\alpha_{\mathrm{h}}(\tilde{u})+\Omega_{\mathrm{flow}}(\tilde{u},\tilde{\beta},t),t, G_{0}^{-1}\\
& y_{\mathrm{h}}(\tilde{u})+ Y_{\mathrm{flow}}(\tilde{u},\tilde{\beta},t, G_{0}+J_{\mathrm{flow}}(\tilde{u},\tilde{\beta},t),E_{\mathrm{flow}}(\tilde{u},\tilde{\beta},t) ) \big\}
\end{split}
\end{equation}
such that 
\begin{equation}\label{eq:invarianceparametrizationflow}
\phi^{\mathrm{s}} (\mathtt{W}^{\mathrm{u}}(\tilde{u},\tilde{\beta},t))=\mathtt{W}^{\mathrm{u}}(\tilde{u}+s,\tilde{\beta},t+G_{0}^3s)
\end{equation}
where $\phi^{\mathrm{s}}_{\mathrm{pol}}$ is the time $s$ flow generated by the Hamiltonian \eqref{eq: Hamiltonian in polar coordinates}. Notice that this extension is a rather standard procedure since we will consider domains which are at distance order $\sim 1$ from the singularities $u=\pm i/3$.

Let $\mathcal{W}^{\mathrm{u}}$ be the Lagrangian graph parametrization associated to the generating function $T^{\mathrm{u}}$ obtained in Theorem  \ref{thm:stableparam}. The first step is to  perform a change of variables $h$ of the form 
\begin{equation}\label{eq:changeofvarsflow}
(u,\beta,t)=h(\tilde{u},\tilde{\beta},t)=(\tilde{u}+h_u(\tilde{u},\tilde{\beta},t),\tilde{\beta}+h_\beta(\tilde{u},\tilde{\beta},t),t)
\end{equation}
such that the parametrization  $\phi_{\mathrm{h}}\circ \mathcal{W}^{\mathrm{u}}\circ h$ is of the form \eqref{eq:parametrizationflow} and satisfies \eqref{eq:invarianceparametrizationflow}. This is the content of Lemma \ref {lem:firstlemmaflow} below. Second, we use the flow $\phi^{\mathrm{s}}_{\mathrm{pol}}$ to extend this parametrization  to a domain $(\tilde{u},\tilde{\beta},t)\in D^{u,\mathrm{flow}}_{\kappa_1}\times\mathbb{T}_{\rho_1}\times \mathbb{T}_{\sigma_1}$ (for suitable $\kappa_1>\kappa, \rho_1<\rho, \sigma_1<\sigma$ ) where 
\begin{equation}\label{eq:flowdomain}
D_{\kappa_1}^{\mathrm{u},\mathrm{flow}}= \{ \tilde{u}\in\mathbb{C}\colon |\operatorname{Im}  \tilde{u}|\leq -\tan \beta_1 \operatorname{Re}u+1/3-\kappa_1  G_{0}^{-3},\ |\operatorname{Im}\tilde{u}|\leq \tan\beta_2 \operatorname{Re}\tilde{u}+1/6+ \kappa_ 1G_{0}^{-3} \}.
\end{equation}
This domain contains $\tilde{u}=0$, is at distance $\sim \mathcal{O}(1)$ from $u=\pm i/3$, and satisfies $D_{\kappa_1}^{\mathrm{u},\mathrm{flow}}\cap D^{\mathrm{u}}_\kappa\neq \emptyset$ , $D_{\kappa_1}^{\mathrm{u},\mathrm{flow}}\cap D^{\mathrm{s}}\cap\mathbb{R}\neq \emptyset$ (see Figure \ref{fig:DominisFlux-Tilde}).

\begin{lem}\label{lem:firstlemmaflow}
Let $\lVert \cdot \rVert_{0,0,\rho}$ be as in \eqref{eq:unstableBanachspace} but referred to the domain $D_{\kappa_1}^{\mathrm{u},\mathrm{flow}}\cap D^{\mathrm{u}}_\kappa$. Then,  on the overlapping domain $(D_{\kappa_1}^{\mathrm{u},\mathrm{flow}}\cap D^{\mathrm{u}}_\kappa)\times\mathbb{T}_{\rho_1}\times \mathbb{T}_{\sigma_2}$, there exists an analytic change of coordinates $h$ of the form \eqref{eq:changeofvarsflow} such that 
\[
\lVert h_u\rVert _{0,0,\rho}\lesssim G_{0}^{-4}\qquad\qquad \lVert h_\beta\rVert_{0,0,\rho}\lesssim G_{0}^{-5/2},
\]
and for which the parametrization $
 \phi_{\mathrm{h}}\circ \mathcal{W}^{\mathrm{u}}\circ h: (D_{\kappa_1}^{\mathrm{u},\mathrm{flow}}\cap D^{\mathrm{u}}_\kappa)\times\mathbb{T}_{\rho_1}\times \mathbb{T}_{\sigma_2}\to \mathbb{C}^6 $ is of the form \eqref{eq:parametrizationflow} and satisfies \eqref{eq:invarianceparametrizationflow}.

\end{lem}

The proof of this lemma follows the same lines as the proof of Theorem 5.16  in  \cite{MR3455155}. As commented above, we now extend the parametrization obtained in Lemma \ref{lem:firstlemmaflow} to the domain $D^{u,\mathrm{flow}}_ {\kappa_1}$. Notice that this parametrization will be well defined at $\tilde{u}=0$ since the vector field associated to the Hamiltonian \eqref{eq: Hamiltonian in polar coordinates} is not singular at $r=G_{0}^2r_{\mathrm{h}}(0)\neq 0$.

\begin{lem}\label{lem:secondlemmaflow}
The parametrization $
\phi_{\mathrm{h}}\circ \mathcal{W}^{\mathrm{u}}\circ h: (D_{\kappa_1}^{\mathrm{u},\mathrm{flow}}\cap D^{\mathrm{u}}_\kappa)\times\mathbb{T}_{\rho_1}\times \mathbb{T}_{\sigma_2}\to \mathbb{C}^6$ obtained in Lemma \ref{lem:firstlemmaflow} can be extended analytically to a parametrization $\mathtt{W}^{\mathrm{u}}: D^{u,\mathrm{flow}}_{\kappa_1}\times\mathbb{T}_{\rho_1}\times \mathbb{T}_{\sigma_1}\to\mathbb{C}^6$ of the form \eqref{eq:parametrizationflow} which satisfies \eqref{eq:invarianceparametrizationflow} and such that 
\[
G_{0}^2 (\ln(G_{0})^{-1}\lVert R_{\mathrm{flow}} \rVert_{0,0,\rho_1},\ G_{0}^{-1/2}\ \lVert \Omega_{\mathrm{flow}} \rVert_{0,0,\rho_1} ,\ G_{0}^{-1}\ \lVert Y_{\mathrm{flow}} \rVert_{0,0,\rho_1},\ G_{0}^{-3/2} \lVert J_{\mathrm{flow}} \rVert_{0,0,\rho_1} \lesssim G_{0}^{-3},
\] 
where the norm $\lVert \cdot \rVert_{0,0,\rho}$ is as in \eqref{eq:unstableBanachspace} but referred to the domain $D_{\kappa_1}^{\mathrm{u},\mathrm{flow}}$.
\end{lem}

The proof of this lemma follows the same lines as the proof of Proposition 5.20  in  \cite{MR3455155}. Finally, we come back to the graph parametrization. To that end, for suitable $\kappa_2>\kappa_1, \rho_2<\rho_1, \sigma_2<\sigma_1$, we define the domain
\begin{equation}\label{eq:tildedomain}
\begin{split}
\tilde{D}_{\kappa_2}^{\mathrm{u}}=  \{u \in\mathbb{C}\colon& |\operatorname{Im} u|\leq - \tan \beta_1 \operatorname{Re}u +1/3-\kappa_2 G_{0}^{-3},\  |\operatorname{Im} u|\leq \tan \beta_2 \operatorname{Re}u +1/6-\kappa_2  G_{0}^{-3},\\
& |\operatorname{Im} u| \geq -\tan \beta_2 \operatorname{Re} u+1/6+ \kappa_2 G_{0}^{-3}\}
\end{split}
\end{equation}
which  is at distance $\sim\mathcal{O}(1)$ from $u=0$ and verifies $\tilde{D}_{\kappa_2}^{\mathrm{u}}\subset D_{\kappa_1}^{\mathrm{u},\mathrm{flow}}$ (see Figure \ref{fig:DominisFlux-Tilde}).

\begin{lem}
Let $\mathtt{W}^{\mathrm{u}}$ be the parametrization obtained in Lemma \ref{lem:secondlemmaflow}, which is of the form \eqref{eq:parametrizationflow}. Let $\lVert \cdot \rVert_{0,0,\rho}$ be as in \eqref{eq:unstableBanachspace} but referred to the domain $\tilde{D}^{\mathrm{u}}_{\kappa_2}$.  Then, there exists  an analytic change of coordinates $g=(\tilde{u}+g_u(\tilde{u},\tilde{\beta},t),\tilde{\beta}+g_\beta(\tilde{u},\tilde{\beta},t), t)$ such that 
\[
\lVert g_u\rVert _{0,0,\rho}\lesssim G_{0}^{-4}\qquad\qquad \lVert g_\beta\rVert_{0,0,\rho}\lesssim G_{0}^{-5/2},
\]
and such that $\phi_{\mathrm{h}}^{-1}\circ\mathtt{W}^{\mathrm{u}}\circ g$ constitutes the unique analytic extension,  to the domain $(u,\beta,t)\in \tilde{D}_{\kappa_2}^{\mathrm{u}}\times\mathbb{T}_{\rho_2}\times \mathbb{T}_\sigma$, of the Lagrangian graph parametrization $\mathcal{W}^{\mathrm{u}}$  associated to the function $T^{\mathrm{u}}$ obtained in Theorem \ref{thm:stableparam}.
\end{lem}

The proof of this lemma follows the same lines as the proof of Proposition 5.21 in  \cite{MR3455155}. In conclusion, we have proven the existence of the analytic continuation of the unstable generating function $T^{\mathrm{u}}$ to the domain (see Figure \ref{fig:DominisBoomerang2})
\begin{equation}\label{eq:finalunstabledomain}
\begin{split}
D_{\kappa_2}=  \{u \in\mathbb{C}\colon& |\operatorname{Im} u|\leq - \tan \beta_1 \operatorname{Re}u +1/3-\kappa_2 G_{0}^{-3},\  |\operatorname{Im} u|\leq \tan \beta_1 \operatorname{Re}u+ 1/3-\kappa_2 G_{0}^{-3},\\
& |\operatorname{Im} u| \geq -\tan \beta_2 \operatorname{Re} u+1/6+\kappa_2 G_{0}^{-3}\}.
\end{split}
\end{equation}

\noindent Indeed, introducing the Banach spaces
\begin{equation}\label{eq:finalBanachsplitting}
\mathcal{Y}_{\nu,\rho}=\left\{ h=\{h^{[l]}\}_{l\in\mathbb{Z}}\colon h^{[l]}:D_{\kappa_2}\times \mathbb{T}_{\rho}\to\mathbb{C}\  \text{is analytic for all } l\in \mathbb{Z}\ \text{and}\   \lVert h\rVert_{\nu,\rho}<\infty\right\},
\end{equation}
where $\lVert \cdot \rVert_{\nu,\rho}$ is the Fourier sup norm
\begin{equation}\label{eq:finalnormsplitting}
\lVert h \rVert_{\nu,\rho}=\sum_{l\in\mathbb{Z}} \lVert h^{[l]} \rVert_{\nu,\rho,l} e^{-|l| \sigma}\quad\quad  \lVert h^{[l]} \rVert_{\nu,\rho,l}=\sup_{(u,\beta)\in D_{\kappa_2}\times \mathbb{T}_\rho} \left|  (u-i/3)^{\nu+l/2} (u+i/3)^{\nu-l/2} h^{[l]} (u,\beta)\right|
\end{equation}
 (notice that the weight $u^\eta$ becomes now meaningless since $D_{\kappa_2}$ is bounded), the following proposition, which extends the domain of definition of the function element $T^{\mathrm{u}}$ in Theorem  \ref{thm:stableparam}, holds.

\begin{figure}\label{fig:DominisBoomerang2}
\centering
\includegraphics[scale=0.60]{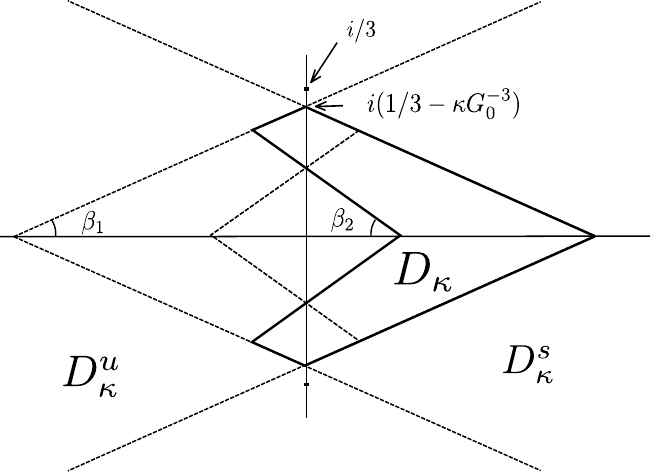}
\caption{The domain  $ D_{\kappa}$ defined in \eqref{eq:finalunstabledomain}. }
\end{figure}

\begin{prop}\label{prop:extensionunstableparam}
There exist $\kappa_2,\sigma_2, \rho_2>0$  such that for $G_0$ large enough, $0\leq \zeta\leq G_0^{-3}$ and  $G^{\mathrm{u}},G^{\mathrm{s}}\in\mathbb{G}_{\rho_2}(G_0)$, there exists  $T^{\mathrm{u}}\in \mathcal{Y}_{1/2,\rho_2}$ which constitutes the unique analytic continuation to  $(u,\beta,t)\in D_{\kappa_2}\times\mathbb{T}_{\rho_2}\times \mathbb{T}_{\sigma_2}$ of the function obtained in Theorem \ref{thm:stableparam}. Moreover, in this domain
\[
\lVert T^{\mathrm{u}} \rVert_{1/2,\rho_2}\lesssim G_{0}^{-4} \qquad\qquad\text{and}\qquad\qquad \lVert T^{\mathrm{u}}- L^{\mathrm{u}}\rVert_{1,\rho_2}\lesssim G_{0}^{-8},
\]
where $L^{\mathrm{u}}$ is the unstable half Melnikov potential defined in \eqref{eq:halfstableMelnikov}. In addition, we have that
\[
\lVert T^{\mathrm{u}}-T^{\mathrm{u}}_{\mathrm{circ}}-(L^{\mathrm{u}}-L^{\mathrm{u}}_{\mathrm{circ}})\rVert_{1,\rho_2} \lesssim \zeta G_{0}^{-8}
\]
where $T^{\mathrm{u}}_{\mathrm{circ}}$ is defined in \eqref{eq:generatingcircular} and $L^{\mathrm{u}}_{\mathrm{circ}}(u,t-\beta;G_{0})=L^{\mathrm{u}} (u,\beta,t; G_{0},0)$.
\end{prop}


\subsection{The difference $\Delta S$ between the generating functions of the invariant manifolds}\label{sec:differencegenfunctionsmanifolds}

In Theorem \ref{thm:stableparam}  we have proved that, for suitable $\kappa,\sigma>0$ and $\rho>0$, the formal Fourier series $T^{\mathrm{u}}$ (see Remark \ref{rem:formalseries})  in the parametrization \eqref{eq:stableparam} of the unstable manifold  of the torus $\mathcal{T}_{G^{\mathrm{u}}}$ is uniformly $\mathcal{O}(G_{0}^{-8})$ approximated in $\mathcal{X}_{1/3,1/2,\rho}$, by the half Melnikov potential $L^{\mathrm{u}}$ introduced in \eqref{eq:halfstableMelnikov}. Moreover, in Proposition \ref{prop:extensionunstableparam} we have shown that  $T^{\mathrm{u}}$ admits a unique analytic continuation to the domain $(u,\beta,t)\in D_{\kappa_2}\times\mathbb{T}_{\rho_2}\times \mathbb{T}_{\sigma_2}$  for suitable $\kappa_2>\kappa, \rho_2<\rho$ and $\sigma_2<\sigma$.

The very same argument in the proof of Theorem \ref{thm:stableparam} shows the stable counterpart for the formal Fourier series $T^{\mathrm{s}}$ on the domain $(u,\beta,t)\in D^{\mathrm{s}}_{\kappa}\times\mathbb{T}_{\rho}\times\mathbb{T}_\sigma$ where $D^{\mathrm{s}}_{\kappa}=\{u\in\mathbb{C}\colon -u\in D^{\mathrm{u}}_{\kappa}\}$. Moreover, denoting by $\mathcal{X}^{\mathrm{s}}_{1/3,1/2,\rho}$ the associated Banach space for formal Fourier series defined on $D^{\mathrm{s}}_\kappa\times\mathbb{T}_{\rho}\times\mathbb{T}_\sigma$, $T^{\mathrm{s}}$ is uniformly $\mathcal{O}( G_{0}^{-7})$ approximated  in $\mathcal{X}^{\mathrm{s}}_{1/3,1/2,\rho}$ by the stable half Melnikov potential 
\begin{equation}
L^{\mathrm{s}}(u,\beta,t;G_{0},\zeta)=\int_{+\infty}^0 V(u+s,\beta, t+ G_{0}^3 s;G_{0},\zeta) \mathrm{d}s.
\end{equation}

\begin{rem}\label{eq:remarkdomainsdef}
    The domains $D^{\mathrm{u}}$ and $D^{\mathrm{s}}$ which we introduced without definition at the beginning of Section ... can be now defined as 
    \[
    D^{\mathrm u}= D_\kappa^{\mathrm u}\cup D_{\kappa_2},\qquad\qquad    D^{\mathrm s}= D_\kappa^{\mathrm s}
    \]
\end{rem}

Since $D_{\kappa_2}\subset D^{\mathrm{s}}_\kappa$, we can now analyze the difference between the generating functions of the stable and unstable manifolds (see equation \eqref {eq:defngeneratingfunctmanif} and the discussion below it)
\begin{equation}\label{eq:diffgenfunct2}
\Delta S=S^{\mathrm{u}}-S^{\mathrm{s}}=\langle \boldsymbol\delta^{\mathrm{u}}-\boldsymbol\delta^{\mathrm{s}},z\rangle + T^{\mathrm{u}}-T^{\mathrm{s}}
\end{equation}
on the  common domain $z=(u,\beta,t)\in D_{\kappa_2}\times\mathbb{T}_{\rho_2}\times \mathbb{T}_{\sigma_2}$. For the sake of clarity in the forthcoming arguments,  we summarize in Theorem \ref{thm:differencegenfunctcommon} the previous discussion. We denote by 
\begin{equation}\label{eq:differencecircular}
\Delta S_{\mathrm{circ}}(u,t-\beta;G_{0})\equiv T_{\mathrm{circ}}^{\mathrm{u}}(u,t-\beta;G_{0})-T_{\mathrm{circ}}^{\mathrm{s}}(u,t-\beta;G_{0}).
\end{equation}

\begin{thm}\label{thm:differencegenfunctcommon}
There, there exist $\kappa_2,\sigma_2, \rho_2>0$ such that for $G_0$ large enough,  $0\leq \zeta \leq G_0^{-3}$, and $G^{\mathrm{u}},G^{\mathrm{s}}\in\mathbb{G}_{\rho_2}(G_0)$,  the difference $\Delta S=S^{\mathrm{u}}-S^{\mathrm{s}}$  defined in  \eqref{eq:defngeneratingfunctmanif} satisfies $\Delta S\in \mathcal{Y}_{1/2,\rho_2}$ and 
\[
\lVert \Delta S -\langle \boldsymbol\delta^{\mathrm{u}}-\boldsymbol\delta^{\mathrm{s}},z\rangle -\tilde{L} \rVert _{1/2,\rho_2}\lesssim G_{0}^{-8},
\]
where the norm $\lVert \cdot \rVert_{\nu,\rho}$ is defined in \eqref{eq:finalnormsplitting} and the Melnikov potential $\tilde{L}$ is defined in \eqref{eq:defnMelnikovPot}. Moreover, we have
\[
\lVert \Delta S- \Delta S_{\mathrm{circ}}-( \langle\boldsymbol\delta^{\mathrm{u}}-\boldsymbol\delta^{\mathrm{s}},z \rangle +\tilde{L}-\tilde{L}_{\mathrm{circ}}) \rVert_{1,\rho_2}\lesssim \zeta G_{0}^{-8},
\]
where  $\Delta S_{\mathrm{circ}}$ is defined in \eqref{eq:differencecircular} and $\tilde{L}_{\mathrm{circ}}$ is defined in \eqref{eq:Melnikovcircular}.
\end{thm}

We now recall that the aim of Section \ref{sec:generatingfunctsinvmanifolds} is to  show that the existence of nondegenerate critical points of the function $\langle (\boldsymbol\delta^{\mathrm{u}}-\boldsymbol\delta^{\mathrm{s}}),z\rangle +\tilde{L}$ implies the existence of critical points of the function $z\to \Delta S$. Namely, our goal is to prove Theorem \ref{thm:mainthmcriticalpoints}. As first step, we  provide a proof of Proposition \ref{prop:approxgenfunctionbyMelnikov}.  With that objective we  state the following lemma, whose proof is given in Appendix \ref{sec:Melnikov}.

\begin{lem}\label{lem:mainMelnikov}
Let $\rho_0>0$ be given in Lemma \ref{lem:boundspotential}. Then, there exists $G_0>0$  such that for $0\leq \zeta \leq G_0^{-3}$, the Melnikov potential $\tilde{L}$ defined in \eqref{eq:defnMelnikovPot} is a real-analytic function of all its arguments  and can be expressed as the absolutely convergent  series
\[
\tilde{L}(u,\beta,t;G_0,\zeta)=\sum_{l\in\mathbb{N}} \mathcal{L}_l (t-G_{0}^3u,\beta;G_{0},\zeta),
\]
where, writing $\sigma=t-G_{0}^3u$,
\begin{itemize}

\item $\displaystyle \mathcal{L}_0(\beta;G_{0})=\mu(1-\mu) \big(L_{0,0} (G_{0},\zeta)+ L_{0,1} (G_{0},\zeta)\cos \beta + E_0(\beta;G_{0},\zeta) \big)$ with 
\[
\begin{split}
L_{0,0}(G_{0},\zeta)=&\frac{\pi}{2 G_{0}^4} \left(1+ \mathcal{O} ( G_{0}^{-4},\zeta^{-2}) \right)\\
L_{0,1}(G_{0},\zeta)=&-(1-2\mu)\frac{15\pi\zeta}{8 G_{0}^6} \left(1+ \mathcal{O} ( G_{0}^{-4},\zeta^{-2}) \right)\\
|E_0(\beta;G_{0},\zeta)|\lesssim &\zeta^2  G_{0}^{-8},
\end{split}
\]

\item $\displaystyle \mathcal{L}_1(\sigma,\beta;G_{0},\zeta)= \mu(1-\mu) \big(2L_{1,1}(G_{0},\zeta)\cos (\sigma-\beta)+2L_{1,2}(G_{0},\zeta)\cos (\sigma-2\beta)+ E_1(\sigma,\beta;G_{0},\zeta) \big)$
with
\begin{equation}\label{eq:firstharmonicMelnikov}
\begin{split}
2L_{1,1}(G_{0},\zeta)=&(1-2\mu)\sqrt{\frac{\pi}{8 G_{0}^3}} \left(1+ \mathcal{O} ( G_{0}^{-1},\zeta^{-2}) \right) \exp(-G_{0}^3/3)\\
2L_{1,2}(G_{0},\zeta)=&-3\zeta \sqrt{2\pi G_0} \left(1+ \mathcal{O} ( G_{0}^{-1},\zeta^{-1}) \right) \exp(-G_{0}^3/3)\\
|E_1(\sigma,\beta;G_{0},\zeta)|\lesssim &\zeta ( G_{0}^{-5/2}+\zeta  G_{0}^{3/2} )\exp(-G_{0}^3/3),
\end{split}
\end{equation}

\item The sum of the higher coefficients 
\[
\mathcal{L}_{\geq 2}(u,\beta,t; G_{0})=\sum_{l\geq 2} \mathcal{L}_l(\sigma,\beta; G_{0},\zeta)
\]
satisfies the estimate
\[
|\mathcal{L}_{\geq 2}| \lesssim   G_{0}^{1/2} \exp(-2 G_{0}^3/3).
\]
\end{itemize}
\end{lem}\vspace{0.3cm}

\noindent Notice that the estimates in Theorem \ref{thm:stableparam} only imply
\[
|\partial_u \Delta (S-\tilde{L})|\lesssim  G_{0}^{-8}
\]
while 
\[
|\partial_u \tilde{L}|\sim  G_{0}^{3/2}\exp(-(G_{0}^{3}/3).
\]
The existence of critical points of $\Delta S$ as a consequence of the existence of nondegenerate critical points of the function $\langle(\boldsymbol\delta^{\mathrm{u}}-\boldsymbol\delta^{\mathrm{s}}),z\rangle + \tilde{L}$ is therefore not clear at the moment. This ``mismatch'' is caused by not looking at the problem in the right set of coordinates. In Lemma \ref{lem:straighteningcv} below, we prove the existence of a change of variables $(u,\beta,t)=\Phi(v,\theta,t)$ such that $\Delta \mathcal{S}=\Delta S\circ \Phi$ only depends on $v$ and $t$ through the difference $\sigma=t-G_{0}^3v$. This fact is equivalent to $\Delta \mathcal{S}\in \mathrm{Ker}\mathcal{L}$ where $\mathcal{L}$ is the linear operator
\[
\mathcal{L}=\partial_v+ G_{0}^3 \partial_t.
\]
Then, in Lemma \ref{lem:exponsmallness} it is shown that functions in $\mathcal{Y}_{\nu,\rho}\cap \mathrm{Ker} \mathcal{L}$  (see \eqref{eq:finalBanachsplitting}), present an exponential decay in the size of their Fourier coefficients. Finally, this last property, together with the approximation of $\Delta \mathcal{S}$ by $\langle (\boldsymbol\delta^{\mathrm{u}}-\boldsymbol\delta^{\mathrm{s}}),q\rangle +\tilde{L}$ in the norm \eqref{eq:finalnormsplitting}, given in Theorem \ref{thm:differencegenfunctcommon}, are used to complete the proof of Proposition \ref{prop:approxgenfunctionbyMelnikov}.

\begin{lem}\label{lem:straighteningcv}
There exists $\rho_3>0$ such that, for $G_0$ large enough, $0\leq \zeta\leq G_0^{-3}$ and  $G^{\mathrm{u}},G^{\mathrm{s}}\in\mathbb{G}_{\rho_3}(G_0)$,  there exists an analytic change of variables of the form 
\[
(u,\beta,t)=\Phi(v,\theta,t)=(v+\phi_v(v,\theta,t),\theta+\phi_\theta(v,\theta,t),t)
\]
with $\phi_v\in \mathcal{Y}_{0,\rho_3}$, $\phi_\theta\in \mathcal{Y}_{1/2,\rho_3}$ and $\lVert \phi_v\rVert_{0,\rho_3}\lesssim G_{0}^{-4}$,  $\lVert \phi_\theta \rVert_{1/2,\rho_3}\lesssim G_{0}^{-4}$, such that $\Delta \mathcal{S}=\Delta S\circ \Phi$ satisfies
\begin{equation}
\mathcal{L} \Delta \mathcal{S}= (\partial_v+G_{0}^3 \partial_t) \Delta \mathcal{S}=0.
\end{equation}
Moreover, under the same hypotheses, there exists an analytic change of variables 
\[
(u,t-\beta)=\Phi_{\mathrm{circ}}(v,t-\theta)=(v+\phi_{v,\mathrm{circ}}(v,t-\theta),\theta+\phi_{\theta,\mathrm{circ}}(v,t-\theta))
\]
with $\phi_{v,\mathrm{circ}}\in \mathcal{Y}_{0,\rho_3}$, $\phi_{\theta,\mathrm{circ}}\in \mathcal{Y}_{1/2,\rho_3}$ and $\lVert \phi_v-\phi_{v,\mathrm{circ}}\rVert_{0,\rho_3}\lesssim \zeta G_{0}^{-4}$,  $\lVert \phi_\theta-\phi_{\theta,\mathrm{circ}} \rVert_{1/2,\rho_3}\lesssim \zeta G_{0}^{-4}$, such that $\Delta \mathcal{S}_{\mathrm{circ}}=\Delta S_{\mathrm{circ}}\circ \Phi_{\mathrm{circ}}$ satisfies
\begin{equation}
\Delta  \mathcal{S}_{\mathrm{circ}}(v,t-\theta)= \Delta \widehat{\mathcal{S}}_{\mathrm{circ}}(t-\theta-G_{0}^3 v;G_{0}),
\end{equation}
for some periodic function $\Delta \widehat{\mathcal{S}}_{\mathrm{circ}}(t-\theta-G_{0}^3 v;G_{0})$.
\end{lem}

\begin{proof}
 Using that both $S^{\mathrm{u},\mathrm{s}}$ satisfy the same Hamilton-Jacobi equation $ 
H(q,\nabla S^{\mathrm{u},\mathrm{s}})=0$
it is an straightforward computation to check that $\Delta S$ is a solution to $
\mathcal{\widetilde{L}} \Delta S=0$
with
\begin{equation}\label{eq:semisumdiffoperator}
\mathcal{\widetilde{L}}=\left(1+(\hat A^{\mathrm{s}}+\hat A^{\mathrm{u}}) \right) \partial_u +(\hat B^{\mathrm{s}}+\hat B^{\mathrm{u}}) \partial_\beta +G_{0}^3 \partial_t
\end{equation}
where 
\begin{equation}\label{eq:definitionAusBushat}
\hat{A}^{\mathrm{u}}=A^{\mathrm{u}}+\frac{\delta^{\mathrm{u}} }{2y_{\mathrm{h}}^{2}r_{\mathrm{h}}^{2}},\qquad\qquad\hat{B}^{\mathrm{u}}=B^{\mathrm{u}}-\frac{\delta^{\mathrm{u}} }{y_{\mathrm{h}}^{2}r_{\mathrm{h}}^{3}}
\end{equation}
and $A^{\mathrm{u},\mathrm{s}},B^{\mathrm{u},\mathrm{s}}$ are defined as in \eqref{eq:definitionAusBus}. One can now check that $\Delta\mathcal{S}\in\mathrm{Ker}\mathcal{L}$, if and only if, $ \Phi$ satisfies
\begin{equation}\label{eq:proofstraighteningdiff}
\mathcal{L} \phi_v= (\hat A^{\mathrm{s}}+\hat A^{\mathrm{u}})\circ \Phi\qquad\qquad\text{and}\qquad\qquad \mathcal{L} \phi_\beta=(\hat B^{\mathrm{s}}+\hat B^{\mathrm{u}})\circ \Phi.
\end{equation}
In order to rewrite \eqref{eq:proofstraighteningdiff} as a fixed point equation for $\Phi$ we introduce the left inverse operator $\mathcal{G}$ for $\mathcal{L}$ defined by the expression (here $v_+$  and $v_-$ are the top and bottom points and $v_0$ is any real point in the domain $D_{\kappa_2}$ defined in \eqref{eq:finalunstabledomain}),
\begin{equation}\label{eq:leftinverseoperatorfinaldomain}
\mathcal{G}(h)=\sum_{l\in\mathbb{Z}} \mathcal{G}^{[l]}(h)
\end{equation}
with
\[
\begin{aligned}
\mathcal{G}^{[l]}(h)=& \int_{v_+-v}^0 h^{[l]}(v+s,\theta) e^{ilG_{0}^3 s}\mathrm{d}s &\qquad\qquad\text{for}\qquad l>0\\
\mathcal{G}^{[0]}(h)=&\int_{v_0-v}^0 h^{[0]}(v+s,\theta)\mathrm{d}s &\qquad\qquad\text{for}\qquad l=0\\
\mathcal{G}^{[l]}(h)=&\int_{v_--v}^0 h^{[l]}(v+s,\theta) e^{ilG_{0}^3 s}\mathrm{d}s &\qquad\qquad\text{for}\qquad l<0.\\
\end{aligned}
\]
Therefore, it is enough to look for $\Phi$ satisfying
\[
\phi_v= \mathcal{G}((\hat A^{\mathrm{s}}+\hat A^{\mathrm{u}})\circ \Phi) \qquad\qquad\text{and}\qquad\qquad \phi_\theta=\mathcal{G}((\hat B^{\mathrm{s}}+\hat B^{\mathrm{u}})\circ \Phi).
\]
The proof of the first part of the lemma now follows from a standard fixed point argument along the lines (but considerably simpler) of the proof of Lemma \ref{lem:straighteninglinoperator} (see also Theorem 6.3 in  \cite{MR3455155}). In particular, the proof is easily completed using the estimates
\[
\lVert \hat A^{\mathrm{u},\mathrm{s}} \rVert_{1/2,\rho}\lesssim G_{0}^{-4} \qquad \lVert \hat B^{\mathrm{u},\mathrm{s}} \rVert_{3/2,\rho} \lesssim G_{0}^{-4},
\]
which are obtained in an straightforward manner from Proposition \ref{prop:extensionunstableparam} and the discussion at the beginning of Section \ref{sec:differencegenfunctionsmanifolds} by taking, for example, $\rho_3\leq \rho_2/2$. To deal with compositions, we make use of a natural extension of Lemma \ref{lem:compositionangularvariabletechnical} which allows to treat also changes of variables in $v$ (details can be found in \cite{MR3455155}).

We now prove the second part of the lemma. Introduce the angle $\xi=t-\beta$, and write $\Delta S_{\mathrm{circ}}(u,\xi;G_{0})$. Therefore, $\Delta S_{\mathrm{circ}}$ is a solution to $
\mathcal{\widetilde{L}}_{\mathrm{circ}} \Delta S_{\mathrm{circ}}=0$
where 
\begin{equation}\label{eq:semisumdiffoperatorcircular}
\mathcal{\widetilde{L}}_{\mathrm{circ}}=\left(1+(A^{\mathrm{s}}_{\mathrm{circ}}+A^{\mathrm{u}}_{\mathrm{circ}}) \right) \partial_u +(G_{0}^3 -B^{\mathrm{s}}_{\mathrm{circ}}-B^{\mathrm{u}}_{\mathrm{circ}})\partial_\xi
\end{equation}
and $A^{\mathrm{u},\mathrm{s}}_{\mathrm{circ}},B^{\mathrm{u},\mathrm{s}}_{\mathrm{circ}}$ are defined in \eqref{eq:defnABcircular}. Thus, $\Delta\mathcal{S}_{\mathrm{circ}}\in\mathrm{Ker}\mathcal{L}$, if and only if, $\tilde{\Phi}_{\mathrm{circ}}(v,\tilde{\xi})\equiv (v+\phi_{v,\mathrm{circ}}(v,\tilde{\xi}),\tilde{\xi}-\phi_{\theta,\mathrm{circ}}(v,\tilde{\xi}))$ satisfies
\[
\mathcal{L} \phi_{v,\mathrm{circ}}= (A^{\mathrm{s}}_{\mathrm{circ}}+A^{\mathrm{u}}_{\mathrm{circ}})\circ \tilde{\Phi}_{\mathrm{circ}}\qquad\qquad\text{and}\qquad\qquad \mathcal{L} \phi_{\theta,\mathrm{circ}}=(B^{\mathrm{s}}_{\mathrm{circ}}+B^{\mathrm{u}}_{\mathrm{circ}})\circ \tilde{\Phi}_{\mathrm{circ}}.
\]
The lemma follows using the estimates (see the proof of Theorem \ref{thm:stableparam})
\[
\lVert A^{\mathrm{u},\mathrm{s}}-A^{\mathrm{u},\mathrm{s}}_{\mathrm{circ}} \rVert_{1/2,\rho}\lesssim \zeta G_{0}^{-4} \qquad \lVert B^{\mathrm{u},\mathrm{s}}-B^{\mathrm{u},\mathrm{s}}_{\mathrm{circ}}  \rVert_{3/2,\rho} \lesssim \zeta G_{0}^{-4}\qedhere\
\]

\end{proof}

The following lemma  gives the exponential decay of the Fourier coefficients for functions in $\mathcal{Y}_{\nu,\rho}\cap\mathrm{Ker}\mathcal{L}$ (see also Lemma 6.7  in \cite{MR3455155})

\begin{lem}\label{lem:exponsmallness}
Fix $\nu,\rho\geq 0$ and let $h\in \mathcal{Y}_{\nu,\rho}$ be such that $h\in\mathrm{Ker}\mathcal{L}$. Then $h$ can be written as
\[
h(v,\theta,t)=\sum_{l\in\mathbb{Z}} \Lambda^{[l]}(\theta) e^{il(t-G_{0}^3v)}
\]
and, for some $C>0$ independent of $\lVert h\rVert_{\nu,\rho}$ and $G_{0}$,
\[
\begin{split}
\sup_{\theta\in\mathbb{T}_{\rho}}|\Lambda^{[l]}(\theta)|\lesssim & \lVert h\rVert_{\nu,\rho}\  (CG_{0})^{3(\nu+|l|/2)}\exp(-|l| G_{0}^3/3).
\end{split}
\]
\end{lem}
\begin{proof}
Write 
\[
h(v,\theta,t)=\sum_{l\in\mathbb{Z}} h^{[l]}(v,\theta)e^{il t}.
\]
Since $h\in \mathrm{Ker}\mathcal{L}$ 
\[
h(v,\theta,t)=\sum_{l\in\mathbb{Z}} \Lambda^{[l]}(\theta)e^{il (t-G_{0}^3v)},
\]
where $\Lambda ^{[l]}(\theta)= h^{[l]} (v,\theta) e^{il G_{0}^3 v}$ is independent of $v$. For $l>0$, we evaluate at $v_+= i( 1/3- \kappa G_{0}^{-3})$ and use that 
\[
\lVert h^{[l]} \rVert_{\nu,\rho}\leq G_{0}^{3|l|/2} \lVert h \rVert_{\nu,\rho}
\]
to obtain that 
\[
\begin{split}
|\Lambda^{[l]}|\leq & G_{0}^{3\nu} \lVert h^{[l]}\rVert_{\nu,\rho}\  \exp(-|l| G_{0}^3 (1/3- \kappa G_{0}^{-3}) )\\
\leq &\lVert h\rVert_{\nu,\rho}\  (CG_{0})^{3(\nu+|l|/2)} \exp(-|l| G_{0}^3/3)\\
\end{split}
\]
for some $C>0$. The result for $l<0$ is obtained analogously evaluating at  $v_-=- i( 1/3- \kappa G_{0}^{-3})$.
\end{proof}

We now have all the ingredients to complete the proof of Proposition \ref{prop:approxgenfunctionbyMelnikov}.

\begin{rem}
In the following, we rename as $\rho$ the constant $\rho_3>0$ which was obtained in Lemma \ref{lem:straighteningcv}.
\end{rem}

\begin{proof}[Proof of Proposition \ref{prop:approxgenfunctionbyMelnikov}]
Recall that $\tilde{L}$, which was defined in \eqref{eq:defnMelnikovPot}, satisfies
\[
\tilde{L}(u,\beta,t; G_{0},\zeta)=L(t-G_{0}^3 u,\beta;G_{0},\zeta),
\]
where $L(\sigma,\beta;G_{0},\zeta)$ was defined in \eqref{eq:defnMelnikovPotential}. Let $\tilde{z}=(v,\theta,t)$. Since $\mathcal{E}=\Delta\mathcal{S}-\langle (\boldsymbol\delta^{\mathrm{u}}-\boldsymbol\delta^{\mathrm{s}}), \tilde{q}\rangle -\tilde{L}\in\mathrm{Ker}\mathcal{L}$, and $\mathcal{E}\in\mathcal{Y}_{1/2,\rho}$, it is enough to estimate $\lVert \mathcal{E}\rVert_{1/2,\rho}$ and apply Lemma \ref{lem:exponsmallness}. To that end, we write 
\[
\mathcal{E}=  \mathcal{E}_1+ \mathcal{E}_2
\]
with $\mathcal{E}_1=\Delta S-\langle \boldsymbol\delta^{\mathrm{u}}-\boldsymbol\delta^{\mathrm{s}},\tilde{z}\rangle-\tilde{L}$,  and $\mathcal{E}_2=\Delta \mathcal{S}-\Delta S$. Using that $|\delta^{\mathrm{u},\mathrm{s}}|\lesssim \zeta G_{0}^{-5}$, the estimate for $\lVert \phi_\theta\rVert_{0,\rho}$ in Lemma \ref{lem:straighteningcv} and the estimate for $\Delta S-\langle\delta^{\mathrm{u}}-\delta ^{\mathrm{s}},q\rangle -\tilde{L}$ in Theorem \ref{thm:differencegenfunctcommon}
we obtain
\[
\lVert \mathcal{E}_1 \rVert_{1/2,\rho} \lesssim G_{0}^{-8}.
\]
In order to bound $\mathcal{E}_2$, it follows from the mean value theorem, the estimates for $\lVert S^{\mathrm{u},\mathrm{s}} \rVert_{3/2,\rho}$,  which can be deduced from Proposition \ref{prop:extensionunstableparam} and the analogous version for $T^{\mathrm{s}}$ (see the discussion at Section \ref{sec:differencegenfunctionsmanifolds}), and the estimates for$\lVert \phi_v\rVert_{0,\rho}, \lVert \phi_\theta\rVert_{1/2,\rho}$ in Lemma \ref{lem:straighteningcv},  that
\[
\lVert \mathcal{E}_2 \rVert_{1/2,\rho}\lesssim G_{0}^{-13/2}.
\]
Applying Lemma \ref{lem:exponsmallness}, we obtain that 
\begin{equation}\label{eq:Asymptoticdifferenceharmonics}
\Delta \mathcal{S}-\langle \boldsymbol\delta^{\mathrm{u}}-\boldsymbol\delta^{\mathrm{s}},\tilde{z}\rangle-\tilde{L}=\sum_{l\in\mathbb{Z}} \mathcal{E}^{[l]}(\theta) e^{il(t-G_{0}^3v)}
\end{equation}
where, there exists some $C>0$, such that for $l\neq0$,
\[
\begin{split}
\sup_{\theta\in \mathbb{T}_{\rho}} |\mathcal{E}^{[l]} (\theta) |\lesssim & (CG_{0})^{3(1+|l|)/2} \lVert \mathcal{E} \rVert_{1/2,\rho} \exp(-|l|G_{0}^3/3)\\
\lesssim &(CG_{0})^{-4+3|l|/2} \exp(-|l| G_{0}^3/3)\\
\end{split}
\]
as was to be shown. 
\end{proof}

Finally, we also state the following lemma which will prove useful in the proof, in Section \ref{sec:cricitcalpointsgenfunct},  of Theorem \ref{thm:mainthmcriticalpoints}.

\begin{lem}\label{lem:summaryerrorsapprox}
Define the function $\mathcal{E}_{\mathrm{circ}}(\theta,\sigma;G^{\mathrm{u}},G^{\mathrm{s}})$ given by
\begin{equation}\label{eq:finalerrorapproxcircular}
\mathcal{E}_{\mathrm{circ}}=\Delta \mathcal{S}-\Delta \mathcal{S}_{\mathrm{circ}}-( (\delta^{\mathrm{u}}-\delta^{\mathrm{s}})\theta+ L-L_{\mathrm{circ}})
\end{equation}
where $\Delta \mathcal{S}$ and $\Delta \mathcal{S}_{\mathrm{circ}}$ are defined in  Lemma \ref{lem:straighteningcv}, $L$ is defined in \eqref{eq:defnMelnikovPotential} and $L_{\mathrm{circ}}(\sigma-\theta;G_{0})=L(\sigma,\theta;G_{0},0)$. Then, we have that 
\[
\lVert \mathcal{E}_{\mathrm{circ}} \rVert_{1/2,\rho} \lesssim \zeta G_{0}^{-13/2}.
\]
\end{lem}

\begin{proof}
We write $\mathcal{E}_{\mathrm{circ}}=\mathcal{E}_{\mathrm{circ},1}+\mathcal{E}_{\mathrm{circ},2}+\mathcal{E}_{\mathrm{circ},3}+\mathcal{E}_{\mathrm{circ},4}$ with
\[
\begin{split}
\mathcal{E}_{\mathrm{circ},1}= &\left(\Delta S-\Delta S_{\mathrm{circ}}- (\delta^{\mathrm{u}}-\delta^{\mathrm{s}})\beta-(\tilde{L}-\tilde{L}_{\mathrm{circ}}) \right)\circ \Phi\\
\mathcal{E}_{\mathrm{circ},2}= & \Delta S_{\mathrm{circ}}\circ \Phi-\Delta \mathcal{S}_{\mathrm{circ}}\\
\mathcal{E}_{\mathrm{circ},3}=& -(\delta^{\mathrm{u}}-\delta^{\mathrm{s}})\phi_\theta\\
\mathcal{E}_{\mathrm{circ},4}=&-(\tilde{L}-\tilde{L}_\mathrm{circ})+(\tilde{L}-\tilde{L}_\mathrm{circ})\circ \Phi
\end{split}
\]
On one hand, Theorem \ref{thm:differencegenfunctcommon} implies that $\lVert \mathcal{E}_{\mathrm{circ},1}\rVert_{1,\rho}\lesssim \zeta G_{0}^{-8}$. On the other hand, the  estimates 
\[
\lVert \partial_v\Delta S_{\mathrm{circ}}\rVert_{3/2,\rho},\lVert \partial_\beta \Delta S_{\mathrm{circ}}\rVert_{1/2,\rho}\lesssim G_{0}^{-4},
\]
which can be deduced from Theorem \ref{thm:differencegenfunctcommon}, and the estimates $\lVert \phi_v-\phi_{v,\mathrm{circ}}\rVert_{0,\rho}, \lVert \phi_\theta-\phi_{\theta,\mathrm{circ}}\rVert_{1/2,\rho}\lesssim \zeta G_{0}^{-4}$ obtained in Lemma \ref{lem:straighteningcv} imply that $\lVert \mathcal{E}_{\mathrm{circ},2}\rVert_{1/2,\rho}\lesssim \zeta  G_{0}^{-13/2}$. For the third term, since $|G^{\mathrm{u}}-G^{\mathrm{s}}|\lesssim \zeta  G_{0}^{-4}$, the estimate $\lVert \phi_v\rVert_{0,\rho}\lesssim  G_{0}^{-4}$ in Lemma  \ref{lem:straighteningcv} shows that $\lVert \mathcal{E}_{\mathrm{circ},3}\rVert_{1/2,\rho}\lesssim \zeta  G_{0}^{-8}$. Finally, since $\lVert \partial_v (\tilde{L}-\tilde{L}_{\mathrm{circ}})\rVert_{3/2,\rho},\lVert \partial_\beta (\tilde{L}-\tilde{L}_{\mathrm{circ}})\rVert_{1/2,\rho}\lesssim \zeta  G_{0}^{-4}$, which can be deduced from   Theorem \ref{thm:differencegenfunctcommon}, we obtain that $\lVert \mathcal{E}_{\mathrm{circ},4}\rVert_{1/2,\rho}\lesssim \zeta  G_{0}^{-13/2}$.
\end{proof}


\subsection{The critical points of the function $\Delta \mathcal{S}$}\label{sec:cricitcalpointsgenfunct}

In this section we use Lemma \ref{lem:mainMelnikov} and  Proposition \ref{prop:approxgenfunctionbyMelnikov} to  provide a proof of Theorem \ref{thm:mainthmcriticalpoints} and Proposition \ref{prop:sndthmcriticalpoints}.

\begin{rem}
Since we always assume that $G^{\mathrm{u}}, G^{\mathrm{s}}\in\mathbb{G}_\rho(G_0)$, all the errors, which a priori depend on both $G^{\mathrm{u}},G^{\mathrm{s}}$ can be estimated in terms of the value of  $G_0$ alone
\end{rem}

\begin{proof}[Proof of Theorem \ref{thm:mainthmcriticalpoints}]
Throughout the proof we will use the following notation. Let 
\[
K=\{(\theta,\sigma,G^{\mathrm{u}},G^{\mathrm{s}})\in \mathbb{T}^2_\rho\times(\mathbb{G}_\rho(G_0))^2\}.
\]
We look for zeros of the function
\begin{equation}\label{eq:mapF}
F(\theta,\sigma,G^{\mathrm{u}},G^{\mathrm{s}})\equiv (\partial_\sigma \Delta \mathcal{S}, \partial_\theta \Delta \mathcal{S})(\theta,\sigma;G^{\mathrm{u}}, G^{\mathrm{s}}),
\end{equation}
which are of the form $(\theta,\sigma,G^{\mathrm{u}},G^{\mathrm{s}})=(\theta,\sigma_\pm(\theta,G^{\mathrm{u}}),G^{\mathrm{u}},G^{\mathrm{s}}(\theta,G^{\mathrm{u}}))$.

In order to obtain asymptotic formulas for the critical points, we divide the proof in two steps. First we study the existence of critical points $\sigma_{\pm,\mathrm{circ}}(\theta,G^{\mathrm{u}})$  of the function $\Delta \mathcal{S}_{\mathrm{circ}}(\sigma,\theta,G^{\mathrm{u}})=\Delta \mathcal{\widehat{S}}_{\mathrm{circ}}(\sigma-\theta;G^{\mathrm{u}})$,  and then prove the existence of critical points of the function $F$ which are $\zeta$ close to $(\sigma,G^{\mathrm{s}})=(\sigma_{\pm,\mathrm{circ}}(\theta,G^{\mathrm{u}}) ,G^{\mathrm{u}})$.  Since for all $(\theta,\sigma,G^{\mathrm{u}},G^{\mathrm{s}})\in K$ (see \eqref{eq:firstharmonicMelnikov})
\[
\partial_\sigma L_{\mathrm{circ}}(\sigma-\theta;G^{\mathrm{u}})=\mu(1-\mu)(1-2\mu)\sqrt{\frac{\pi}{2(G^{\mathrm{u}})^3}}\exp(-(G^{\mathrm{u}})^3/3) \sin (\sigma-\theta)+ \mathcal{O}(G_0^{-5/2} \exp(-G_0^3/3))
\]
and
\[
|\partial_\sigma \Delta S_{\mathrm{circ}}(\sigma,\theta;G^{\mathrm{u}})-\partial_\sigma L_{\mathrm{circ}}(\sigma-\theta; G^{\mathrm{u}}) |\lesssim G_0^{-7/2} \exp(-G_0^3/3),
\]
a direct application of the implicit function theorem shows that there exist nondegenerate critical points 
\begin{equation}\label{eq:sigmatildecirc}
\sigma_{+,\mathrm{circ}}(\theta,G^{\mathrm{u}})=\theta+\mathcal{O}(G_0^{-1})\qquad\qquad \sigma_{-,\mathrm{circ}}(\theta,G^{\mathrm{u}})=\theta+\pi+\mathcal{O}(G_0^{-1})
\end{equation}
of the function $\partial_\sigma \Delta \mathcal{S}_{\mathrm{circ}}(\sigma,\theta,I^ u)$. Therefore, to analyze the zeros of $F$, we write 
\[
\begin{split}
\partial_\sigma \Delta \mathcal{S}= &  \partial_\sigma \Delta \mathcal{S}_{\mathrm{circ}} +\mathcal{E}_\sigma\\
\partial_\theta \Delta \mathcal{S}=& G_0^{-1}(G^{\mathrm{u}}-G^{\mathrm{s}})+\partial_\theta (L-L_{\mathrm{circ}}) +\mathcal{E}_{\theta,1}+\mathcal{E}_{\theta,2}\\
\end{split}
\]
with
\begin{equation}\label{eq:CPnotationerrors}
\begin{split}
\mathcal{E}_\sigma=&\partial_\sigma ( \Delta \mathcal{S}- \Delta \mathcal{S}_{\mathrm{circ}})\\
\mathcal{E}_{\theta,1}=&\partial_\theta (\Delta \mathcal{S}- \Delta \mathcal{S}_{\mathrm{circ}}- ( G_0^{-1}(G^{\mathrm{u}}-G^{\mathrm{s}})\theta+ L-L_{\mathrm{circ}}))\\
\mathcal{E}_{\theta,2}=&\partial_\theta \Delta \mathcal{S}_{\mathrm{circ}}
\end{split}
\end{equation}
The existence of nondegenerate zeros of the function $F(\theta,\sigma,G^{\mathrm{u}},G^{\mathrm{s}})$ will be a direct consequence of the asymptotic formulas in Lemma \ref{lem:mainMelnikov}, the estimates in  Lemma \ref{lem:summaryerrorsapprox} and the implicit function theorem. The first step is to estimate the error terms $\mathcal{E}_\sigma, \mathcal{E}_{\theta,1}$ and  $\mathcal{E}_{\theta,2}$. We write $\mathcal{E}_\sigma=\partial_\sigma (L-L_{\mathrm{circ}})+\mathcal{E}_{\mathrm{circ}}$ where $\mathcal{E}_{\mathrm{circ}}$ has been defined in \eqref{eq:finalerrorapproxcircular}. Therefore, the  asymptotic formulas in Lemma \ref{lem:mainMelnikov}, the fact that $\mathcal{E}_{\mathrm{circ}} \in\mathrm{Ker}\mathcal{L}$ and the estimates in  Lemma \ref{lem:summaryerrorsapprox} imply that 
\[
|\mathcal{E}_\sigma|\lesssim \zeta G_0^{1/2}  \exp(-G_0^3/3).
\]
The estimate in  Lemma \ref{lem:summaryerrorsapprox} implies that 
\[
|\mathcal{E}_{\theta,1}|\lesssim \zeta G_0^{-13/2}.
\]
Moreover, since
\[
|\partial^2_{\sigma\theta} \Delta \mathcal{S}_{\mathrm{circ}} |\lesssim  |G^{\mathrm{u}}|^{-3/2} \exp(-G_0^3/3),
\]
if we define (for a sufficiently large, but fixed, $C>0$)
\[
K_\pm=\{(\theta,\sigma,G^{\mathrm{u}},G^{\mathrm{s}})\in K\colon |\sigma-\sigma_{\pm,\mathrm{circ}}(\theta,G^{\mathrm{u}})|\leq C \zeta G_0^{2} \}
\]
we obtain that 
\[
\sup_{(\theta,\sigma)\in K_\pm}|\mathcal{E}_{\theta,2}|\lesssim \zeta |G^{\mathrm{u}}|^{1/2} \exp(-G_0^3/3).
\]
Therefore, in view of the asymptotic expression in Lemma \ref{lem:mainMelnikov} 
\begin{equation}\label{eq:CPderivativesMelnikov}
\partial_\theta (L-L_\mathrm{circ})(\theta,G^{\mathrm{u}})=\mu(1-\mu)(1-2\mu) \frac{15\pi \zeta}{8(G^{\mathrm{u}})^6} \sin\theta +\mathcal{O}(\zeta G_0^{-13/2}),
\end{equation}
we take (we use that $G^{\mathrm{u}}\sim G_0)$
\begin{equation}\label{eq:CPfirstapprox}
\tilde{\sigma}_{\pm}(\theta,G^{\mathrm{u}})=\sigma_{\pm,\mathrm{circ}}(\theta,G^{\mathrm{u}})\qquad\qquad \hat{G}^{\mathrm{s}}_\pm(\theta,G^{\mathrm{u}})=G^{\mathrm{u}}+ G^{\mathrm{u}}\partial_\theta (L-L_\mathrm{circ})(\theta,\tilde{\sigma}_\pm(\theta,G^{\mathrm{u}}),G^{\mathrm{u}}),
\end{equation}
where $\sigma_{\pm,\mathrm{circ}}(\theta,G^{\mathrm{u}})$ are defined in \eqref{eq:sigmatildecirc}, as approximate solutions. Indeed, taking into account the estimates for $\mathcal{E}_\sigma,\mathcal{E}_{\theta,1}$ and $\mathcal{E}_{\theta,2}$ defined in\eqref{eq:CPnotationerrors}, for all $(\theta,\tilde{\sigma}_{\pm},G^{\mathrm{u}},\hat{I}^{\mathrm{s}}_\pm)\in K$
\begin{equation}\label{eq:CPerrorsfirstapprox}
F(\theta,\tilde{\sigma}_{\pm},G^{\mathrm{u}},\tilde{G}^{\mathrm{s}})=\left( \mathcal{O}(\zeta G_0^{-1}\exp(-G_0^3/3),\mathcal{O}(\zeta G_0^{-13/2})\right),
\end{equation}
and these estimates extend to $(\theta,\sigma,G^{\mathrm{u}},G^{\mathrm{s}})\in \tilde{K}_\pm\equiv \{(\theta,\sigma,G^{\mathrm{u}},G^{\mathrm{s}})\in K_\pm \colon |G^{\mathrm{s}}-\hat{G}^{\mathrm{s}}_\pm|\leq \zeta G_0^{-5} \}$.
Denote by $A\pm$ the differential of the map $(\sigma,G^{\mathrm{s}})\mapsto F(\theta,\sigma,G^{\mathrm{u}},G^{\mathrm{s}})$ evaluated at $(\theta,\tilde{\sigma}_{\pm},G^{\mathrm{u}},\hat{G}^{\mathrm{s}}_\pm)$. It is an straightforward but tedious computation to check that the asymptotic expression in Lemma \ref{lem:mainMelnikov}  and the estimates in \ref{lem:summaryerrorsapprox} imply
\[
\begin{split}
A_\pm= &\left( \begin{array} {cc}
\pm 2 \mu(1-\mu)L_{1,1} & 0\\
0 & -1\\ \end{array} \right) \\
&+ \left( \begin{array} {cc}
\mathcal{O}( (G_0^{-7/2}+\zeta G_0^{1/2})\exp(-G_0^3/3))& \mathcal{O}(G_0^{1/2}\exp(-G_0^3/3))\\
 \mathcal{O}( G_0^{-3/2}\exp(-G_0^3/3) )& \mathcal{O}(G_0^{-8})\\ \end{array} \right).
\end{split}
\] 
Therefore, a direct application of the Implicit function theorem, together with the fact that (see \eqref{eq:firstharmonicMelnikov})
\[
|L_{1,1}|\sim  G_{0}^{-3/2} \exp(-G_0^3/3)),
\]
yields the existence of $G_0$ and
\begin{equation}\label{eq:CPsolutions}
\sigma_\pm(\theta,G^{\mathrm{u}})=\tilde{\sigma}_\pm(\theta,G^{\mathrm{u}})+\mathcal{O}(\zeta G_0^2)\qquad\qquad \tilde{G}^{\mathrm{s}}_\pm(\theta,G^{\mathrm{u}})=\hat{G}^{\mathrm{s}}_\pm(\theta,G^{\mathrm{u}})+\mathcal{O}(\zeta G_{0}^{-13/2})
\end{equation}
such that, for all $(\theta,G^{\mathrm{u}})\in  \Lambda_{\rho}=\mathbb{T}_\rho\times \Lambda_{\rho,I}$, we have
\[
F(\theta,\sigma_\pm(\theta,G^{\mathrm{u}}),G^{\mathrm{u}},\tilde{G}_\pm^{\mathrm{s}}(\theta,G^{\mathrm{u}}))=0.\qedhere\
\]
\end{proof}

It will be convenient for the proofs of Theorems \ref{thm:mainthmscattmaps} and \ref{thm:mainthmscattmaps2}, which will be given in Section \ref{sec:asymptformscattmap}, to state now the following more technical version of Theorem \ref{thm:mainthmcriticalpoints}, which includes the asymptotic formulas for the functions $\sigma_\pm(\theta,G^{\mathrm{u}}),\tilde{G}^{\mathrm{s}}_\pm(\theta,G^{\mathrm{u}})$  obtained in the proof of Theorem \ref{thm:mainthmcriticalpoints} above.

\begin{lem}\label{lem:technicallemmacriticalpoints}
Let $G_0$ be large enough, let  $(\theta,G^{\mathrm{u}})\in \mathbb{T}_\rho\times\mathbb{G}_\rho(G_0)$ and let
\[
(\theta,G^{\mathrm{u}})\mapsto(\sigma_\pm(\theta,G^{\mathrm{u}}),\tilde{G}^{\mathrm{s}}_\pm(\theta,G^{\mathrm{u}}))
\]
be the real analytic functions satisfying
\[
\partial_\sigma \Delta \mathcal{S} (\sigma_\pm(\theta,G^{\mathrm{u}}),\theta;G^{\mathrm{u}},\tilde{G}^{\mathrm{s}}_\pm(\theta,G^{\mathrm{u}}))=0\qquad\qquad \partial_\theta \Delta \mathcal{S} (\sigma_\pm(\theta,G^{\mathrm{u}}),\theta;G^{\mathrm{u}},\tilde{G}^{\mathrm{s}}_\pm(\theta,G^{\mathrm{u}}))=0,
\]
which were obtained in Theorem \ref{thm:mainthmcriticalpoints}. Then, for all $(\theta,G^{\mathrm{u}})\in \mathbb{T}_\rho\times\mathbb{G}_\rho(G_0)$ we have
\[
\sigma_+=\theta+\mathcal{O}(G_0^{-1}),\qquad\sigma_-=\theta+\pi+\mathcal{O}(G_0^{-1})
\]
and
\[
\tilde{G}^{\mathrm{s}}_\pm(\theta,G^{\mathrm{u}})= G^{\mathrm{u}}+\partial_\theta \mathcal{L}_\pm(\theta,G^{\mathrm{u}})+ \mathcal{O}(\zeta G_0^{-13/2}),
\]
where $\mathcal{L}_\pm(\theta,G^{\mathrm{u}})$  is the Melnikov potential defined in \eqref{eq: DefinitionredMelnikovPotential}.
\end{lem}

\begin{proof}
The result for $\sigma_\pm$ is deduced from \eqref{eq:sigmatildecirc} and \eqref{eq:CPsolutions}. We only have to prove the result for $\tilde{G}^{\mathrm{s}}_\pm$. In the proof of Theorem \ref{thm:mainthmcriticalpoints} above, we have obtained that (see \eqref{eq:CPfirstapprox} and \eqref{eq:CPsolutions})
\[
\tilde{G}^{\mathrm{s}}_\pm(\theta,G^{\mathrm{u}})= G^{\mathrm{u}}+G^{\mathrm{u}}\partial_\theta (L-L_{\mathrm{circ}})(\theta,\sigma_\pm(\theta,G^{\mathrm{u}});G^{\mathrm{u}})+ \mathcal{O}(\zeta G_0^{-13/2}).
\]
Let $\tilde{\sigma}_+(\theta)=\theta$ and $\tilde{\sigma}_-(\theta)=\theta+\pi$. On one hand, since 
\[
|\sigma_\pm- \tilde{\sigma}_\pm|\lesssim G_0^{-1},
\]
from the asymptotic expression in Lemma \ref{lem:mainMelnikov}, we obtain that, for all $(\theta,G^{\mathrm{u}})\in \Lambda_{\rho}$,
\[
|\partial_\theta  (L-L_{\mathrm{circ}})(\theta,\tilde{\sigma}_\pm(\theta);G^{\mathrm{u}})-\partial_\theta  (L-L_{\mathrm{circ}})(\theta,\sigma_\pm(\theta,G^{\mathrm{u}});G^{\mathrm{u}})|\lesssim \zeta G_0^{1/2} \exp(-G_0^3/3)
\]
and,
\[
|\partial_\sigma  (L-L_{\mathrm{circ}})\partial_\theta \tilde{\sigma}_\pm (\theta,\tilde{\sigma}_\pm(\theta);G^{\mathrm{u}}) |\lesssim \zeta G_0^{1/2} \exp(-G_0^3/3),
\]
and the lemma follows from the fact that $L_{\mathrm{circ}}(\tilde{\sigma}_\pm(\theta)-\theta;G^{\mathrm{u}})$ does not depend on $\theta$.
\end{proof}

We now finish this section with the proof of Proposition \ref{prop:sndthmcriticalpoints}.

\begin{proof}[Proof of Proposition  \ref{prop:sndthmcriticalpoints}]
We fix $G_*$ large enough, let $0\leq \zeta\leq (G_*+R)^{-3}$ and consider pairs $G^{\mathrm{u}},G^{\mathrm{s}}\in [G_*,G_*+R]$. For $G^{\mathrm{u}},G^{\mathrm{s}}$ such that
\[
|G^{\mathrm{s}}-G^{\mathrm{u}}|\leq \frac{\mu (1-\mu)(1-2\mu) 15\pi \zeta}{16 |G^{\mathrm{u}}|^{5}},
\]
we define
\[
\tilde{\theta} (G^{\mathrm{s}},G^{\mathrm{u}})=\sin^{-1} \left( \frac{8(G^{\mathrm{u}})^5(G^{\mathrm{s}}-G^{\mathrm{u}}) } {\mu (1-\mu)(1-2\mu) 15\pi \zeta }\right)\qquad\qquad\tilde{\sigma}_{\pm}(G^{\mathrm{u}},G^{\mathrm{s}})=\sigma_{\pm,\mathrm{circ}}(\tilde{\theta}(G^{\mathrm{s}},G^{\mathrm{u}}),G^{\mathrm{u}}).
\]
where $\sigma_{\pm,\mathrm{circ}}$ are defined in \eqref{eq:sigmatildecirc}. Denoting by $\widetilde{A}_\pm$ the differential of the map $(\sigma,\theta)\mapsto F(\theta,\sigma,G^{\mathrm{u}},G^{\mathrm{s}})$, where $F$ is defined in \eqref{eq:mapF}, evaluated at $( \tilde{\theta}(G^{\mathrm{u}},G^{\mathrm{s}}),\tilde{\sigma}_\pm(G^{\mathrm{u}},G^{\mathrm{s}}),G^{\mathrm{u}},G^{\mathrm{s}})$,
we obtain that 
\[
\begin{split}
\widetilde{A}_\pm=& \left( \begin{array} {cc}
\pm \mu(1-\mu)L_{1,1}& 0\\
0 &  \mu(1-\mu)(1-2\mu) \frac{15\pi \zeta}{8G^5} \cos\tilde{\theta}\\ \end{array} \right) \\
&+ \left( \begin{array} {cc}
\mathcal{O}( (G_0^{-7/2}+\zeta G_0^{3/2})\exp(-G_0^3/3))& \mathcal{O}( G_0^{-3/2}\exp(-G_0^3/3))\\
 \mathcal{O}( G_0^{-3/2}\exp(-G_0^3/3))& \mathcal{O}(\zeta G_0^{-13/2})\\ \end{array} \right),
\end{split}
\] 
and again, it follows from a direct application of the Implicit function theorem the existence of a value $I_*$ (which might be different from the one obtained in the proof of Theorem \ref{thm:mainthmcriticalpoints}) and functions
\[
\hat{\sigma}_\pm(G^{\mathrm{u}},G^{\mathrm{s}})=\tilde{\sigma}_{\pm}(\theta,G^{\mathrm{u}})=\mathcal{O}( \zeta G_0^{2})\qquad\qquad \hat{\theta}_\pm(G^{\mathrm{u}},G^{\mathrm{s}})=\tilde{\theta}(G^{\mathrm{u}},G^{\mathrm{s}})+\mathcal{O}(\zeta G_0^{-13/2})
\]
such that 
\[
F(\theta_\pm(G^{\mathrm{u}},G^{\mathrm{s}}),\hat{\sigma}_\pm(G^{\mathrm{u}},G^{\mathrm{s}}),G^{\mathrm{u}},G^{\mathrm{s}})=0
\]
for all 
\[
(G^{\mathrm{u}},G^{\mathrm{s}})\in \left\{(G^{\mathrm{u}},G^{\mathrm{s}})\in [G_*,G_*+R]^2\ |G^{\mathrm{s}}-G^{\mathrm{u}}|\leq \frac{\mu (1-\mu)(1-2\mu) 15\pi \zeta }{16 |G^{\mathrm{u}}|^{5}}\right\}.\qedhere\
\]
\end{proof}


\section{The generating functions of the scattering maps}\label{sec:proofgeneratingscatt}

As explained in Section \ref{sec:constrscattmaps}, Theorem \ref{thm:mainthmcriticalpoints} implies the existence of two scattering maps $\mathbb{P}_\pm :\mathbb{A}_\pm\subset\mathcal{P}_\infty^*\to \mathcal{P}_\infty^*$ (see \eqref{eq:defnscatteringmaps}). In this section, we provide the rather technical proof of Theorem \ref{prop:generatingfunctscattmap}, in which we prove the existence (and obtain an explicit expression) of a generating function for each of the scattering maps $\mathbb{P}_\pm$.

\begin{proof}[Proof of Theorem \ref{prop:generatingfunctscattmap}]
Consider the time $T$ map $\phi^T_{H_{\mathrm{pol}}}$ of the Hamiltonian $H_{\mathrm{pol}}$ introduced in \eqref{eq: Hamiltonian in polar coordinates}. The transformation $\phi^T_{H_{\mathrm{pol}}}$ is exact symplectic and therefore there exists a function $P^T:M_{\mathrm{pol}}\to\mathbb{R}$ such that $\mathrm{d} P^T=(\phi^T_{H_{\mathrm{pol}}})^*\lambda-\lambda_{\mathrm{pol}}$. The function (it is defined modulo constants) $P^T$ is known as the \textit{primitive function} associated to the exact symplectic map $\phi^T_{H_{\mathrm{pol}}}$. It is a standard computation (see the proof of Theorem 13 in \cite{MR2383896}) that (up to a constant)
\[
P^T=\int_{0}^T (i_{H_{\mathrm{pol}}} \lambda_{\mathrm{pol}} +H_{\mathrm{pol}})\circ \phi^\tau_{H_{\mathrm{pol}}} \mathrm{d}\tau,
\]
where $i_{H_{\mathrm{pol}}}\lambda_{\mathrm{pol}}$ denotes the contraction of the one form $\lambda_{\mathrm{pol}}=y\mathrm{d}r+G\mathrm{d}\alpha+E\mathrm{d}t$  with the vector field associated to the Hamiltonian $H_{\mathrm{pol}}$. Now we obtain an expression for the primitive function associated to the (exact symplectic) scattering maps $\mathbb{P}_\pm$. The natural candidate to consider as primitive function of the scattering map $\mathbb{P}_\pm$  defined in \eqref{eq:defnscatteringmaps} would be (see Theorem 13 in \cite{MR2383896}) to consider the function $P^T$ restricted to $\Gamma_\pm$, which is given by
\begin{equation}\label{eq:improperprimitive}
\tilde{P}_\pm (\alpha^{\mathrm{u}},G^{\mathrm{u}})= \lim_{T\to\infty} \int_{-T}^T i_{H_{\mathrm{pol}}} \lambda_{\mathrm{pol}}\circ \phi^\tau_{H_{\mathrm{pol}}}\circ (\Omega^{\mathrm{u}})_\pm^{-1} (\alpha^{\mathrm{u}}, G^{\mathrm{u}})\ \mathrm{d}\tau,
\end{equation}
where we have already taken into account that $H_{\mathrm{pol}}\circ \phi^\tau_{H_{\mathrm{pol}}}\circ (\Omega^{\mathrm{u}})_\pm^{-1} =0$ and that the dynamics in $\mathcal{P}_\infty$ is trivial. However, the improper integral \eqref{eq:improperprimitive} is not convergent (Theorem 13 in \cite{MR2383896} is proved for scattering maps associated to normally hyperbolic invariant manifolds, however, in the present case the rate of contraction/expansion along the stable/unstable leaves of $\mathcal{P}_\infty$ is only polynomial). Indeed, for $\tau\to\pm\infty$ (see Lemma \ref{lem:uperturbedhomoclinic} and Proposition \ref{prop:extensiondecaygenfunctions})
\[
\begin{split}
i_{H_{\mathrm{pol}}} \lambda_{\mathrm{pol}}\circ \phi^\tau_{H_{\mathrm{pol}}}\circ (\Omega^{\mathrm{u}})_\pm^{-1} (\alpha^{\mathrm{u}}, G^{\mathrm{u}}) = & \left(y^2+\frac{G^2}{r^2}\right)\circ\phi^\tau_{H_{\mathrm{pol}}}\circ (\Omega^{\mathrm{u}})_\pm^{-1} (\alpha^{\mathrm{u}}, G^{\mathrm{u}})\\
\sim & y_{\mathrm{h}}^2(\tau)+\frac{1}{(G^{\mathrm{u}})^2r_{\mathrm{h}}^2(\tau)}= \frac{2}{r_{\mathrm{h}}(\tau)}\sim \tau^{-2/3}.
\end{split}
\] 
Therefore, we consider instead the renormalized primitive function $P_\pm:  \mathcal{P}_\infty^* \to \mathbb{R}$, defined as
\begin{equation}\label{eq:primitive}
\begin{split}
 P_\pm (\alpha^{\mathrm{u}},G^{\mathrm{u}})=& \int_{\mathbb{R}} \left(i_{H_{\mathrm{pol}}} \lambda_{\mathrm{pol}}\circ \phi^\tau_{H_{\mathrm{pol}}}\circ (\Omega^{\mathrm{u}})_\pm^{-1} (\alpha^{\mathrm{u}}, G^{\mathrm{u}})- G^{\mathrm{u}}Q'(\tau)\right) \mathrm{d}\tau,\\
\end{split}
\end{equation}
where $Q(u)$ is any function satisfying $Q'(u)=2/r_{\mathrm{h}}(u)$. We now want to express the integrand in \eqref{eq:primitive} in terms of the parametrizations \eqref{eq:stableparam} and \eqref{eq:unstableparam}.  To that end we notice that
\begin{align*}
\begin{split}
i_{H_{\mathrm{pol}}} \lambda_{\mathrm{pol}} \circ \phi^\tau_{H_{\mathrm{pol}}}\circ (\Omega^{\mathrm{u}})_\pm^{-1}=& i_{H_{\mathrm{pol}}} \lambda_{\mathrm{pol}} \circ \phi^\tau_{H_{\mathrm{pol}}}\circ\eta_{G^{\mathrm{u}}}\circ \Phi_{\mathrm{h}}\circ (\Omega^{\mathrm{u}}_\pm\circ \eta_{G^{\mathrm{u}}}\circ\Phi_{\mathrm{h}})^{-1}\\
=&i_{H_{\mathrm{pol}}} \lambda_{\mathrm{pol}} \circ\eta_{G^{\mathrm{u}}}\circ \Phi_{\mathrm{h}}\circ \phi_{\widehat{H}}^\tau\circ (\Omega^{\mathrm{u}}_\pm\circ \eta_{G^{\mathrm{u}}}\circ\Phi_{\mathrm{h}})^{-1}.\\
\end{split}
\end{align*}
where we have written $\widehat{H}=(G^{\mathrm{u}})^{-2}H$. Then, it follows from the definition of $Q$ and the change of variables $\Phi_h$ and the scaling $\eta_{G^{\mathrm{u}}}$ that ,
\[
i_{H_{\mathrm{pol}}} \lambda_{\mathrm{pol}} \circ \phi^\tau_{H_{\mathrm{pol}}}\circ (\Omega^{\mathrm{u}})_\pm^{-1}=G^{\mathrm{u}}\  i_{\widehat{H}} (\lambda +\mathrm{d}Q+\mathrm{d}\beta)\circ \phi_{\widehat{H}}^\tau\circ (\Omega^{\mathrm{u}}_\pm\circ\eta_{G^{\mathrm{u}}}\circ \Phi_{\mathrm{h}})^{-1}\\
\]
where  $\lambda=Y\mathrm{d}u+G\mathrm{d}\beta+E\mathrm{d}t$. Yet, the parametrization \eqref{eq:stableparam} is not defined at $u=0$ so $\phi_{H}^\tau\circ (\Omega^{\mathrm{u}}_\pm\circ\eta_{G^{\mathrm{u}}}\circ \Phi_{\mathrm{h}})^{-1}$ might not be defined for all $\tau\in\mathbb{R}$ . The rather simple solution to this annoyance goes as follows. By Cauchy's initial value theorem,   the function $\phi^\tau_{H_{\mathrm{pol}}}\circ (\Omega^{\mathrm{u}})_\pm^{-1}$  can be extended analytically to a real analytic function, which, by abuse of notation, we denote as $\tau\mapsto \phi^\tau_{H_\mathrm{pol}}\circ (\Omega^{\mathrm{u}})_\pm^{-1}$,  defined in a complex neighborhood of $\mathbb{R}$ and such that
\[
\frac{\mathrm{d}}{\mathrm{d}\tau}\left(  \phi^\tau_{H_\mathrm{pol}}\circ (\Omega^{\mathrm{u}})_\pm^{-1} \right) =X_{H_\mathrm{pol}}\circ  \phi^\tau_{H_\mathrm{pol}}\circ (\Omega^{\mathrm{u}})_\pm^{-1},
\]
where $X_{H_\mathrm{pol}}$ is the vector field associated to the Hamiltonian \eqref{eq: Hamiltonian in polar coordinates}. Therefore, we can change the integration path in the definition of $P_\pm$ to a complex path $\gamma\subset \mathbb{C}$ on the domain of analyticity of the function $\tau\mapsto \phi^\tau_{H_{\mathrm{pol}}}\circ (\Omega^{\mathrm{u}})_\pm^{-1}$ and such that $0\notin\gamma$. Moreover, we can choose $\gamma$ to also satisfy that ($\pi_u$ denotes the projection onto the $u$ component)
\begin{equation}\label{eq:unotingamma}
u(\tau;\alpha^{\mathrm{u}},G^{\mathrm{u}})\equiv \pi_u \left(\eta^{-1}_{G^{\mathrm{u}}}\circ\Phi_{\mathrm{h}}^{-1}\circ\phi^\tau_{H_{\mathrm{pol}}}\circ (\Omega^{\mathrm{u}})_\pm^{-1}(\alpha^{\mathrm{u}},G^{\mathrm{u}}) \right)\neq 0\qquad\qquad \forall \tau\in \gamma .
\end{equation}
This is possible since  away from $u\neq0$
\[
\frac{\mathrm{d}}{\mathrm{d}\tau} \pi_u \circ\phi^\tau_{H} (u,\beta,t)=\partial_Y H\circ \phi^\tau_{H}(u,\beta,t)=1+\mathcal{O}( (G^{\mathrm{u}})^{-5}),
\]
so by taking $\gamma$ which does not enter a $\mathcal{O}( (G^{\mathrm{u}})^{-5})$ neighborhood of $\tau=0$ we can guarantee that \eqref{eq:unotingamma} holds. Then, 
\begin{equation}\label{eq:primitivemodifiedpath}
\begin{split}
 P_\pm(\alpha^{\mathrm{u}},G^{\mathrm{u}})=& \int_{\gamma} i_{H_{\mathrm{pol}}} \lambda_{\mathrm{pol}} \circ \phi^\tau_{H_{\mathrm{pol}}}\circ (\Omega^{\mathrm{u}})_\pm^{-1} (\alpha^{\mathrm{u}}, G^{\mathrm{u}})-G^{\mathrm{u}}Q'(\tau)\ \mathrm{d}\tau \\
=&G^{\mathrm{u}}\int_{\gamma}i_{\widehat{H}} (\lambda +\mathrm{d}Q+\mathrm{d}\beta)\circ \phi^\tau_{\widehat{H}}\circ (\Omega^{\mathrm{u}}_\pm\circ\eta_{G^{\mathrm{u}}}\circ \Phi_{\mathrm{h}})^{-1} (\alpha^{\mathrm{u}}, G^{\mathrm{u}})-Q'(\tau) \mathrm{d}\tau
\end{split}
\end{equation}
is well defined. Moreover,
\[
\begin{split}
\lim_{T\to\infty}  \int_{-T}^T i_{\widehat{H}} \mathrm{d}Q\ &\circ \phi^\tau_{\widehat{H}}\circ (\Omega^{\mathrm{u}}_\pm\circ\eta_{G^{\mathrm{u}}}\circ \Phi_{\mathrm{h}})^{-1}(\alpha^{\mathrm{u}},G^{\mathrm{u}})-Q'(\tau)\ \mathrm{d}\tau\\
=&\lim_{T\to\infty}  \int_{-T}^T  \frac{\mathrm{d}}{\mathrm{d}\tau}( Q\circ \phi^\tau_{H}\circ (\Omega^{\mathrm{u}}_\pm\circ\eta_{G^{\mathrm{u}}}\circ \Phi_{\mathrm{h}})^{-1})(\alpha^{\mathrm{u}},G^{\mathrm{u}}) -Q'(\tau)\ \mathrm{d}\tau\\
=&\lim_{T\to\infty} (Q(u(T;\alpha^{\mathrm{u}}, G^{\mathrm{u}}))-Q(T))+ (Q(-T)-Q(u(-T;\alpha^{\mathrm{u}}, G^{\mathrm{u}}))).
\end{split}
\]
We claim that this limit is zero. Indeed, from the definition of the Hamiltonian $H$ and Proposition \ref{prop:extensiondecaygenfunctions} we observe, that for large values of $u$,
\[
\begin{split}
\frac{\mathrm{d}}{\mathrm{d}\tau} u(\tau;\alpha^{\mathrm{u}},G^{\mathrm{u}})=&\frac{\mathrm{d}}{\mathrm{d}\tau} \pi_u \circ\phi^\tau_ H \circ(\Omega^{\mathrm{u}}_\pm\circ\eta_{G^{\mathrm{u}}}\circ \Phi_{\mathrm{h}})^{-1}(\alpha^{\mathrm{u}},G^{\mathrm{u}})=\partial_Y H\circ \phi^\tau_{H}(\Omega^{\mathrm{u}}_\pm\circ\eta_{G^{\mathrm{u}}}\circ \Phi_{\mathrm{h}})^{-1}(\alpha^{\mathrm{u}},G^{\mathrm{u}})\\
=&1+\mathcal{O}(|u(\tau;\alpha^{\mathrm{u}},G^{\mathrm{u}})|^{-2/3}).
\end{split}
\]
So for large $T$ we have
\[
|u(\pm T;\alpha^{\mathrm{u}},G^{\mathrm{u}}) \mp T)|=\mathcal{O}(T^{1/3}).
\]
Moreover, $Q'(\pm T)\sim T^{-2/3}$ for $T\to \infty$ so, by application of the mean value theorem
\[
|Q(u( \pm T;\alpha^{\mathrm{u}}, G^{\mathrm{u}}))-Q( \pm T)|\lesssim Q'(\pm T) |u( \pm T;\alpha^{\mathrm{u}},G^{\mathrm{u}})\mp T|\leq \mathcal{O}(T^{-1/3}).
\]
Therefore, expression \eqref{eq:primitivemodifiedpath} reduces to
\begin{equation}\label{eq:primitive2}
 P_\pm(\alpha^{\mathrm{u}},G^{\mathrm{u}})= G^{\mathrm{u}}\int_{\gamma} (i_{H} \lambda+\mathrm{d}\beta) \circ \phi^\tau_{H}\circ (\Omega^{\mathrm{u}}_\pm\circ\eta_{G^{\mathrm{u}}}\circ \Phi_{\mathrm{h}})^{-1} (\alpha^{\mathrm{u}}, G^{\mathrm{u}}) \mathrm{d}\tau .
\end{equation}
Let now $\gamma^{\mathrm{u}}=\gamma|_{\tau\leq0 }$, $\gamma^{\mathrm{s}}=\gamma|_{\tau\geq 0}$  and introduce the functions
\[
\begin{split}
P^{\mathrm{u}}(u,\beta,t;G^{\mathrm{u}},G^{\mathrm{s}})=& G^{\mathrm{u}}\int_{\gamma^{\mathrm{u}}} i_{H}( \lambda+\mathrm{d}\beta) \circ \phi^\tau_{H}\circ \mathcal{W}^{\mathrm{u}}(u,\beta,t;G^{\mathrm{u}},G^{\mathrm{s}})\mathrm{d}\tau\\
P^{\mathrm{s}}(u,\beta,t;G^{\mathrm{u}},G^{\mathrm{s}})=&G^{\mathrm{u}}\int_{\gamma^{\mathrm{s}}} i_{H}( \lambda+\mathrm{d}\beta) \circ \phi^\tau_{H}\circ\mathcal{W}^{\mathrm{s}}(u,\beta,t;G^{\mathrm{u}},G^{\mathrm{s}})\mathrm{d}\tau,
\end{split}
\]
where $\mathcal{W}^{\mathrm{u},\mathrm{s}}$ are the parametrizations of the invariant manifolds introduced in \eqref{eq:stableparam} and \eqref{eq:unstableparam}. Therefore, recalling the definition of $S^{\mathrm{u,s}}$ in ...
\begin{equation}\label{eq:primitivefunctionsubetat}
\begin{split}
P^{\mathrm{u}}(u,\beta,t; G^{\mathrm{u}},G^{\mathrm{s}})=& G^{\mathrm{u}}\int_{\gamma^{\mathrm{u}}}  \frac{\mathrm{d}}{\mathrm{d}\tau}\left((S^{\mathrm{u}}+ \beta)\circ \phi^\tau_{H}\circ \mathcal{W}^{\mathrm{u}} \right)(u,\beta,t; G^{\mathrm{u}}, G^{\mathrm{s}}) \mathrm{d}\tau\\
=&G^{\mathrm{u}} S^{\mathrm{u}}(u,\beta,t;G^{\mathrm{u}},G^{\mathrm{s}})+ G^{\mathrm{u}}\beta-G^{\mathrm{u}}\alpha^{\mathrm{u}}(u,\beta,t;G^{\mathrm{u}},G^{\mathrm{s}})\\
P^{\mathrm{s}}(u,\beta,t; G^{\mathrm{u}}, G^{\mathrm{s}})=&G^{\mathrm{u}}\int_{\gamma^{\mathrm{s}}}\frac{\mathrm{d}}{\mathrm{d}\tau}\left((S^{\mathrm{s}}+\beta)\circ \phi^\tau_{H}\circ \mathcal{W}^{\mathrm{s}} \right)(u,\beta,t;  G^{\mathrm{u}}, G^{\mathrm{s}}) \mathrm{d}\tau\\
=&-G^{\mathrm{u}}S^{\mathrm{s}}(u,\beta,t;G^{\mathrm{u}},G^{\mathrm{s}})-G^{\mathrm{u}}\beta+ G^{\mathrm{s}}\alpha^{\mathrm{s}}(u,\beta,t; G^{\mathrm{u}},G^{\mathrm{s}}),
\end{split}
\end{equation}
where $\alpha^{\mathrm{u},\mathrm{s}}(u,\beta,t;G^{\mathrm{u}},G^{\mathrm{s}},\zeta)$ denotes the asymptotic value of the $\alpha$ coordinate along the unstable or stable leave of a point in $W^{\mathrm{u}}$ or $W^{\mathrm{s}}$ given by the parametrizations  \eqref{eq:stableparam} and \eqref{eq:unstableparam}. Notice that, in particular,
\[
(P^{\mathrm{u}}+P^{\mathrm{s}})(u,\beta,t;G^{\mathrm{u}},G^{\mathrm{s}})= G^{\mathrm{u}}\Delta S(u,\beta,t;G^{\mathrm{u}},G^{\mathrm{s}})+G^{\mathrm{s}}\alpha^{\mathrm{s}}(u,\beta,t;G^{\mathrm{u}},G^{\mathrm{s}})-G^{\mathrm{u}}\alpha^{\mathrm{u}}(u,\beta,t;G^{\mathrm{u}},G^{\mathrm{s}}).
\]
Let now $(G^{\mathrm{u}},G^{\mathrm{s}})\in\mathcal{R}_G$ (see \eqref{eq:domainactions}) and denote by $\alpha_\pm^{\mathrm{u},\mathrm{s}}(G^{\mathrm{u}},G^{\mathrm{s}})$ the backwards and forward asymptotic value of the $\beta$ component along the heteroclinic orbit which passes through the heteroclinic point $x_\pm=(u,\beta,t,Y,J,E)$ given by 
\[
\begin{split}
x_\pm(G^{\mathrm{u}},G^{\mathrm{s}})=&\mathcal{W}^{\mathrm{u}}\circ\Phi (-(G^{\mathrm{u}})^3\hat{\sigma}_\pm (G^{\mathrm{u}},G^{\mathrm{s}}),\hat{\theta}_\pm(G^{\mathrm{u}},G^{\mathrm{s}}),0;G^{\mathrm{u}}, G^{\mathrm{s}})\\
=&\mathcal{W}^{\mathrm{s}}\circ\Phi (-(G^{\mathrm{u}})^3\hat{\sigma}_\pm (G^{\mathrm{u}},G^{\mathrm{s}}),\hat{\theta}_\pm(G^{\mathrm{u}},G^{\mathrm{s}}),0;G^{\mathrm{u}}, G^{\mathrm{s}}),\\
\end{split}
\]
where $\Phi(v,\theta,t;G^{\mathrm{u}},G^{\mathrm{s}},\zeta)$ is  the change of variables obtained in Proposition \ref{prop:approxgenfunctionbyMelnikov} and $\hat{\sigma}_\pm,\hat{\theta}_\pm$ were obtained in Proposition \ref{prop:sndthmcriticalpoints}. That is, 
\[
\mathbb{P}_\pm (\alpha_\pm^{\mathrm{u}}(G^{\mathrm{u}},G^{\mathrm{s}}),G^{\mathrm{u}})=(\alpha_\pm^{\mathrm{s}}(G^{\mathrm{u}},G^{\mathrm{s}}),G^{\mathrm{s}}).
\]
Then, using expression \eqref{eq:primitivefunctionsubetat}, we obtain that the primitive function in \eqref{eq:primitive2} can be expressed as
\[
\begin{split}
P_\pm (\alpha^{\mathrm{u}}_\pm (G^{\mathrm{u}},G^{\mathrm{s}}), G^{\mathrm{u}})=& (P^{\mathrm{u}}+P^{\mathrm{s}})\circ\Phi_\pm (\hat{\theta}_\pm(G^{\mathrm{u}},G^{\mathrm{s}});G^{\mathrm{u}})\\
=&\mathtt{S}_\pm (G^{\mathrm{u}},G^{\mathrm{s}})+G^{\mathrm{s}}\alpha^{\mathrm{s}}_\pm(G^{\mathrm{u}},G^{\mathrm{s}})-G^{\mathrm{u}} \alpha^{\mathrm{u}}_\pm(G^{\mathrm{u}},G^{\mathrm{s}}),
\end{split}
\]
where $\mathtt{S}_\pm$ is the function defined in \eqref{eq:difngenfunctactions}. The proposition plainly follows from the definition of primitive function of an exact symplectic map. Indeed 
\[
\mathrm{d} \mathtt{S}_\pm=\mathrm{d} P_\pm-G^{\mathrm{s}}\mathrm{d}\alpha_\pm^{\mathrm{s}}-\alpha^{\mathrm{s}}_\pm \mathrm{d}G^{\mathrm{s}}+G^{\mathrm{u}}\mathrm{d}\alpha_\pm^{\mathrm{u}}+\alpha^{\mathrm{u}}_\pm \mathrm{d}G^{\mathrm{u}}=\alpha^{\mathrm{u}}_\pm \mathrm{d}G^{\mathrm{u}}-\alpha^{\mathrm{s}}_\pm \mathrm{d}G^{\mathrm{s}}.\qedhere\
\]
\end{proof}


\section{Asymptotic analysis of the scattering maps}\label{sec:asymptformscattmap}

In this section we prove Theorems \ref{thm:mainthmscattmaps} and \ref{thm:mainthmscattmaps2}. Namely, we establish an asymptotic formula for the scattering maps defined in \eqref{eq:defnscatteringmaps} and for their difference in terms of the reduced Melnikov potentials $\mathcal{L}_\pm$ defined in \eqref{eq: DefinitionredMelnikovPotential}. Let $\Phi_{\mathrm{h}}$ be the change of variables defined in \eqref{eq:Mathieutransf}, consider the function $\tilde{G}^{\mathrm{s}}_\pm(\theta,G^{\mathrm{u}})$ obtained in Theorem \ref{thm:mainthmcriticalpoints} (see also Lemma \ref{lem:technicallemmacriticalpoints}), let $\Phi_\pm$ be the map defined in \eqref{eq:definitionPhi+-} and let $\Omega_\pm^{\mathrm{u}}$ be the wave maps introduced in \eqref{eq:restrwavemaps}.  By the expressions \eqref{eq:usefulexpressionswavemaps} for the wave maps, it follows that for all $(\alpha^{\mathrm{u}},G^{\mathrm{u}})\in \mathcal{P}_\infty^*$, the $G$ coordinate of the scattering map $\mathbb{P}_\pm$ is given by
\[
G^{\mathrm{s}}_\pm(\alpha^{\mathrm{u}},G^{\mathrm{u}})=\tilde{G}^{\mathrm{s}}_\pm\circ  (\Omega^{\mathrm{u}}_\pm\circ\eta_{G_0}\circ \Phi_{\mathrm{h}}\circ \Phi_\pm)^{-1}(\alpha^{\mathrm{u}},G^{\mathrm{u}}).
\]


\subsection{The wave maps and their difference}\label{sec: the difference between the wave maps}

Let 
\begin{equation}\label{eq:straightenunstablefoliation}
\alpha^{\mathrm{u}}=\Theta(u,\beta,t; G^{\mathrm{u}},G^{\mathrm{s}},\zeta)\equiv\beta+\vartheta (u,\beta,t; G^{\mathrm{u}},G^{\mathrm{s}},\zeta)
\end{equation}
be the map which to  $z=(u,\beta,t)$ associates the backward asymptotic $\beta$ component along the leave of the unstable foliation which passes through the point $\mathcal{W}^{\mathrm{u}}(z;G^{\mathrm{u}},G^{\mathrm{s}},\zeta)$. The map \eqref{eq:straightenunstablefoliation} is indeed the inverse of the map $\Psi_\infty=\mathrm{Id}+\psi_\infty$ in  \eqref{eq:finalstraightening} and therefore  (the norm $\llbracket \cdot \rrbracket_{\eta,\nu,\rho}$ is defined in \eqref{eq:bracketnorm})
\begin{equation}\label{eq:estimatevartheta}
\llbracket \vartheta \rrbracket_{1/3,1/2,\rho}\lesssim G_0^{-15/4}.
\end{equation}

\begin{rem}
A more refined analysis of the first step in the iterative process carried out in Sections \ref{sec:firststepiterative} and \ref{sec:iterativeargument} shows that the map $\Psi_\infty=\mathrm{Id}+\psi_\infty$ in  \eqref{eq:finalstraightening} satisfies indeed $\llbracket \psi_\infty \rrbracket_{1/3,1/2,\rho_n}\lesssim G_0^{-4}$ and consequently $\llbracket \vartheta \rrbracket_{1/3,1/2,\rho}\lesssim G_0^{-4}$. Performing this extra step would complicate unnecessarily the iterative process in Sections \ref{sec:firststepiterative} and \ref{sec:iterativeargument}. Therefore we continue our discussion making use of the rougher estimate \eqref{eq:estimatevartheta} which is sufficient for our purposes.
\end{rem}

Let now $\Phi_\pm$ be the change of coordinates defined in \eqref{eq:definitionPhi+-} and $\sigma_\pm, \tilde{G}^{\mathrm{s}}_\pm$ be the functions obtained in Theorem \ref{thm:mainthmcriticalpoints}.  Define now the maps 
\begin{equation}\label{eq:modifiedwavemaps}
\widetilde{\Omega}^{\mathrm{u}}_\pm=\Omega^{\mathrm{u}}_\pm\circ\eta_{G_0}\circ\Phi_{\mathrm{h}}\circ\Phi_\pm
\end{equation}
where $\Omega_\pm^{\mathrm{u}}$ are the backward wave maps introduced in \eqref{eq:restrwavemaps}. By construction
\begin{equation}
\widetilde{\Omega}_{\pm}^{\mathrm{u}}(\theta,G^{\mathrm{u}})=\left(\theta+(\vartheta\circ\Phi_\pm)(\theta,G^{\mathrm{u}})+\phi_\theta(-G_{0}^{3}\sigma_\pm(\theta,G^{\mathrm{u}})),\theta,0;G^{\mathrm{u}}, \tilde{G}^{\mathrm{s}}_\pm(\theta,G^{\mathrm{u}})), \ \ G^{\mathrm{u}}\right)
\end{equation}
where $\Phi=(v+\phi_v,\theta+\phi_\theta,t)$ was defined in Lemma \ref{lem:straighteningcv}. In this section we show that $\widetilde{\Omega}^{\mathrm{u}}_\pm$ is a $\mathcal{O}( G_{0}^{-15/4})$-close to identity map and show that the difference between the map $\widetilde{\Omega}^{\mathrm{u}}_+$ and $\widetilde{\Omega}^{\mathrm{u}}_-$ is exponentially small. To do so we will show that the function 
\begin{equation}\label{eq:Upsilon}
\Upsilon=\vartheta\circ\Phi+ \phi_\theta
\end{equation}
is the sum of a function $\Upsilon_{\mathrm{hom}}\in\mathrm{Ker}\mathcal{L}$ and a function which vanishes when evaluated at $(\theta,\sigma_\pm(\theta,G^{\mathrm{u}}),G^{\mathrm{u}},\tilde{G}^{\mathrm{s}}_\pm(\theta,G^{\mathrm{u}}))$.

By construction, if we denote by $X_{H}$ the vector field associated to the Hamiltonian in \eqref{eq:HJfinal} and write
\[
X_{H}^{\mathrm{u}}=(X_{H,u}^{\mathrm{u}},X_{H,\beta}^{\mathrm{u}}, X_{H,t}^{\mathrm{u}})=(X_{H,u}\circ\mathcal{W}^{\mathrm{u}},X_{H,\beta}\circ\mathcal{W}^{\mathrm{u}}, G_{0}^3),
\]
then $\Theta$, defined in \eqref{eq:straightenunstablefoliation},  conjugates the vector field 
\[
\dot{u}=X_{H,u}^{\mathrm{u}}\circ \Theta^{-1}\qquad\qquad \dot{\alpha}=0\qquad\qquad \dot{t}=G_{0}^3
\]
to the vector field $X^{\mathrm{u}}_{\mathrm{h}}$. That is, $\Theta^{-1}$ straightens the dynamics in the $\beta$ component. It is straightforward to check that this conjugacy is equivalent to the fact that, $\vartheta$ defined in \eqref{eq:straightenunstablefoliation}, solves
\begin{equation}\label{eq:dynunstablefoliation1}
\mathcal{L}^{\mathrm{u}}\vartheta=-X_{H,\beta}^{\mathrm{u}}
\end{equation}
with
\[
\mathcal{L}^{\mathrm{u}}= X_{H,u}^{\mathrm{u}} \partial_u+ X_{H,\beta}^{\mathrm{u}} \partial_\beta+G_{0}^3 \partial_t
\]
Notice now that 
\[
X^{\mathrm{u}}_{H,u}=1+2 \hat{A}^{\mathrm{u}}\qquad\qquad  X^{\mathrm{u}}_{H,\beta}=2 \hat{B}^{\mathrm{u}},
\]
where $\hat A^{\mathrm{u}}$ and $\hat B^{\mathrm{u}}$ are the functions defined in \eqref{eq:definitionAusBushat}. Therefore, denoting by $\widetilde{\mathcal{L}}$ the differential operator defined in \eqref{eq:semisumdiffoperator} one can rewrite \eqref{eq:dynunstablefoliation1} as
\[
\begin{split}
\mathcal{\widetilde{L}}\vartheta=&-2\hat{B}^{\mathrm{u}}+ \left(\hat{A}^{\mathrm{s}}-\hat{A}^{\mathrm{u}}\right)\partial_u \vartheta+ \left(\hat{B}^{\mathrm{s}}-\hat{B}^{\mathrm{u}}\right)\partial_\beta \vartheta.\\
\end{split}
\]
 It now follows from the definition of $\Phi$ in Lemma \ref{lem:straighteningcv} that $\Upsilon$, defined in \eqref{eq:Upsilon}, solves
\begin{equation}\label{eq:sinnombre}
\begin{split}
\mathcal{L}\Upsilon=&\bigg((\hat A^{\mathrm{u}}-\hat A^{\mathrm{s}})\partial_u \vartheta+\left(\hat B^{\mathrm{s}}-\hat B^{\mathrm{u}}\right)(1+\partial_\beta \vartheta)\bigg)\circ\Phi.
\end{split}
\end{equation}
Write $\Phi^{-1}=(u+\tilde{\phi}_u, \beta+\tilde{\phi}_\beta, t)$. Thus, from the definition of $\hat A^{\mathrm{u},\mathrm{s}}$ and $\hat B^{\mathrm{u},\mathrm{s}}$,  expression \eqref{eq:sinnombre} can be rewritten as
\[
\mathcal{L}\Upsilon= F \partial_v\Delta \mathcal{S}+G \partial_\theta \Delta \mathcal{S}
\]
where $\Delta \mathcal{S}$ is defined in \eqref{eq:deltascali},
\[
F=f(1+\partial_v\tilde{\phi}_u)+g\partial_v \tilde{\phi}_\beta \qquad\qquad G=g(1+\partial_\beta \tilde{\phi}_\beta) +f\partial_\beta \tilde{\phi}_v
\]
and
\[
f=-\frac{1}{2y_{\mathrm{h}}^2r_{\mathrm{h}}^2} \left(  1+ \partial_\beta \vartheta-r_{\mathrm{h}}^2\partial_u \vartheta \right)\qquad\qquad g=\frac{1}{y_{\mathrm{h}}^2r_{\mathrm{h}}^3 } \left(  1+\partial_\beta \vartheta -r_{\mathrm{h}}^{-1} \partial_u \vartheta \right)
\]
Let $\mathcal{G}$ be the left inverse operator for $\mathcal{L}$ defined in \eqref{eq:leftinverseoperatorfinaldomain}. Since $\Delta \mathcal{S}\in\mathrm{Ker}\mathcal{L}$ (and $\partial_v \Delta \mathcal{S}, \partial_\theta\Delta \mathcal{S}$ too),
\[
\mathcal{L}\left( \mathcal{G} (F) \partial_v\Delta \mathcal{S} +\mathcal{G}(G)\partial_\theta \Delta \mathcal{S}\right) =  F \partial_v\Delta \mathcal{S}+G \partial_\theta \Delta \mathcal{S}
\]
and hence,
\[
\Upsilon=\Upsilon_{\mathrm{hom}}+\mathcal{G}(F)\partial_v \Delta \mathcal{S}+\mathcal{G}(G)\partial_\theta \Delta \mathcal{S}
\]
for some function $\Upsilon_{\mathrm{hom}}\in\mathrm{Ker}\mathcal{L}$. Define now 
\begin{equation}\label{eq:upsilonpm}
\Upsilon_\pm(\theta,G^{\mathrm{u}})=\Upsilon(-G_{0}^{-3}\sigma_\pm(\theta,G^{\mathrm{u}}),\theta,0; G^{\mathrm{u}},\tilde{G}^{\mathrm{s}}_\pm(\theta,G^{\mathrm{u}}))
\end{equation}
where $\sigma_\pm(\theta,G^{\mathrm{u}}),\tilde{G}^{\mathrm{s}}_\pm(\theta,G^{\mathrm{u}})$ are the functions obtained in Theorem \ref{thm:mainthmcriticalpoints}. Then, the functions $\widetilde{\Omega}_\pm$ defined in \eqref{eq:modifiedwavemaps} satisfy 
\begin{equation}\label{eq:omegatildeupsilon}
\widetilde{\Omega}_\pm(\theta,G^{\mathrm{u}})=(\theta+\Upsilon_\pm(\theta,G^{\mathrm{u}}), G^{\mathrm{u}}).
\end{equation}

\begin{lem}\label{lem:differencewavemaps}
For all $(\theta,G^{\mathrm{u}})\in \Lambda_{\rho}$,
\[ 
| \Upsilon_\pm | \lesssim G_0^{-15/4} \qquad\qquad\text{and}\qquad\qquad | \Upsilon_+- \Upsilon_-|\lesssim G_0^{-5/8}  \exp(-G_0^3/3).
\]
\end{lem}

\begin{proof}
Taking into account the estimate for $\phi_\theta$ in Lemma \ref{lem:straighteningcv} and the estimate \eqref{eq:estimatevartheta} (the norm $\lVert \cdot \rVert_{\nu,\rho}$ is defined in \eqref{eq:finalnormsplitting})
\[
\lVert \Upsilon \rVert_{1/2,\rho}\leq \lVert \vartheta \rVert_{1/2,\rho} + \lVert \phi_\theta \rVert_{1/2,\rho} \lesssim G_0^{-15/4},
\]
which implies the first estimate. In order to prove the result for the difference we only need to estimate
\[
\left\lVert \Upsilon_{\mathrm{hom}} \right\rVert_{1/2,\rho} =\left\lVert \Upsilon -\mathcal{G}(F)\partial_v \Delta \mathcal{S}-\mathcal{G}(G)\partial_\theta \Delta \mathcal{S}\right\rVert_{1/2,\rho}.
\]
Indeed, since $\partial_v\Delta \mathcal{S}( \sigma_\pm(\theta,G^{\mathrm{u}}),\theta;G^{\mathrm{u}},\tilde{G}^{\mathrm{s}}_\pm(\theta,G^{\mathrm{u}}),\zeta)=\partial_\theta \Delta \mathcal{S}( \sigma_\pm(\theta,G^{\mathrm{u}}),\theta;G^{\mathrm{u}},\tilde{G}^{\mathrm{s}}_\pm(\theta,G^{\mathrm{u}}),\zeta)=0$, it follows from Lemma \ref{lem:exponsmallness}
\[
|\Upsilon_+-\Upsilon_-|\leq 2 \left|(\mathrm{Id}-\pi_0) \Upsilon_{\mathrm{hom}} \right| \lesssim G_0^{3}\left\lVert \Upsilon_{\mathrm{hom}} \right\rVert_{1/2,\rho}  \exp(-G_0^3 /3).
\] 
To estimate $\left\lVert \Upsilon_{\mathrm{hom}} \right\rVert_{1/2,\rho} $, one can check from the definition of $F,G$ and the estimates for $\lVert \Delta S \rVert_{1/2,\rho}$ which can be deduced from  Theorem \ref{thm:differencegenfunctcommon} that
\[
\begin{split}
\lVert \mathcal{G}(F) \partial_v \Delta \mathcal{S}\rVert_{1/2,\rho} \lesssim & G_0^{-1} \lVert \partial_v \Delta \mathcal{S}\rVert_{3/2,\rho} \lesssim G_0^{-4}\\
\lVert \mathcal{G}(G) \partial_\theta \Delta \mathcal{S}\rVert_{1/2,\rho} \lesssim & G_0^{-1} \lVert \partial_\theta \Delta \mathcal{S}\rVert_{1,\rho} \lesssim G_0^{-4}\\
\end{split}
\]
and the proof is completed.
\end{proof}


\begin{rem}
From now on we will decrease the value of $\rho>0$ without mentioning.
\end{rem}

\subsection{Proof of Theorem \ref{thm:mainthmscattmaps}}\label{sec:proof38}
From the definition  of the scattering maps $\mathbb{P}_\pm(\alpha^{\mathrm{u}},G^{\mathrm{u}})=(\alpha^{\mathrm{s}}_\pm,G^{\mathrm{s}}_\pm)$ in \eqref{eq:defnscatteringmaps}, the expressions \eqref{eq:usefulexpressionswavemaps} for the wave maps $\Omega_\pm^{\mathrm{u}}$ and $\Omega^{\mathrm{s}}_\pm$ ,  and the definition of $\widetilde{\Omega}_\pm$ in \eqref{eq:modifiedwavemaps}
\[
G_\pm^{\mathrm{s}}=\tilde{G}^{\mathrm{s}}_\pm \circ (\widetilde{\Omega}_\pm)^{-1}.
\]
We then write 
\begin{align*}
G_\pm^{\mathrm{s}}= G^{\mathrm{u}}+\partial_\theta \mathcal{L}_\pm+ (\tilde{G}^{\mathrm{s}}_\pm-( G^{\mathrm{u}}+\partial_\theta \mathcal{L}_\pm)) \circ (\widetilde{\Omega}_\pm)^{-1} +  (( G^{\mathrm{u}}+\partial_\theta \mathcal{L}_\pm)\circ (\widetilde{\Omega}_\pm)^{-1}-  (G^{\mathrm{u}}+\partial_\theta \mathcal{L}_\pm)).
\end{align*}
Therefore, from Lemma \ref{lem:technicallemmacriticalpoints}, we obtain that for all $(\alpha^{\mathrm{u}},G^{\mathrm{u}})$ in a complex $\rho$-neighborhood of $\mathbb{T}\times\mathbb{R}$
\[
| (\tilde{G}^{\mathrm{s}}_\pm-( G^{\mathrm{u}}+\partial_\theta \mathcal{L}_\pm)) \circ (\widetilde{\Omega}_\pm)^{-1}|\lesssim \zeta 
|G^{\mathrm{u}}|^{-11/2}.
\]
Also, from the estimates of the Melnikov potential $L$ given in Lemma \ref{lem:mainMelnikov}, the expression \eqref{eq:omegatildeupsilon} for $\widetilde{\Omega}_\pm$ and the estimate for $\Upsilon_\pm$ in Lemma \ref{lem:differencewavemaps}, we deduce that, for all $(\alpha^{\mathrm{u}},G^{\mathrm{u}})$ in a complex $\rho$-neighborhood of $\mathbb{T}\times\mathbb{R}$
\[
|( G^{\mathrm{u}}+\partial_\theta \mathcal{L}_\pm)\circ (\widetilde{\Omega}_\pm)^{-1}-  (G^{\mathrm{u}}+\partial_\theta \mathcal{L}_\pm)|\lesssim \zeta |G^{\mathrm{u}}|^{-69/8}.
\]
Combining both estimates
\[
|G_\pm^{\mathrm{s}}-(G^{\mathrm{u}}+\partial_\theta \mathcal{L}_\pm)|\lesssim \zeta |G^{\mathrm{u}}|^{-11/2}.
\]
The result 
\[
|\alpha^{\mathrm{s}}_\pm-(\alpha^{\mathrm{u}}-\partial_{G^{\mathrm{u}}}\mathcal{L}_\pm)|\lesssim |G^{\mathrm{u}}|^{-7}
\]
has already been proved in \cite{MR3583476} (see also the proof of Proposition \ref{prop:Asymptoticformulasgenfunct}).
\begin{rem}
In \cite{MR3583476} the authors consider the case $0\leq \zeta\leq \exp (-G_0^3/3))$, however, since both the main term in the asymptotic expansion and the error come from the circular part, the result holds for $0\leq \zeta<1$).
\end{rem}


\subsection{Proof of Theorem \ref{thm:mainthmscattmaps2}}\label{sec:proof39}

We now derive asymptotic formulas for the difference between the components of each of the maps $\mathbb{P}_+$ and $\mathbb{P}_-$ defined in \eqref{eq:defnscatteringmaps}, thus completing the proof of Theorem   \ref{thm:mainthmscattmaps2}.

Let $\tilde{G}^{\mathrm{s}}_\pm(\theta,G^{\mathrm{u}})$ be the functions obtained in Theorem \ref{thm:mainthmcriticalpoints}, let $\hat{\sigma}_\pm(G^{\mathrm{u}},G^{\mathrm{s}}),\hat{\theta}_\pm(G^{\mathrm{u}},G^{\mathrm{s}})$ be the functions obtained in Proposition  \ref{prop:sndthmcriticalpoints}, denote by $\Xi_\pm$ be the maps
\begin{equation}
(G^{\mathrm{u}}, G^{\mathrm{s}})\mapsto \Xi_\pm(G^{\mathrm{u}}, G^{\mathrm{s}})= (\hat{\theta}_\pm(G^{\mathrm{u}},G^{\mathrm{s}}),G^{\mathrm{u}})
\end{equation}
and define the function (see Proposition \ref{prop:generatingfunctscattmap})
\[
\begin{split}
\mathtt{S}_\pm(G^{\mathrm{u}}, G^{\mathrm{s}})=& G^{\mathrm{u}}\Delta \mathcal{S}(\hat{\sigma}_\pm(G^{\mathrm{u}},G^{\mathrm{s}}),\hat{\theta}_\pm(G^{\mathrm{u}},G^{\mathrm{s}}); G^{\mathrm{u}},G^{\mathrm{s}}).\\
\end{split}
\]
Then, it follows from Proposition \ref{prop:generatingfunctscattmap} that, for $(\alpha^{\mathrm{u}},G^{\mathrm{u}})\in \mathcal{P}^*_{\mathrm{vert}} =\mathcal{P}^*_\infty\cap \{\pi/8\leq \alpha^{\mathrm{u}}\leq \pi/4\}$ (see Remark \ref{rem:verticalstriprestriction}) the scattering maps $\mathbb{P}_\pm :(\alpha^{\mathrm{u}},G^{\mathrm{u}})\mapsto (\alpha^{\mathrm{s}}_\pm, G^{\mathrm{s}}_\pm)$   are given by the implicit expression
\begin{equation}\label{eq:closedscattmapsec4}
\begin{split}
(\alpha^{\mathrm{u}},G^{\mathrm{u}})\mapsto (\alpha^{\mathrm{u}}+(\partial_ {G^{\mathrm{u}}} \mathtt{S}_\pm+\partial_ {G^{\mathrm{s}}} \mathtt{S}_\pm)\circ (\Omega^{\mathrm{u}}_\pm\circ \Phi_{\mathrm{h}}\circ \Phi_\pm\circ \Xi_\pm)^{-1},\ \tilde{G}^{\mathrm{s}}_\pm\circ (\Omega^{\mathrm{u}}_\pm \circ \Phi_{\mathrm{h}}\circ \Phi_\pm)^{-1}).
\end{split}
\end{equation}

\begin{prop}\label{prop:Asymptoticformulasgenfunct}
Let $G_*$ be sufficiently large and consider $0\leq \zeta\leq (G_*+R)^{-3}$. Let $\mathcal{L}_\pm(\theta,G^{\mathrm{u}})$ be the reduced Melnikov potentials introduced in \eqref{eq: DefinitionredMelnikovPotential}. Then, 
\begin{itemize}
\item For all $(\theta,G^{\mathrm{u}})\in \mathbb{T}\times[G_*,G_*+R]$
\begin{equation}\label{eq:differenceItildes}
 |(\tilde{G}_+^{\mathrm{s}}-\tilde{G}_-^{\mathrm{s}})(\theta,G^{\mathrm{u}})- \partial_\theta( \mathcal{L}_+-\mathcal{L}_-)(\theta,G^{\mathrm{u}})|\lesssim \zeta (G^{\mathrm{u}})^{-5/2} \exp(-(G^{\mathrm{u}})^3)/3)
\end{equation}
\item For all $(\theta, G^{\mathrm{u}})\in  [\pi/8,\pi/4]\times[G_*,G_*+R]\subset \mathrm{Dom}(\Xi_\pm^{-1})$, we have
\begin{equation}\label{eq:differencegeneratingfunctions}
\begin{split}
 |(\partial_ {G^{\mathrm{u}}} \mathtt{S}_+ +\partial_ {G^{\mathrm{s}}} \mathtt{S}_+ -\partial_{G^{\mathrm{u}}}  \mathcal{L}_+)\circ \Xi_+^{-1} - &(\partial_{G^{\mathrm{u}}} \mathtt{S}_-+\partial_{G^{\mathrm{s}}} \mathtt{S}_--\partial_{G^{\mathrm{u}}}\mathcal{L}_-)\circ \Xi_-^{-1}|\\
& \lesssim (G^{\mathrm{u}})^{-1/2} \exp(-(G^{\mathrm{u}})^3)/3) .
\end{split}
\end{equation}
\end{itemize}
\end{prop}

\begin{proof}
We first check \eqref{eq:differenceItildes}. To do so, we write
\[
\begin{split}
\partial_\sigma \Delta \mathcal{S}= & \partial_\sigma \Delta \mathcal{S}_{\mathrm{circ}}+ \partial_\sigma (L-L_{\mathrm{circ}})+\mathcal{E}_\sigma \\
\partial_\theta \Delta \mathcal{S}= &\partial_\theta \Delta \mathcal{S}^{[0]}+\partial_{\theta} \left( (\mathrm{Id}-\pi_0) (L-L_{\mathrm{circ}})\right)+ \partial_{\theta} \left( (\mathrm{Id}-\pi_0) \Delta \mathcal{S}_{\mathrm{circ}}\right)+ \mathcal{E}_\theta,\\
\end{split}
\]
where $\mathcal{E}_\sigma=\partial_\sigma\mathcal{E}_{\mathrm{circ}}$, $\mathcal{E}_\theta=\partial_\theta  \left( (\mathrm{Id}-\pi_0)\mathcal{E}_{\mathrm{circ}}\right)$ and  $\mathcal{E}_{\mathrm{circ}}$ has been defined in \eqref{eq:finalerrorapproxcircular}. Let now $\hat{\sigma}_\pm(\theta,G^{\mathrm{u}},G^{\mathrm{s}})$ and $\hat{I}^{\mathrm{s}}(\theta,G^{\mathrm{u}})$ be such that 
\[
\partial_\sigma \Delta \mathcal{S}(\theta,\tilde{\sigma}_\pm(\theta,G^{\mathrm{u}},G^{\mathrm{s}}),G^{\mathrm{u}},G^{\mathrm{s}})=0\qquad\qquad   (\partial_\theta \Delta\mathcal{S})^{[0]} (\theta,G^{\mathrm{u}},\hat{I}^{\mathrm{s}}(\theta,G^{\mathrm{u}}))=0.
\]
One expects that the solution $(\sigma,G^{\mathrm{s}})=(\sigma_\pm(\theta,G^{\mathrm{u}}),\tilde{G}^{\mathrm{s}}_\pm(\theta,G^{\mathrm{u}}))$ to $(\partial_\sigma \Delta \mathcal{S}, \partial_\theta \Delta \mathcal{S})=0$ is close to $(\sigma,G^{\mathrm{s}})=(\hat{\sigma}_\pm(\theta,G^{\mathrm{u}},\hat{I}^{\mathrm{s}}(\theta,G^{\mathrm{u}})),\hat{I}^{\mathrm{s}}(\theta,G^{\mathrm{u}}))$. The main term in the correction of the solution to the second equation of the system $(\partial_\sigma \Delta\mathcal{S}, \partial_\theta \Delta \mathcal{S})=0$ is given by the term
\begin{equation}\label{eq:proofproposition621}
\partial_{\theta} \left( (\mathrm{Id}-\pi_0) (L-L_{\mathrm{circ}})\right)(\theta,\hat{\sigma}_\pm(\theta,G^{\mathrm{u}},G^{\mathrm{s}});G^{\mathrm{u}},\hat{I}^{\mathrm{s}}(\theta,G^{\mathrm{u}}),\zeta)+ \partial_{\theta} \left( (\mathrm{Id}-\pi_0) \Delta \mathcal{S}_{\mathrm{circ}}\right)(\theta,\hat{\sigma}_\pm(\theta,G^{\mathrm{u}},\hat{I}^{\mathrm{s}}(\theta,G^{\mathrm{u}})); G^{\mathrm{u}},G^{\mathrm{s}}).
\end{equation}
Therefore, using the fact that $\Delta \mathcal{S}_{\mathrm{circ}}(\theta,\sigma;G^{\mathrm{u}},G^{\mathrm{s}})= \Delta \widehat{\mathcal{S}}_{\mathrm{circ}}(\sigma-\theta;G^{\mathrm{u}},G^{\mathrm{s}})$ and the definition of $\hat{\sigma}_{\pm}(\theta,G^{\mathrm{u}},G^{\mathrm{s}})$, the term \eqref{eq:proofproposition621} can be expressed as
\[
\begin{split}
\partial_{\theta} \left( (\mathrm{Id}-\pi_0) (L-L_{\mathrm{circ}})\right)(\theta,\hat{\sigma}_\pm(\theta,G^{\mathrm{u}},G^{\mathrm{s}});G^{\mathrm{u}},\hat{I}^{\mathrm{s}}(\theta,G^{\mathrm{u}}),\zeta)&+ \partial_{\sigma} (L-L_{\mathrm{circ}})(\theta,\hat{\sigma}_\pm(\theta,G^{\mathrm{u}},\hat{I}^{\mathrm{s}}(\theta,G^{\mathrm{u}}));G^{\mathrm{u}},G^{\mathrm{s}},\zeta)\\
&+\mathcal{E}_\sigma(\theta,\hat{\sigma}_\pm(\theta,G^{\mathrm{u}},G^{\mathrm{s}});G^{\mathrm{u}},\hat{I}^{\mathrm{s}}(\theta,G^{\mathrm{u}}),\zeta).
\end{split}
\]
It  follows from the fact that $L_{\mathrm{circ}}=L_{\mathrm{circ}}(\sigma-\theta,G^{\mathrm{u}})$ and the definitions of $L(\theta,G^{\mathrm{u}};\varepsilon)$ and $\mathcal{L}_\pm(\theta,G^{\mathrm{u}};\zeta)$ that
\[
(G^{\mathrm{u}})^{-1}\partial_\theta \mathcal{L}_\pm (\theta,G^{\mathrm{u}};\zeta)= \partial_{\theta}\left( (\mathrm{Id}-\pi_0) (L-L_{\mathrm{circ}})\right) (\theta,\tilde{\sigma}_\pm(\theta);G^{\mathrm{u}},\zeta)+ \partial_\sigma (L-L_{\mathrm{circ}})(\theta,\tilde{\sigma}_\pm(\theta);G^{\mathrm{u}},\zeta),
\]
where
\[
\tilde{\sigma}_+(\theta)=\theta\qquad \tilde{\sigma}_+(\theta)=\theta+\pi.
\]
Then, the asymptotic formula \eqref{eq:differenceItildes} follows from the estimates
\[
|\mathcal{E}_\sigma|,|\mathcal{E}_\theta|\lesssim \zeta |G^{\mathrm{u}}|^{-5/2}\exp(-G_0^3/3))
\]
the fact that 
\[
|\tilde{\sigma}(\theta)-\hat{\sigma}(\theta,G^{\mathrm{u}},\hat{I}^{\mathrm{s}}(\theta,G^{\mathrm{u}})|\lesssim |G^{\mathrm{u}}|^{-1},
\]
and Lemma \ref{lem:mainMelnikov}. 

We now prove the asymptotic formula \eqref{eq:differencegeneratingfunctions}. Let $\Phi_\pm$ be defined in \eqref{eq:definitionPhi+-} and $\Xi_\pm$ be defined in \eqref{eq:mapsXi+-}. Then, using that 
\[
(\partial_\sigma \Delta \mathcal{S})\circ\Phi_\pm\circ \Xi_\pm=(\partial_\theta \Delta \mathcal{S})\circ\Phi_\pm\circ \Xi_\pm=0
\]
we have
\[
\begin{split}
\partial_{G^{\mathrm{u}}} \mathtt{S}_\pm+\partial_{I^{\mathrm{s}}} \mathtt{S}_\pm= & \partial_{G^{\mathrm{u}}} (G^{\mathrm{u}}\Delta \mathcal{S}\circ\Phi_\pm\circ \Xi_\pm)+\partial_{G^{\mathrm{s}}} (G^{\mathrm{u}}\Delta \mathcal{S}\circ\Phi_\pm\circ \Xi_\pm)\\
=&  ( \partial_{G^{\mathrm{u}}} \left(G^{\mathrm{u}} \Delta S\right)+\partial_{G^{\mathrm{s}}} \left(G^{\mathrm{u}}\Delta S\right))\circ \Phi_\pm\circ \Xi_\pm\\
=& (\partial_{G^{\mathrm{u}} } (G^{\mathrm{u}}\tilde L)+ \partial_ {G^{\mathrm{s}}} (G^{\mathrm{u}}\tilde L) )\circ\Phi_\pm\circ \Xi_\pm + \mathcal{E}_I
\end{split}
\]
where $\tilde{L}$ is the Melnikov potential defined in \eqref{eq:defnMelnikovPot},  and
\[
\begin{split}
\mathcal{E}_I= ( \partial_{G^{\mathrm{u}}}(G^{\mathrm{u}}\Delta S -G^{\mathrm{u}}\tilde{L})+ \partial_{G^{\mathrm{s}}}(G^{\mathrm{u}}\Delta S-G^{\mathrm{u}}\tilde{L}) )\circ\Phi_\pm\circ \Xi_\pm.
\end{split}
\]
It follows from the estimate
\[
|(\mathrm{Id}-\pi_0) (\Delta S-\tilde{L})(\theta,\sigma; G^{\mathrm{u}},G^{\mathrm{s}},\zeta) |\lesssim (G^{\mathrm{u}})^{-7/2}  \exp(-(G^{\mathrm{u}})^3/3)
\]
 in Proposition \ref{prop:approxgenfunctionbyMelnikov}, that for all $(\theta,G^{\mathrm{u}})\in \Lambda$
\[
|(\mathrm{Id}-\pi_0)\partial_{G^{\mathrm{u},\mathrm{s}}} (\Delta S-\tilde{L})(\theta,\sigma; G^{\mathrm{u}},G^{\mathrm{s}},\zeta) |\lesssim (G^{\mathrm{u}})^{-3/2}  \exp(-(G^{\mathrm{u}})^3/3).
\]
so 
\[
|\mathcal{E}_I(\theta,\sigma; G^{\mathrm{u}},G^{\mathrm{s}},\zeta) |\lesssim  (G^{\mathrm{u}})^{-1/2} \exp(-(G^{\mathrm{u}})^3/3).
\]
Therefore, it follows from the definition of $\Phi_\pm(\theta,G^{\mathrm{u}})$ that 
\[
|\mathcal{E}_I\circ\Phi_+-\mathcal{E}_I\circ \Phi_-|\lesssim (G^{\mathrm{u}})^{-1/2} \exp(-(G^{\mathrm{u}})^3/3),
\]
 for all $(\theta,G^{\mathrm{u}})\in \Lambda$  and the asymptotic formula \eqref{eq:differencegeneratingfunctions} is inmediate.
\end{proof}

Finally, we complete the proof of Theorem \ref{thm:mainthmscattmaps2}.

\begin{proof}[Proof of Theorem \ref{thm:mainthmscattmaps2}]

 We write
\[
G^{\mathrm{s}}_+-G^{\mathrm{s}}_-= (\tilde{G}_+^{\mathrm{s}}-\tilde{G}_-^{\mathrm{s}})\circ (\widetilde{\Omega}_+)^{-1}+ \mathcal{E}_1\qquad\qquad\text{with}\qquad\qquad\mathcal{E}_1=\tilde{G}^{\mathrm{s}}_-\circ (\widetilde{\Omega}_+)^{-1}-\tilde{G}^{\mathrm{s}}_-\circ (\widetilde{\Omega}_-)^{-1}
\]
and the result for the $G$ component follows using that $\widetilde{\Omega}_\pm(\theta,G^{\mathrm{u}})=(\theta+\Upsilon_\pm(\theta,G^{\mathrm{u}}), G^{\mathrm{u}})$ and the estimate $| \Upsilon_+-\Upsilon_-|\lesssim |G^{\mathrm{u}}|^{-5/8}\exp(-G^{\mathrm{u}})^3/3)$ given in Lemma \ref{lem:differencewavemaps}.  Indeed by the mean value theorem
\[
\begin{split}
|\mathcal{E}_1|= &|\tilde{G}^{\mathrm{s}}_-\circ (\widetilde{\Omega}_+)^{-1}-\tilde{G}^{\mathrm{s}}_-\circ (\widetilde{\Omega}_-)^{-1}|
\lesssim \sup_{\theta\in\mathbb{T}_\rho} |\partial_{\theta} \tilde{G}^{\mathrm{s}}_-| | \Upsilon_+-\Upsilon_-| 
\lesssim \zeta (G^{\mathrm{u}})^{-45/8}\exp (-(G^{\mathrm{u}})^3/3),
\end{split}
\]
where we have used that for all $(\theta,G^{\mathrm{u}})\in \Lambda_\rho$
\[
|\partial_{\theta} \tilde{G}^{\mathrm{s}}_-|\lesssim |\partial^2_{\theta\theta} \mathcal{L}_-|\lesssim \zeta (G^{\mathrm{u}})^{-5}.
\]
We now study the angular component, which for $(\alpha^{\mathrm{u}},G^{\mathrm{u}})\in \Lambda_{\mathrm{vert}}$, is given by
\[
\alpha^{\mathrm{s}}_+-\alpha^{\mathrm{s}}_-= ( (\partial_{G^{\mathrm{u}}} \mathtt{S}_++ \partial_{G^{\mathrm{s}}} \mathtt{S}_+)\circ \Xi_+^{-1} - (\partial_{G^{\mathrm{u}}} \mathtt{S}_-+ \partial_{G^{\mathrm{s}}} \mathtt{S}_-) \circ \Xi_-^{-1})\circ \widetilde{\Omega}_+^{-1}+ \mathcal{E}_2,
\]
where
\[
\begin{split}
\mathcal{E}_2= &(\partial_{G^{\mathrm{u}}} \mathtt{S}_-+ \partial_{G^{\mathrm{s}}} \mathtt{S}_-)\circ (\widetilde{\Omega}_+\circ \Xi_-)^{-1}-(\partial_{G^{\mathrm{u}}} \mathtt{S}_-+ \partial_{G^{\mathrm{s}}} \mathtt{S}_-)\circ (\widetilde{\Omega}_-\circ \Xi_-)^{-1}.\\
\end{split}
\]
The asymptotic formulas for the Melnikov potential given in  Lemma \ref{lem:mainMelnikov} and the uniform estimates in Proposition \ref{prop:Asymptoticformulasgenfunct} imply that
\[
 |\partial_\theta ((\partial_{G^{\mathrm{u}}} \mathtt{S}_-+ \partial_{G^{\mathrm{s}}} \mathtt{S}_-)\circ \Xi_-^{-1}) |\lesssim \zeta (G^{\mathrm{u}})^{-6}.
\]
Since
\[
|\Upsilon_+-\Upsilon_-|\lesssim  (G^{\mathrm{u}})^{-5/8}  \exp(-(G^{\mathrm{u}})^3)/3),
\]
we obtain that, for all $(\theta,G^{\mathrm{u}})\in \Lambda_{\mathrm{vert}}$,
\[
|\mathcal{E}_2|\lesssim \zeta  (G^{\mathrm{u}})^{-53/8} \exp(-(G^{\mathrm{u}})^3/3).
\]
Theorem \ref{thm:mainthmscattmaps2} now follows combining these estimates with the ones  given in Proposition \ref{prop:Asymptoticformulasgenfunct}.
\end{proof}

\section{Nonexistence of common invariant curves}\label{sec:abstractresultscattmaps}
We now present the proof of Theorem \ref{thm:scattmapmain}, which might be of interest on its own and its independent of the previous sections. It consists of three parts. First, we make use of a result by Kuksin and P\"{o}schel (\cite{MR1279392}) which produces a time periodic Hamiltonian whose time-one map coincides with the map $g_+$ in \eqref{eq:firstcondmapmoeckel} (this choice is completely arbitrary, one can choose the map $g_-$ instead). Then, in Lemma \ref{lem: Neishtadt's thm}, we perform several steps of averaging to eliminate the dependence on time up to an exponentially small term. Finally, we check that the condition \eqref{eq:suffcondmoeckel} guarantees the non-existence of common essential invariant curves.

Before elaborating on our argument, the introduction of some notation is in order. Given a domain $D\subset \mathbb{T}\times \mathbb{R}$ and $\rho>0$ we write $D_\rho$ to denote its $\rho$-neighborhood in $(\mathbb{C}/2\pi\mathbb{Z})\times\mathbb C$. We write $|\cdot|_\rho$ for the sup norm for functions $f:D_\rho\to \mathbb{C}$ and use $\lVert \cdot \rVert_\rho$ for the case where $f$ is vector valued. We abuse notation and also use $|\cdot|_\rho$ (respectively $\lVert \cdot\rVert_{\rho}$) for functions (vector fields) defined on $D_\rho\times\mathbb T_\rho$.

\begin{thm}[Theorem 1 in \cite{MR1279392}]\label{thm:Kuksin}
Fix $\rho>0$ and let $g:D_{\rho}\subset\mathbb{T}_{\rho}\times\mathbb{C}\to \mathbb{T}_{\rho}\times\mathbb{C}$ be a real-analytic exact symplectic map of the form $g=\tilde g+\hat g$ where
\[
\tilde g(\alpha,G)=(\alpha+\partial_G h(G), G)
\]
for some $h: \mathbb{R}\to \mathbb{R}$ and 
\[
\lVert \hat g\rVert_{\rho} \leq \delta.
\]
Then, there exists $\delta_0 (\rho,|h|_\rho,|Dh|_\rho,|D^2h|_\rho)>0$ such that for all $0\leq\delta\leq\delta_0$, there exists a non-autonomous time periodic real-analytic Hamiltonian $K(\alpha,G,\tau):D_{\rho}\times \mathbb{T}_{\rho}\to \mathbb{C}$ 
such that the time-one map $\phi_K$ associated to the flow of the Hamiltonian $K$ satisfies
\[
g=\phi_K.
\]
Moreover,
\[
|K-h|_{\rho}\lesssim \delta
\]
\end{thm}

\begin{rem}
    For the sake of clarity, we have just stated Theorem \ref{thm:Kuksin} for maps of the cylinder. However, in \cite{MR1279392}, Theorem \ref{thm:Kuksin} is stated and proved in any dimension. 
\end{rem}

We notice that the maps $g_\pm$ in \eqref{eq:firstcondmapmoeckel} satisfy the hypotheses of Theorem \ref{thm:Kuksin} with $D=\mathbb{A}=[a,b]\times\mathbb T$ and
\[
\partial_G h(G)=\varepsilon\omega(G),\qquad\qquad \delta=\delta (\varepsilon)
\]
Hence, Theorem \ref{thm:Kuksin} yields a real analytic Hamiltonian function $K_+$ such that 
\[
g_+=\phi_{K_+} 
\]
and 
\begin{equation}\label{eq:estimatekplush}
|K_+-h|_\rho\lesssim  \delta(\varepsilon),
\end{equation}
where $\rho= \rho_0/2$ . Writing $X_{K_+}$ for the vector field generated by $K_+$ and expanding its time-one map in Taylor series we get that
for all $(\alpha,G)\in D_{\rho}$,
\begin{equation}\label{eq:vfieldapprox}
X_{K_+}-(g_+-\mathrm{Id}) = \left( \mathcal{O}( \varepsilon^2 +\varepsilon\delta(\varepsilon)),\ \mathcal{O} (\varepsilon\delta(\varepsilon)) \right).
\end{equation}
Hence, we observe that the Hamiltonian vector field $X_{K_+}$ is a slow fast system on $(\alpha,G)\in D$,  $\tau\in\mathbb T$ since $\dot{\tau}=1$ while $\dot{\alpha}=\mathcal{O}(\varepsilon)$ and $\dot{G}=\mathcal{O}(\delta(\varepsilon))$. We now obtain a normal form similar to that obtained by Neishtadt in (\cite{MR802878}) for the Hamiltonian function $K_+$ to push the $\tau$ dependence to an exponentially small remainder.\vspace{0.4cm}

\begin{lem}\label{lem: Neishtadt's thm}
There exists a real analytic change of variables $\psi:  D_{\rho/8}\to D_{\rho/2}$ with 
\begin{equation}\label{eq:Neishtadtchangevar}
 \lVert \mathrm{Id}-\psi\rVert_{\rho/8} \lesssim \varepsilon \delta(\varepsilon)
\end{equation}
and a real-analytic autonomous Hamiltonian function $\mathcal{K}_+:D_{\rho/8}\to \mathbb{C}$ such that the map 
\[
\mathtt{g}_+= \psi^{-1}\circ g_+  \circ \psi
\]
and the time-one map $\phi_{\mathcal{K}_+}$ associated to the Hamiltonian function $\mathcal{K}_+$ satisfy
\begin{equation}\label{eq: exponentially small remainder interpolation}
\lVert \mathtt{g}_+ - \phi_{\mathcal{K}_+}\rVert_{\rho/8}\lesssim \delta(\varepsilon) \exp\left(-c/\varepsilon\right)
\end{equation}
for some $c=c (\rho)>0$.
\end{lem}

\begin{proof}

The proof follows the ideas developed in \cite{MR802878} but in a Hamiltonian setting. We only sketch the proof  in order to keep track of the dependence of the error terms on $\delta$.

We  look for a symplectic change of variables as the time-one map of the Hamiltonian flow $\phi_{F_1}$ induced by a function $F_1$ to be determined. We write $K_+=K_0+R_0$ where $K_0=\langle K_+\rangle$. Notice that, by \eqref{eq:vfieldapprox}
\[
 \lVert X_{K_0} \rVert_{\rho/2}\lesssim   \varepsilon \qquad\qquad \lVert X_{R_0} \rVert_{\rho/2} \lesssim \varepsilon \delta(\varepsilon)\equiv \tilde{\varepsilon}.
\]
By Taylor's formula we find that
\[
\begin{split}
K_+\circ \phi_{F_1}=& K_0+\partial_\tau F_1+ R_0 +P_0
\end{split}
\]
where (here the bracket $\{\cdot,\cdot\}$ only involves the derivatives with respect to $\alpha$ and $G$)
\[
P_0=\left\{K_0,F_1\right\}+\int_0^1\left\{R_0+(1-s)\left\{K_0,F_1\right\},F_1\right\}\circ \phi_{F_1}^{\mathrm{s}}\mathrm{d}s\\.
\]
Since $\langle R_0 \rangle=0$ we can choose $F_1$ given by $F_1=-\int_0^\tau R_0\mathrm{d}s$. Hence,
\[
\left\lVert X_{F_1}\right\rVert_{\rho/2}\leq \tilde{\varepsilon}.
\]
Now we write  $K_+\circ\phi_{F_1}=K_1+R_1$ where $K_1=K_0+ \langle P_0\rangle$ and $R_1=P_0-\langle P_0\rangle$. Write $\tilde{\rho}=\rho/2$, then, the estimates 
\begin{equation}\label{eq:firstiterationNeishtadt}
\left\lVert X_{F_1}\right\rVert_{\tilde{\rho}}\lesssim \tilde{\varepsilon},\qquad\qquad |R_1|_{\tilde{\rho}-r}\lesssim \tilde{\varepsilon} \varepsilon, \qquad\qquad \left\lVert X_{K_1}-X_{K_0}\right\rVert_{\tilde{\rho}-2r}\lesssim \tilde{\varepsilon} r^{-1}\varepsilon
\end{equation}
for any $0<\varepsilon<r<\tilde{\rho}$ are straightforward. Indeed, 
\[
|\{K_0,F_1\}|_{\tilde{\rho}}\lesssim \lVert X_{K_+}\rVert_{\tilde{\rho}} \lVert X_{F_1}\rVert_{\tilde{\rho}} \lesssim \varepsilon\tilde{\varepsilon}\qquad\qquad |\{R_0,F_1\}|_{\tilde{\rho}}\lesssim \lVert X_{R_0}\rVert_{\tilde{\rho}} \lVert X_{F_1}\rVert_{\tilde{\rho}}\lesssim \tilde{\varepsilon}^2 
\]
and
\[
|\{\{K_0,F_1\},F_1\}|_{\tilde{\rho}-r}\lesssim r^{-2} \lVert X_{K_+}\rVert_{\tilde{\rho}} \lVert X_{F_1}\rVert_{\tilde{\rho}}  \lVert X_{F_1}\rVert_{\tilde{\rho}-r}\lesssim \varepsilon \tilde{\varepsilon}^2 r^{-2}\leq \varepsilon r^{-1} \tilde{\varepsilon}
\]
from where the second and third inequalities in \eqref{eq:firstiterationNeishtadt} plainly follow.  Assume now that we are able to carry on the process iteratively and find $n$ functions $F_i,\ i=1,\dots,n$ such that
\[
K\circ\phi_{F_1}\circ\cdots\circ\phi_{F_n}=K_n+R_n
\]
with 
\[
\left\lVert X_{F_n}\right\rVert_{\tilde{\rho}-2(n-1)r}\lesssim \tilde{\varepsilon} r^{-n+1} \varepsilon^{n-1},\quad\quad |R_n|_{\tilde{\rho}-(2n-1)r}\lesssim \tilde{\varepsilon} r^{-n+1} \varepsilon^n, \quad\quad \left\lVert X_{K_n}-X_{K_{n-1}}\right\rVert_{\tilde{\rho}-2nr}\lesssim \tilde{\varepsilon}r^{-n}\varepsilon^{n}
\]
where the symbol $a\lesssim b$ means that there exists $C>0$ which does not depend on $n,\varepsilon,\tilde{\varepsilon}$ and $r$ such that $a\leq C b$.

Then, if $r^{-1}\varepsilon <1$  and $\tilde{\rho}-2(n+1)r>0$ is an easy computation to show that we can perform one averaging step more to obtain a new function $F_{n+1}$ such that 
\[
K\circ\phi_{F_1}\circ\cdots\circ\phi_{F_{n+1}}=K_{n+1}+R_{n+1}
\]
with 
\[
\left\lVert X_{F_{n+1}}\right\rVert_{\tilde{\rho}-2nr}\lesssim \tilde{\varepsilon} r^{-n} \varepsilon^n,\qquad\qquad |R_{n+1}|_{\tilde{\rho}-(2n+1)r}\lesssim\tilde{\varepsilon} r^{-n}\varepsilon^{n+1}, 
\]
and 
\[
\left\lVert X_{K_{n+1}}-X_{K_{n}}\right\rVert_{\tilde{\rho}-2(n+1)r}\lesssim \tilde{\varepsilon}r^{-(n+1)}\varepsilon^{n+1}.
\]
Therefore, taking $r=2 \varepsilon$, after a number $N= \left[ \tilde{\rho} r^{-1}\right]/4$ of averaging steps we get that $\tilde{\rho}-2Nr\geq \tilde{\rho}/2=\rho/4$ and the reminder has size
\[
 |R_{N}|_{\rho/4}\lesssim\tilde{\varepsilon}r \left(\varepsilon/r\right)^N= 2\tilde{\varepsilon} \varepsilon 2^{-N}=2 \tilde{\varepsilon} \varepsilon \exp\left(\frac{-[\tilde{\rho}r^{-1}]\ln 2}{4}\right)
\]
so it follows, using the definition of $r$, that, for some $c>0$ depending only on $\varepsilon$,
\[
|R_N|_{\rho/4}\lesssim \varepsilon \delta(\varepsilon) \exp(-c/\varepsilon).
\]
Define now $\mathcal{K}_+=K_N$ and let $\psi=\phi_{F_1}\circ\cdots\circ\phi_{F_{N}}$. Then, it follows from our construction that 
\[
\mathtt g_+=\psi^{-1}\circ\phi_K\circ\psi=\phi_{K_N+R_N}=\phi_{\mathcal{K}_+}+ (\phi_{K_N+R_N}-\phi_{\mathcal{K}_+})
\]
and the estimate \eqref{eq: exponentially small remainder interpolation} follows.
On the other hand, standard Cauchy estimates show that
\[
\left\lVert \mathrm{Id}-\psi\right\rVert_{\rho/8}\lesssim \left\lVert X_{F_1}\right\rVert_{\rho/2}\lesssim \tilde\varepsilon=\varepsilon\delta(\varepsilon).\qedhere\
\]
\end{proof}

From the previous lemma, we observe that the curves $\{ \mathcal{K}_+=\mathrm{const} \}$ are almost invariant for the map 
\begin{equation}\label{eq:modifiedgmap+}
\mathtt{g}_+= \psi^{-1} \circ g_+\circ\psi.
\end{equation}
Moreover, standard KAM techniques show that, under a suitable nondegeneracy condition on $\omega(G;\varepsilon)$, a large part of the cylinder $\mathbb{T}\times[a,b]$ is filled by essential invariant curves, which are close to the level sets of the Hamiltonian $\mathcal{K}_+$, and  which leave small gaps between them. Since any other essential invariant curve must be confined between two KAM curves, we obtain the following result.

\begin{prop}\label{prop:KAM}
Let $\mathtt{g}_+$ be the map defined in \eqref{eq:modifiedgmap+}, let 
$
\mathbb{A}_\varepsilon=\mathbb{T}\times [a+\varepsilon^{1/4},b-\varepsilon^{1/4}],
$
let  
\[
\mathcal{A}_+=\{\gamma\subset \mathbb{A}_\varepsilon\colon \gamma\ \text{is an essential invariant curve for the map}\ \mathtt{g}_+\}
\]
and let  $c(\rho)>0$ be the constant obtained in Lemma \ref{lem: Neishtadt's thm}. Then, for any $\gamma\in\mathcal{A}_+$,
    \[
    \max\{ |\mathcal{K}_+(z_2)-\mathcal{K}_+(z_1)|,\ z_1,z_2\in\gamma\}\lesssim \frac{\sqrt{\delta(\varepsilon) \exp(-c(\rho)/\varepsilon)}}{\rho^2\varepsilon\omega'(G;\varepsilon) }.
    \]
\end{prop}

\begin{rem}
    Notice that Proposition \ref{prop:KAM} applies to all essential curves, independently of their inner dynamics.
\end{rem}

\begin{proof}
Let $K, h$ be the functions obtained in Theorem \ref{thm:Kuksin} and let $\mathcal{K}_+$ be the function obtained in Lemma \ref{lem: Neishtadt's thm}. Writing $h(G)=\varepsilon \int_0^G \omega(s)\mathrm{d}s$, we have 
\[
|\mathcal{K}_+-h|_{\rho}\lesssim \delta(\varepsilon),
\]
so there exists a real-analytic $\mathcal{O}(\delta(\varepsilon))$-close to identity canonical transformation $\chi:(\theta,I)\mapsto (\alpha,G)$ and a function $\widetilde{\mathcal{K}}_+$ such that $\widetilde{\mathcal{K}}_+(I)=\mathcal{K}_+\circ\chi(I,\theta)$. Namely, $\mathcal{\chi}$ puts the Hamiltonian $\mathcal{K}_+$ in action-angle variables. Moreover,
\[
\partial_I \widetilde{\mathcal{K}}_+(I)=\varepsilon \omega(I)+\mathcal{O}(\delta(\varepsilon)).
\]
Now, we notice that  $K\circ \vartheta \circ\chi (\lambda,I,t)$ satisfies that 
\[
|K\circ \vartheta\circ\chi -\widetilde{\mathcal{K}}_+|_{\rho}\lesssim \delta(\varepsilon) \exp(-c(\rho)).
\]
Then, standard KAM techniques, see for instance \cite{MR1858551}, show that there is a collection $\mathcal{C}_+\subset\mathcal{A}_+$ of essential invariant curves for the map $\mathtt{g}_+$ whose dynamics is conjugated to a Diophantine rotation. Moreover, for any $\gamma\in \mathcal{C}_+$  we have
    \[
    \max\{ |\mathcal{K}_+(z_2)-\mathcal{K}_+(z_1)|\colon\ z_1,z_2\in\gamma\}\lesssim \frac{\sqrt{\delta(\varepsilon) \exp(-c(\rho)/\varepsilon)}}{\rho^2\varepsilon },
    \]
    and
    \[
    \sup\{ \inf|\mathcal{K}_+(z_2)-\mathcal{K}_+(z_1)|:\ z_1\in\gamma,\ z_2\in\gamma',\ \gamma'\neq \gamma,\ \gamma'\in\mathcal{A}_+\}\lesssim \frac{\sqrt{\delta(\varepsilon) \exp(-c(\rho)/\varepsilon)}}{\rho^2\varepsilon \omega'(G;\varepsilon) }.
    \]
The conclusion in the statement follows since the KAM curves form a codimension one lamination of $\mathbb{A}$. Indeed, any essential invariant curve is either a KAM curve or is contained between two different KAM curves.\qedhere\
\end{proof}

In the following proposition we obtain quantitative estimates which will guarantee that the invariant curves of the map 
\begin{equation}\label{eq:modifiedgmap-}
\mathtt g_-=\psi^{-1}\circ g_-\circ\psi
\end{equation}
must be transverse to those of $\mathtt{g}_+$ if conditions \eqref{eq:secondcondmapmoeckel}, \eqref{eq:suffcondmoeckel}
 and \eqref{eq:suffcondmoeckel2} are satisfied. In the statement and proof of Proposition \ref{prop:transversality}, $\lVert \cdot \rVert$ stands for the Euclidean sup norm in $\mathbb{R}^2$.

\begin{prop}\label{prop:transversality}
Let $\mathtt{g}_-$ be the map defined in  \eqref{eq:modifiedgmap-}. Then, 
\begin{equation}\label{eq:finaltransversality}
\mathcal{K}_+\circ \mathtt{g}_- -\mathcal{K}_+=\langle \mathcal{J}(g_+-\mathrm{Id}),\ g_--g_+\rangle+\mathcal O(\varepsilon \delta(\varepsilon)\lVert g_--g_+\rVert)
\end{equation}
where $\mathcal{J}$ denotes the standard complex structure in $\mathbb{R}^2$.
\end{prop}

\begin{proof}
 Let $\mathcal{K}_+$ be the autonomous Hamiltonian obtained in Lemma \ref{lem: Neishtadt's thm} and write 
\[
\mathcal K_+\circ\mathtt g_--\mathcal K_+=(\mathcal{K}_+\circ \mathtt{g}_- -\mathcal{K}_+\circ \mathtt{g}_+) +(\mathcal{K}_+\circ \mathtt{g}_+ -\mathcal{K}_+)
\]
and we expand in Taylor series 
\begin{align*}
\mathcal{K}_+\circ\mathtt{g}_- -\mathcal{K}_+\circ \mathtt{g}_=&\langle \nabla \mathcal{K}_+\circ \mathtt{g}_+, \ \mathtt{g}_--\mathtt{g}_+\rangle\\
&+ \int_{0}^1 (1-t) \langle D^{2}\mathcal{K}_+\circ \mathtt{g}_t\  (\mathtt{g}_--\mathtt{g}_+), \ \mathtt{g}_--\mathtt{g}_+\rangle \mathrm{d}t,
\end{align*}
where we have written $\mathtt{g}_t=t\mathtt{g}_++(1-t)(\mathtt{g}_--\mathtt{g}_+)$. On one hand, denoting by $\mathcal{J}$ the usual complex structure in $\mathbb{R}^2$ and using inequalities  \eqref{eq:vfieldapprox} and \eqref{eq:Neishtadtchangevar}, we have that 
\[
\begin{split}
 \nabla \mathcal{K}_+= &  \nabla K_+ +\nabla (\mathcal{K}_+-K_+)= \mathcal{J}( g_+-\mathrm{Id}) +\left( \mathcal{O}\left(\varepsilon\delta(\varepsilon)\right),\ \mathcal{O}\left( \varepsilon^2\right) \right)^\top.
\end{split}
\]
On the other hand, since $\psi$ is a $\mathcal{O}(\delta(\varepsilon))$-close to identity real analytic transformation defined in a complex neighborhood of size $\rho/8\sim 1$, one easily checks that $\tilde{\psi}\equiv\psi^{-1}-\mathrm{Id} =\mathcal{O}(\delta(\varepsilon))$ and
\[
\begin{split}
\mathtt{g}_--\mathtt{g}_+=&\psi^{-1}\circ g_-\circ\psi- \psi^{-1}\circ g_+\circ\psi=\left(\psi^{-1}\circ g_--\psi^{-1}g_+\right)\circ\psi =\left( g_--g_+ +\tilde{\psi}\circ g_- -\tilde{\psi}\circ g_+ \right)\circ\psi\\
=& \left(( g_- -g_+)+ \left(\int_0^1 D \tilde{\psi} \left(g_- +s(g_+-g_-)\right)\mathrm{d}s\right)  ( g_- -g_+) \right)\circ \psi\\
=&  g_- -g_+  + \mathcal{O}\left(\delta(\varepsilon)\lVert g_- -g_+\rVert \right).
\end{split}
\]
Therefore, since $\mathtt{g}_+=\mathrm{Id}+\mathcal{O}(\varepsilon)$ (see \eqref{eq:firstcondmapmoeckel} and \eqref{eq:Neishtadtchangevar})
\begin{align*}
\langle \nabla \mathcal{K}_+\circ \mathtt{g}_+, \mathtt{g}_- -\mathtt{g}_+ \rangle=&  \big\langle \mathcal{J}( g_+-\mathrm{Id}) +\left(\mathcal{O}(\varepsilon\delta(\varepsilon)), \mathcal{O}(\varepsilon^2)\right)^\top,\ g_- -g_+ + \mathcal{O}\left(\delta(\varepsilon) \lVert g_- -g_+\rVert \right)    \big\rangle\\
=&  \big\langle \mathcal{J}( g_+-\mathrm{Id}) ,\ g_- -g_+    \big\rangle+ \mathcal O(\varepsilon \delta(\varepsilon)\lVert g_--g_+\rVert)
\end{align*}
and the conclusion follows. 
\end{proof}

The proof of  Theorem \ref{thm:scattmapmain} is complete since  \eqref{eq:finaltransversality} implies a bound from below for the maximal variation of $\mathcal{K}_+$ along any orbit of the map $\mathtt{g}_-$. Indeed, for a map satisfying \eqref{eq:secondcondmapmoeckel} and \eqref{eq:suffcondmoeckel2},  inequality \eqref{eq:finaltransversality} implies that, for any $(\alpha,G)\in\mathbb{I}$,
\begin{equation}\label{eq:jumporbitg-}
|\mathcal{K}_+\circ \mathtt g_-(\alpha,G)-\mathcal{K}_+(\alpha,G)|\geq \eta(\varepsilon)>0.
\end{equation}
Then, if \eqref{eq:suffcondmoeckel} is satisfied, the estimate in \eqref{eq:jumporbitg-} and Proposition \eqref{prop:KAM} imply that the variation $|\mathcal{K}_+\circ \mathtt g_-(\alpha,G)-\mathcal{K}_+(\alpha,G)|$ for $(\alpha,G)\in\mathbb I$ is larger than the maximum variation of $\mathcal{K}_+$ along a essential invariant curve of the map $\mathtt{g}_+$. Thus, the maps $\mathtt g_+$ and $\mathtt g_-$ cannot have essential invariant curves in common.


\appendix

\section{The 2-body problem}\label{sec:perturbregime}

In this section we recall a few classical facts about the 2-body problem which are used throughout the text. The 2-body problem (2BP) in polar coordinates is the  Hamiltonian system associated to
\begin{equation}\label{eq:Ham2bp}
H_{\mathrm{2BP}}(r,y,G)= \frac{y^2}{2}+ \frac{G^2}{2r^2}-\frac{1}{r}
\end{equation}
on the phase space $(r,\alpha,y,G)\in \mathbb{R}_+\times\mathbb{T}\times \mathbb{R}^2$. Since the Hamiltonian $H_{\mathrm{2BP}}$ does not depend on the angle $\alpha$, the angular momentum $G$ is a first integral for the 2BP. Moreover, it is functionally independent and commutes with the energy $H_{\mathrm{2BP}}$, what makes the 2BP integrable. The dynamics of the 2BP is completely understood: positive energy levels correspond to hyperbolic motions, negative energy levels to elliptic motions and the zero energy level corresponds to parabolic motions.

\subsection{The parabolic homoclinic manifold of the 2BP}

Of special interest for us are the parabolic motions. Denote by $\mathcal{P}_{\infty}^{\mathrm{2BP}}=\{(\infty,\alpha,0,G)\in \mathbb{R}_+\times\mathbb{T}\times \mathbb{R}^2 \}=\mathcal{P}_\infty\cap \{t=E=0\}$ the parabolic infinity in the reduced phase space (see the extended phase space in polar coordinates in Section \ref{sec:Hamiltonianandeqsofmotion}), which is a 2 dimensional TNHIC. Then,  the set of points leading to parabolic motions, that is, the set $\{H_{\mathrm{2BP}}=0\}$, is a 3 dimensional submanifold $W^{\mathrm{h}}_{\mathrm{2BP}}$  homoclinic to $\mathcal{P}_{\infty}^{\mathrm{2BP}}$. Let $\phi^\tau_{H_{\mathrm{2BP}}}$ be the flow associated to the Hamiltonian \eqref{eq:Ham2bp}, then (note that for the $r$ component $\pi_r \circ \phi^\tau_{H_{\mathrm{2BP}}}(x)\to \infty$ as $\tau\to\pm\infty$)
\begin{equation}\label{eq:homman2BP}
W^{\mathrm{h}}_{\mathrm{2BP}}=\{x\in  \mathbb{R}_+\times\mathbb{T}\times \mathbb{R}^2\colon \exists z\in\mathcal{P}_{\infty}^{\mathrm{2BP}}\quad\text{for which}\quad \lim_{\tau\pm\infty} |\phi^\tau_{H_{\mathrm{2BP}}}(x)-\phi^\tau_{H_{\mathrm{2BP}}}(z)|=0\}.
\end{equation}

The following lemma gives a parametrization of the homoclinic manifold $W^{\mathrm{h}}_{\mathrm{2BP}}$. A proof can be found in \cite{MR1284416}.

\begin{lem}\label{lem:uperturbedhomoclinic}
There exist real analytic functions $\tilde r_{\mathrm{h}}(u;G),\tilde\alpha_{\mathrm{h}}(u;G)$ and $\tilde y_{\mathrm{h}}(u;G)$ such that 
\[
W^{\mathrm{h}}_{\mathrm{2BP}}=\{  \Gamma_{\mathrm{2BP}}(u,\beta)=(G^2 \tilde r_{\mathrm{h}}(u;G), \beta+\tilde\alpha_{\mathrm{h}}(u;G), G^{-1} \tilde y_{\mathrm{h}}(u;G), G)\in \mathbb{R}_+\times\mathbb{T}\times \mathbb{R}^2\colon u\in\mathbb{R},\ \beta\in\mathbb{T},\ G\in\mathbb{R}\setminus\{0\} \}
\]
and, if we denote by $X_{\mathrm{2BP}}$  the vector field associated to the Hamiltonian \eqref{eq:Ham2bp},
\[
X_{\mathrm{2BP}}\circ \Gamma_{\mathrm{2BP}}=D \Gamma_{\mathrm{2BP}}\  \Upsilon\qquad\qquad\text{with}\qquad\qquad \Upsilon=(G^{-3},0).
\]
The functions 
\[
r_{\mathrm{h}}(u;G)=\tilde r_{\mathrm{h}}(G^{-3}u;G),\qquad\qquad y_{\mathrm{h}}(u;G)=\tilde y_{\mathrm{h}}(G^{-3}u;G),\qquad\qquad \alpha_{\mathrm{h}}(u;G)=\tilde \alpha_{\mathrm{h}}(G^{-3}u;G)
\]
and  admit a unique analytic extension to  $\mathbb{C}\setminus \{u= is \colon s\in(-\infty,-1/3]\cup[1/3,\infty)\}$  and satisfy the asymptotic behavior
\[
r_{\mathrm{h}}(u)\sim u^{2/3}\qquad\qquad  \exp(i\alpha_{\mathrm{h}}(u))\sim 1 \qquad\qquad y_{\mathrm{h}}(u)\sim u^{-1/3}\qquad\qquad\text{as}\quad u\to \pm \infty
\]
and 
\[
r_{\mathrm{h}}(u)\sim (u\pm i/3)^{1/2}\qquad\quad \exp(i\alpha_{\mathrm{h}}(u))\sim \left(\frac{u\pm i/3}{u\mp i/3} \right)^{1/2} \qquad\quad y_{\mathrm{h}}(u)\sim (u\pm i/3)^{-1/2}\quad\qquad\text{as}\quad u\to \pm i/3.
\]
Moreover, $y_{\mathrm{h}}(u)=0$ if and only if $u=0$ and  $r_{\mathrm{h}}(u)\geq 1/2$ for all $u\in\mathbb{R}$.
\end{lem}

\section{The perturbative potential $V$  and the Melnikov potential $L$}\label{sec:Melnikov}

In this appendix we provide we provide the proofs of Lemma \ref{lem:boundspotential}, which describes the behavior of the perturbative potential $V$ defined in \eqref{eq:dfnperturbativepotential}, Lemma \ref{lem:mainMelnikov} which states the main properties of the Melnikov potential $L$ defined in \eqref{eq:defnMelnikovPot} and Lemma \ref{lem:nondegdiffmelnikov} concerning the reduced Melnikov potentials $\mathcal{L}_\pm$ introduced in \eqref{eq: DefinitionredMelnikovPotential}. We start by recalling the following well known result, a proof of which can be found in \cite{MR1284416}.

\begin{lem}\label{lem: propunperthomoclinic}
Let $r_{\mathrm{h}}(u)$ and $\alpha_{\mathrm{h}}(u)$ the functions defined in Lemma \ref{lem:uperturbedhomoclinic}. Then, under the real analytic change of variables $u=(\tau+\tau^3/3)/2$, and using the same notation  $r_{\mathrm{h}}(\tau)$ and $\alpha_{\mathrm{h}}(\tau)$, 
we have that
\[
r_{\mathrm{h}}(\tau)=\frac{\tau^2+1}{2}\qquad\qquad e^{i\alpha_{\mathrm{h}}(\tau)}=\frac{\tau-i}{\tau+i}.
\]
\end{lem}
\subsection{Proof of Lemma  \ref{lem:boundspotential}}

From the definition of $V(u,\beta,t;G_{0})$ in \eqref{eq:dfnperturbativepotential} and straightforward manipulations we obtain that
\begin{equation}\label{eq:Potentialtauvars}
\begin{split}
U(\tau,\beta,t;G_{0})=&V(u(\tau),\beta,t;G_{0})\\
=& \frac{\mu} {r_{\mathrm{h}}(\tau)\left(1+\frac{2\left(1-\mu\right)\varrho\left(t\right)}{G_{0}^2 r_{\mathrm{h}}(\tau)}e^{i(\beta+\alpha_{\mathrm{h}}(\tau)-f(t))}\right)^{1/2}\left(1+\frac{2\left(1-\mu\right)\varrho\left(t)\right)}{G_{0}^2 r_{\mathrm{h}}(\tau)}e^{-i(\beta+\alpha_{\mathrm{h}}(\tau)-f(t))}\right)^{1/2}}\\
&+ \frac{(1-\mu)} {r_{\mathrm{h}}(\tau)\left(1-\frac{2\mu\varrho\left(t\right)}{G_{0}^2 r_{\mathrm{h}}(\tau)}e^{i(\beta+\alpha_{\mathrm{h}}(\tau)-f(t))}\right)^{1/2}\left(1-\frac{2\mu\varrho\left(t\right)}{G_{0}^2 r_{\mathrm{h}}(\tau)}e^{-i(\beta+\alpha_{\mathrm{h}}(\tau)-f(t))}\right)^{1/2}}- \frac{1}{r_{\mathrm{h}}(\tau)}.\\
\end{split}
\end{equation}
As we need to bound the Fourier coefficients of  $V(u,\beta,t;G_{0})$ for $u\in D_{\kappa}^{\mathrm{u}}$, we will use the transformation in Lemma \ref{lem: propunperthomoclinic} and bound the potential in these variables, where we have the explicit expressions of $r_{\mathrm{h}}$ and $\alpha_{\mathrm{h}}$. 
Important in the sequel is that when $u\in D^{\mathrm{u}}_\kappa $
we know that $|\tau^2+1|\geq \kappa  G_{0}^{-3/2}$. 
We now define the Fourier coefficients of $t\to U(\tau,\beta,t;G_{0})$ as the integral expression 
\begin{equation}\label{eq:FourierU}
U^{[l]}(\tau,\beta;G_{0})=\frac{1}{2\pi}\int_{0}^{2\pi} U(\tau,\beta,t;G_{0}) e^{-ilt}\mathrm{d}t.
\end{equation}
In this proof we will perform several changes of variables in this integral but we will keep the same notation for the functions $\varrho$ and $f$. In order to analyze this integral, we change the integration variable to the eccentric anomaly $\xi$ by means of Kepler equation $t=\xi-\zeta\sin\xi$ so that \eqref{eq:FourierU} reads
\begin{equation}\label{eq:FourierU2}
U^{[l]}(\tau,\beta;G_{0})=\frac{1}{2\pi}\int_{0}^{2\pi} (1-\zeta\cos\xi) U(\tau,\beta,\xi-\zeta\sin\xi;G_{0}) e^{-il(\xi-\zeta\sin\xi)}\mathrm{d}t.
\end{equation}
In this way, we have the explicit formulas
\begin{equation}\label{eq:anomalies}
\varrho(\xi)=1-\zeta\cos\xi\qquad\qquad \varrho(\xi)e^{if(\xi)}=a^2 e^{i\xi}-\zeta+\frac{\zeta^2}{4a^2} e^{-i\xi},
\end{equation}
where $a=(\sqrt{1+\zeta}+\sqrt{1-\zeta})/2$. Changing the integration contour in \eqref{eq:FourierU2} to the line $\{\xi\in \mathbb{C}/2\pi\mathbb{Z}\colon \xi=\alpha_{\mathrm{h}}(\tau)+s,\ s\in[0,2\pi] \}$ we obtain that,
\[
U^{[l]}(\tau,\beta;G_{0})=\frac{e^{-il\alpha_{\mathrm{h}}(\tau)}}{2\pi}\int_{0}^{2\pi} (1-\zeta\cos(\alpha_{\mathrm{h}}(\tau)+s))U(\tau,\beta, \alpha_{\mathrm{h}}(\tau)+s-\zeta\sin(\alpha_{\mathrm{h}}(\tau)+s);G_{0})e^{-il(s-\zeta\sin(\alpha_{\mathrm{h}}(\tau)+s))} \mathrm{d}s ,
\]
and
\begin{equation}\label{eq:anomaliesens}
\varrho(s)=(1-\zeta\cos(\alpha_{\mathrm{h}}(\tau)+s)) \qquad \varrho(s)e^{if(s)}=
e^{i\alpha_{\mathrm{h}}(\tau)}\left(
a^2 e^{is}-\zeta
e^{-i\alpha_{\mathrm{h}}(\tau)}
+\frac{\zeta^2}{4a^2} e^{-i(2\alpha_{\mathrm{h}}(\tau)+s))}\right)
\end{equation}
Now, the main observation is that, using the assumption $\zeta\lesssim G_{0}^{-2}$,
for fixed $\kappa,\sigma>0$ and all 
$(\tau,\beta)\in \{|\tau^2+1|\geq \kappa  G_{0}^{-3/2}\}\times \mathbb{T}_\sigma$ 
one can easily see that
\[
\left|
\zeta e^{\pm i\alpha_{\mathrm{h}}(\tau)}\right| \lesssim
G_{0}^{-\frac12}
\]
and, therefore,
\[
\left| \varrho\left(s\right)\right| \lesssim 1 \qquad
\left| \varrho\left(s\right)e^{\pm i(\beta+\alpha_{\mathrm{h}}(\tau)-f(t))}\right|\lesssim 1 .
\]
Using these inequalities, as well as the fact that
\[
\left| \frac{1}{G_{0}^2 r_{\mathrm{h}}(\tau)}\right|\lesssim G_{0}^{-\frac12},
\]
we obtain,
\[
\left| \frac{2\left(1-\mu\right)\varrho\left(t\right)}{G_{0}^2 r_{\mathrm{h}}(\tau)}e^{\pm i(\beta+\alpha_{\mathrm{h}}(\tau)-f(t))}\right|\lesssim \frac{1}{G_{0}^{2} r_{\mathrm{h}}(\tau)} \lesssim G_{0}^{-\frac12}
\qquad\qquad \left| \frac{2\mu\varrho\left(t\right)}{G_{0}^2 r_{\mathrm{h}}(\tau)}e^{\pm i(\beta+\alpha_{\mathrm{h}}(\tau)-f(t))}\right|\lesssim \frac{1}{G_{0}^{2} r_{\mathrm{h}}(\tau)}\lesssim G_{0}^{-\frac12} .
\]
This justifies that we can use the Taylor formula
\[
(1+x)^{-\frac12}=1-\frac12 x+\mathcal{O}(x^2),
\]
to bound the Fourier coefficients of the potential.  Using the cancellations of the order $0$ and $1$ terms we get, for a certain $\rho_0,\tilde{\sigma}>0$ small enough but independent of $ G_{0}$, and for $(\tau,\beta)\in \{|\tau^2+1|\geq \kappa  G_{0}^{-3/2}\}\times \mathbb{T}_{\rho_0}$
\[
|U^{[l]}(\tau,\beta)|\lesssim  G_{0}^{-4} |r_{\mathrm{h}}(\tau)|^{-3} |e^{-il\alpha_{\mathrm{h}}(\tau)}| e^{-|l|\tilde{\sigma}}.
\]
Equivalenty, for $(u,\beta) \in D_{\kappa}^{\mathrm{u}}\times \mathbb{T}_{\rho_0} $
\begin{equation}\label{eq:boundsV}
|V^{[l]}(u,\beta)|\lesssim  G_{0}^{-4} |r_{\mathrm{h}}(u)|^{-3} |e^{-il\alpha_{\mathrm{h}}(u)}| e^{-|l|\tilde{\sigma}},
\end{equation}
that taking into account Lemma \ref{lem:uperturbedhomoclinic} gives the desired bound for the norm of $V^{[l]}$ and $V$ and completes the proof of the first estimate in Lemma \ref{lem:boundspotential}. The estimate for the difference $V-V_{\mathrm{circ}}$ is obtained from the fact that $V$ depends analytically on $\zeta$ and a straightforward application of Schwarz's lemma.

\subsection{Proof of Lemmas   \ref{lem:nondegdiffmelnikov} and \ref{lem:mainMelnikov}}
The estimates \eqref{eq:boundsV} are enough to bound the associated Fourier coefficients $L^{[l]}(\beta;G_{0})$ of the Melnikov potential $\tilde{L}(u,\beta,t;G_{0})$ defined in \eqref{eq:defnMelnikovPot}. In fact
\begin{equation}
\tilde{L}(u,\beta,t;G_{0},\zeta)=\sum e^{i l (t-G_{0}^3 u)}L^{[l]}(\beta;G_{0},\zeta), \quad L^{[l]}(\beta;G_{0})=\int_{-\infty}
^{\infty} V^{[l]}(s,\beta;G_{0},\zeta)e^{il G_{0}^3 s} ds,
\end{equation}
so we can write
\[
\tilde{L}(u,\beta,t;G^{\mathrm{u}},G^{\mathrm{s}},\zeta)=\sum_{l\in\mathbb{N}} \mathcal{L}_{l}(t-G_{0}^3u,\beta;G_{0},\zeta)
\]
where
\[
\mathcal{L}_{l}(t-G_{0}^3u,\beta;G_{0},\zeta)=e^{i l (t-G_{0}^3 u)}L^{[l]}(\beta;G_{0},\zeta)+e^{-i l (t-G_{0}^3 u)}L^{[-l]}(\beta;G_{0},\zeta).
\]
Then, for $l \ge 1$, it is enough to change the path of integration to $\Im u=\frac13 - G_{0}^3$ to bound $|L^{[l]}|$, use the bounds   
\eqref{eq:boundsV}, use that
\[
\left|
e^{\pm i\alpha_{\mathrm{h}}(u)}\right| \lesssim
 G_{0}^{\frac32}
\]
and the fact that $L^{[-l]}=\overline{L^{[l]}}$, to obtain, writing $\sigma=t-G_{0}^3u$,
\[
|\mathcal{L}_l(\sigma,\beta;G_{0})| \le  G_{0}^{\frac{3l}{2}+\frac{3}{2}}\exp(-l \operatorname{Re}(G_{0}^3)/3)
\]
and, therefore, for the sum,
\[
|\mathcal{L}_{\geq 2}(\sigma,\beta;G_{0})| \le  G_{0}^{\frac92}\exp (-2 \operatorname{Re}(G_{0}^3)/3).
\]
The coefficients $L^{[0]}$ and $L^{[1]}$ can be computed expanding the potential $U$ up to order four in powers of $1/r_{\mathrm{h}}(\tau)$ and bounding the remainder in an analogous way. We do not do the computations here because they can be found in
Lemmas 31 and 36   in \cite{MR3927089}.
Define  the  coefficients
\[
c^{k,n}_{l}(\mu)=\frac{1}{2\pi} ((1-\mu)^{k-1}-(-\mu)^{k-1})\int_{0}^{2\pi} \varrho^{k}(t)e^{-inf(t))}e^{-ilt}\mathrm{d}t.
\]
Then, one has
\[
\begin{split}
\mathcal{L}_0(\beta; G_{0})= &\mu(1-\mu) \left( c^{2,0}_0(0)\frac{\pi}{2 G_{0}^3} +(1-2\mu) c^{3,1}_0(0)  \frac{3\pi}{4G_{0}^5} \cos\beta+ \mathcal{O}(\zeta^2 G_{0}^{-7}) \right),\\
\end{split}
\]
and
\[
\begin{split}
\mathcal{L}_1(\sigma,\beta; G_{0},\zeta)=\mu(1-\mu)& \left(2 L_{1,1}(G_{0},\zeta)\cos(\sigma-\beta)+2L_{1,2}(G_{0},\zeta)\cos(\sigma-2\beta) + \mathcal{O}(\zeta  G_{0}^{-3/2},  |c^{3,3}_1 G_{0}^4|) \right),
\end{split}
\]
with
\[
\begin{split}
L_{1,1}(G_{0},\zeta)=& (1-2\mu) \left( c^{3,1}_1(0) \sqrt{\frac{\pi}{8G_{0}}}+ \mathcal{O}( G_{0}^{-2}) \right)\exp(-G_{0}^3/3)\\
L_{1,2}(G_{0},\zeta)= & \left(c_1^{2,2}(0 )\sqrt{\frac{\pi G_{0}^{3}}{2}}+\mathcal{O}(\zeta) \right)\exp(-G_{0}^3/3).
\end{split}
\]
The proof of Lemma \ref{lem:mainMelnikov}  is now completed by making use of  Lemma 28 in \cite{MR3927089}  where the coefficients  $c^{k,n}_l(0)$ are computed.   An analogous computation is done in \cite{https://doi.org/10.48550/arxiv.2207.14351}.
\\
Finally, the proof of Lemma \ref{lem:nondegdiffmelnikov} is straightforward after noticing that 
\[
\mathcal{L}_+(\beta;G_{0})= \sum_{l\in\mathbb Z} \tilde {\mathcal L}_l(\beta,G)=\sum_{l\in \mathbb{Z}}e ^{il\sigma(\beta;G_0)} L^{[l]}(\beta;G_{0})
\]
and
\[
\mathcal{L}_-(\beta;G_{0})= \sum_{l\in\mathbb Z}(-1)^l\tilde{\mathcal L}_l(\beta;G_0)=\sum_{l\in \mathbb{Z}} (-1)^l e ^{il\sigma(\beta;G_0)} L^{[l]}(\beta;G_{0})
\]
where $\sigma(\beta;G_0)$ is one critical point of the map $\sigma \to \partial_\sigma \mathcal L$ and we have used that the other critical point of this function is given by $\sigma(\beta;G_0)+\pi$. Indeed, notice that the first order term in the bracket $\{\mathcal L_+,\mathcal L_-\}$ is given by $2\{\mathcal {\tilde{L}}_0,\mathcal {\tilde{L}}_1\}$.

\bibliographystyle{alpha}
\bibliography{biblioMelnikov}

\end{document}